\documentclass{article}
\usepackage[utf8]{inputenc}
\usepackage[T1]{fontenc}
\usepackage{amsmath}
\usepackage{amssymb}
\usepackage{mathrsfs}
\usepackage{amsthm}
\usepackage{stmaryrd}
\usepackage{geometry}
\usepackage{mathtools}
\usepackage{biblatex}

\usepackage{hyperref}
\usepackage[nameinlink,noabbrev]{cleveref}
\usepackage{makeidx}

\hypersetup{
    colorlinks,
    citecolor=black,
    filecolor=black,
    linkcolor=black,
    urlcolor=black
}

\def\W{\mathcal W}

\def\R{\mathbb R}

\def\N{\mathbb N}

\def\T{\mathbb T}

\def\E{\mathbb E}
\def\L{\mathrm L}
\def\H{\mathrm H}

\def\eps{\varepsilon}

\def \p {\partial}
\DeclareMathOperator{\supp}{supp}
\DeclareMathOperator{\re}{Re}
\DeclareMathOperator{\vspan}{Span}
\def\d{\mathrm{d}}

\newtheorem{theo}{Theorem}[section]
\crefname{theo}{Theorem}{Theorems}
\Crefname{theo}{Theorem}{Theorems}

\newtheorem{prop}[theo]{Proposition}
\crefname{prop}{Proposition}{Propositions}
\Crefname{prop}{Proposition}{Propositions}

\newtheorem{lemme}[theo]{Lemma}
\crefname{lemme}{Lemma}{Lemmas}
\Crefname{lemme}{Lemma}{Lemmas}

\crefname{defi}{Definition}{Definitions}
\Crefname{defi}{Definition}{Definitions}

\newtheorem{cor}[theo]{Corollary}
\crefname{cor}{Corollary}{Corollaries}
\Crefname{cor}{Corollary}{Corollaries}

\crefname{section}{Section}{Sections}
\Crefname{section}{Section}{Sections}

\crefname{appendix}{Appendix}{Appendices}
\Crefname{appendix}{Appendix}{Appendices}

\crefname{equation}{}{}
\Crefname{equation}{}{}

\theoremstyle{definition}
\newtheorem{rem}[theo]{Remark}
\crefname{rem}{Remark}{Remarks}
\Crefname{rem}{Remark}{Remarks}

\newcommand{\e}{\mathrm{e}} 
\newcommand{\complexI}{\mathrm{i}} 

\DeclarePairedDelimiterX{\wick}[1]{:}{:}{#1} 

\numberwithin{equation}{section}

\bibliography{biblio.bib}

\title{Local wellposedness of the 2d Anderson-Gross-Pitaevskii equation}

\author{S. Mackowiak\thanks{Univ Brest,  CNRS UMR 6205,  Laboratoire de Mathématiques de Bretagne Atlantique,  F-29200,  Brest,  France} \thanks{Université de Lorraine,  CNRS,  IECL,  F-54000 Nancy,  France}\\
\href{mailto:samael.mackowiak@univ-lorraine.fr}{samael.mackowiak@univ-lorraine.fr}}

\date{}

\begin{document}

\maketitle

\begin{abstract}
    In this paper, the local wellposedness of a general Gross-Pitaevskii equation with rough potential is proven in dimension 2. The class of rough potentials we are considering is large enough to contain the spatial white noise and thus a renormalization procedure may be needed. We first construct the associated Schrödinger operator from its quadratic form. Then, the regularity of elements of its domain is explored. This allows to use a paracontrolled approach in order to obtain Strichartz estimates, which are used to prove the local wellposedness by a contraction argument.
\end{abstract}

\section{Introduction}

In this paper, we are interested in solving the renormalized Gross-Pitaevskii equation with a real spatial white noise potential $\xi$ in two dimensions of space
\begin{equation}{~}
    \begin{cases}\mathrm{i}\p_t u+ H u + \xi u+\lambda \left|u\right|^{2\gamma}u=0\\ u(0)=u_0,\end{cases} \label{Eq:2D}
\end{equation} 
where $u=u(t,x)$ is a complex-valued space-time random field defined on $I\times\R^2_x$ for some time interval $I$, $H=\Delta-|x|^2$ is the Hermite operator, $\lambda\in\R$, and $\gamma>0$. The study of such equation is motivated by the modelling of Bose-Einstein condensation. In fact, the deterministic equation, which is a non-linear Schrödinger equation with confining potential, is an extensively studied dispersive equation known as the Gross-Pitaevskii equation. It appears in the study of Bose-Einstein condensates as a mean-field limit to the coupled linear many-body Schrödinger equation with confining potential (see \cite{BEC_Pitaevskii}). The Gross-Pitaevskii equation with random potential can be seen as a model for Bose-Einstein condensates in presence of inhomogeneities (see for example \cite{PhysRevLett.104.193901,PhysRevLett.109.243902} and references therein). Hence, the white noise potential can be seen as a toy model of random potential with null correlation length.

In \cite{Mackowiak_2025} the author hinted  this equation may not admit solutions, due to the low regularity of the spatial white noise, even for smooth and localised initial data. Instead, the author solved the equation 
\begin{equation}{~}
    \begin{cases}\mathrm{i}\p_t u+ \wick{H+\xi}u+\lambda \left|u\right|^{2}u=0\\ u(0)=u_0,\end{cases} \label{Eq:2D-renorm-cubic}
\end{equation} 
where $\wick{H+\xi}$ is the Wick renormalization of the
Hermite operator with spatial white noise potential. The renormalized operator is defined through an exponential transform inspired by the one introduced in \cite{Haire_Labbe} to study the continuous parabolic Anderson model, and later used in \cite{debussche_weber_T2,Tzvetkov_2023,tzvetkov2020dimensional} to solve the non-linear Anderson-Schrödinger equation on the torus $\T^2$ and in \cite{debussche2017solution,debussche2023global} on the full space $\R^2$. Let $Y=(-H)^{-1}\xi$, it is almost surely of regularity $1-$, so one can define $\rho=\e^Y$ and search for a formula for $(H+\xi)u$ when $u$ is of the form $u=\rho v$. A formal computation gives
$$(H+\xi)u=\rho\left(Hv+2\nabla Y\cdot \nabla v + xY\cdot xv + |\nabla Y|^2v\right).$$
Unfortunately, $\nabla Y$ is of regularity $0-$, thus the product $|\nabla Y|^2$ barely fails to exist. The author proved in \cite{Mackowiak_2025} that the wick product $\wick{|\nabla Y|^2}$ can be defined as the almost sure limit, in suitable space of distributions of regularity $0-$, of
$$|\nabla Y_N|^2-\E[|\nabla Y_N|^2],$$
where $Y_N$ is a suitable regularisation of $Y_N$ with only low modes. Thus, the renormalized operator can be formally defined as
$$\wick{H+\xi}u=\rho\left(Hv+2\nabla Y\cdot \nabla v + xY\cdot xv + \wick{|\nabla Y|^2}v\right).$$
This shows that knowing a trajectory of $\xi$ is not sufficient to define the operator $\wick{H+\xi}$. It is a usual fact in the study of Anderson operators that one needs more data in order to define the operator (see for example \cite{mouzard2021weyl,bailleul2022analysis,allez2015continuous,hsu2024construction}). Following this idea, we construct a formal renormalization of $H+\xi$ given an enhanced noise $\Xi=(Y,Z)$ with $Y=(-H)^{-1}\xi$ and $Z$ playing the role of $|\nabla Y|^2$, as 
$$\wick{H+\xi}u=\rho\left(Hv+2\nabla Y\cdot \nabla v + xY\cdot xv + Zv\right).$$

This formula does not make sense directly, but as usual, the bilinear form of the operator can be defined as
$$\langle-\wick{H+\xi}u_1,u_2\rangle=\re \int_{\R^2}\left(\nabla v_1\cdot \overline{\nabla v_2} + |x|^2(1-Y)v_1\overline{v_2}\right)\rho^2\d x - \langle Z, \overline{v_1}v_2\rho^2\rangle,$$
where $\rho=\e^Y$, $u_i=\rho v_i$ and $\langle\cdot,\cdot\rangle$ stand for a duality bracket between suitable spaces. We then show that this bilinear form is associated to a self-adjoint lower semi-bounded operator of compact resolvent on $\L^2(\R^2)$ which we take as a definition of $-\wick{H+\xi}$. Up to a $\Xi$-dependent shift $\delta_\Xi$, one can define a positive self-adjoint operator $-A=-\wick{H+\xi}+\delta_\Xi$. This makes $-A$ a bit easier to study than $\wick{H+\xi}$. The construction of $-A$ and its elementary properties are the object of \cref{Sec:ConstructionOperator}. A similar construction for $\wick{\Delta+\xi}$ on $\T^2$ and $\T^3$ has been done in \cite{mouzard2023simple}. 

In this paper, we follow a purely deterministic and analytic approach of "formal renormalization" which allows us to treat in the same way cases where a renormalization is needed and cases where $|\nabla Y|^2$ exists as a $\L^q$ function (for $q>4$), as for the radial white noise or any deterministic potential in a Sobolev space $\mathrm{H}^{-1,2q}$.\\

Unfortunately, the domain of $-A$ is not stable under multiplication, hence the standard fixed point fails to apply directly to
\begin{equation}{~}
    \begin{cases}\mathrm{i}\p_t u+ Au+\lambda \left|u\right|^{2\gamma}u=0\\ u(0)=u_0.\end{cases} \label{Eq:2DRenormA}
\end{equation}
In this situation, it is well-known that one may take advantage of dispersive properties of the associated Schrödinger propagator $(\e^{itA})_{t\in\R}$ in order to obtain local solutions. Given an admissible pair $(p,r)$, that is a pair $(p,r)\in[2,+\infty)^2$ satisfying $\frac{1}{p}+\frac{1}{r}=\frac{1}{2}$, then one expects $(\e^{itA})_{t\in\R}$ to verify
\begin{equation}
    \left|\e^{itA}u\right|_{\L^p_{(-T,T)}\L^r_x}\leqslant C |u|_{\W^{\sigma,2}}\label{Eq:StrichartzTypeInequality}
\end{equation}
for some $T>0$ and $\sigma\geqslant0$, where $\W^{\sigma,2}$ is a Sobolev space associated to $-H$. This idea goes back to Strichartz, when studying in \cite{Strichartz1977} the local smoothing property of wave equations. Such inequalities have been obtained for the free Schrödinger propagator $(\e^{it\Delta})_{t\in\R}$ on $\L^2(\R^2)$, using the dispersive estimate
$$|\e^{it\Delta}u|_{\L^\infty_x}\lesssim |t|^{-1}|u|_{\L^1_x},$$
and in that case $T>0$ is arbitrary and $\sigma=0$, thus $\W^{\sigma,2}$ is just $\L^2(\R^2)$. In the presence of a smooth confining potential $V$, there are two different cases depending on the growth at infinity of $V$. On the one hand, if $V$ is subquadratic at infinity, that is $\p^\alpha V$ is bounded when $|\alpha|\geqslant 2$, it was shown in \cite{FujiwaraSubquadratic} that $(\e^{it(\Delta-V)})_{t\in\R}$ verifies the same kind of dispersive estimate on a short time. This implies that it verifies an estimate of the type \cref{Eq:StrichartzTypeInequality} with $T$ small enough and $\sigma=0$. On the other hand, if $V$ is superquatratic, then it is proven in \cite{YajimaSuperquadratic} that $(\e^{it(\Delta-V)})_{t\in\R}$ does not verify a dispersive estimate. Nevertheless, they obtain Strichartz inequalities with a loss of derivative on short time, that is \cref{Eq:StrichartzTypeInequality} with $T$ small enough and some $\sigma>0$. On a compact manifold, Strichartz estimates with losses for $(\e^{it\Delta})_{t\in\R}$ have been obtained in \cite{BurqStrichartzManifold}, where they were proven to be stable under $\mathrm{H}^1$ perturbations. The same method was then used in \cite{mouzard2022strichartz,zachhuber2020strichartz} to obtain Strichartz inequalities with losses for $(\e^{it(\Delta+\xi)})_{t\in\R}$ on 2d compact manifolds. In this paper, we follow this idea of lower order perturbation of Strichartz inequalities for $(\e^{it H})_{t\in\R}$ to obtain Strichartz inequalities for $(\e^{it A})_{t\in\R}$. However, $A$ cannot be seen directly as a perturbation of $H$. First remark that, formally, 
$$\rho^{-1}A(\rho v) = Hv+2\nabla Y\cdot \nabla v + xY\cdot xv + Zv - \delta_\Xi v.$$
Thus, $\Tilde{A}=\rho^{-1}A\rho$ is a better candidate to be a lower order perturbation of $H$. We prove in \cref{Sec:ConstructionOperator} that this operator is well-defined on $\mathrm{D}(-\Tilde{A})=\rho^{-1}\mathrm{D}(-A)\subset\W^{2-,2}$. Then, a paracontrolled approach of $\Tilde{A}$ allows to get rid of the worst regularity terms in the expression of $\Tilde{A}$. 
The idea of paracontrolled calculus to study singular stochastic equations was introduced in \cite{GubinelliParacontrolled} and was then extensively used to study the continuous Anderson operator (see for instance \cite{allez2015continuous,bailleul2022analysis,mouzard2021weyl,mouzard2022strichartz}). The idea is to look for elements $v$ in the domain of the form $v=v^\sharp+P_{\nabla v}X_1+P_{v} X_2$ with $v^{\sharp}\in D(-H)=\W^{2,2}$, $P_f g$ stands for the paraproduct of $g$ by $f$ (see \cref{Sub:paracontrolled} for a precise definition)  and $X_1,X_2$ are chosen in order to cancel the worst regularity terms in $2\nabla Y \cdot \nabla v$ and $Zv$ respectively. Using this ansatz, one obtains a formula of the type
$$\Tilde{A}v = Hv^\sharp + \mathcal{R}(v)$$
with $\mathcal{R}$ a remainder expected to be of lower order. The paracontrolled approach is detailed in \cref{Sub:StrichartzProof}. Then, following the idea of \cite{BurqStrichartzManifold,mouzard2022strichartz}, we prove the following Strichartz estimate.
\begin{theo}{(Strichartz estimates for $\wick{H+\xi}$)}\label{Th:Strichart(H+xi)}
    Let $\kappa\in\left(0,\frac{1}{2}\right)$, $\xi\in\W^{-1-\kappa,\infty}$ and $\Xi$ verifying \cref{Def:EnhancedNoise,Eq:Condqskappa}. Define $\wick{H+\xi}$ as the unbounded self-adjoint operator associated to the quadratic form given by \cref{Eq:QuadraticFormH+xi}. Let $(p,r)\in[2,+\infty]^2$ satisfying $\frac{1}{p}+\frac{1}{r}=\frac{1}{2}$ and $r<+\infty$. Let $\alpha\in[0,1-2\kappa)$ and $\eps>0$.Then, there exists a time $T>0$ and a constant $C>0$ such that
    $$\forall u\in\mathcal{D}^{\alpha+\frac{1}{p}+\kappa+\eps,2},\; \left|\e^{it\wick{H+\xi}}u\right|_{\L^p_{(-T,T)}\mathcal{D}^{\alpha,r}}\leqslant C |u|_{\mathcal{D}^{\alpha+\frac{1}{p}+\kappa+\eps,2}},$$
    where $\mathcal{D}^{\alpha,r}$ and $\mathcal{D}^{\alpha+\frac{1}{q}+\kappa+\eps,2}$ are defined by \cref{Def:Dsq}.
\end{theo}

Once Strichartz estimates are obtained, it becomes easy to solve \cref{Eq:2DRenormA} for low regularity initial data. As expected for a Schrödinger equation, since the noise is real-valued, solutions preserve their $\L^2$-norm (see \cref{Lem:ConservationMass}), and for $u_0$ in the form domain of $-A$ (that is $D(\sqrt{-A})$), the energy
\begin{equation}
    \mathcal{E}(u)=\frac{1}{2}\langle-Au,u\rangle-\frac{\lambda}{2\gamma+2}\int_{\R^2}|u|^{2\gamma+2}\d x\label{Eq:EnergieDef}
\end{equation}
is conserved (see \cref{Cor:GWPD(-A1/2)}). Let us emphasise an important novelty of the present paper compared to \cite{Mackowiak_2025}, Strichartz inequalities allow us to solve \cref{Eq:2DRenormA} with a noise-independent initial data. Namely, we prove the following result.

\begin{theo}{(Local wellposedness)}\label{Th:LocalWP}
    Let $\kappa$ and $\wick{H+\xi}$ as in \cref{Th:Strichart(H+xi)}. Let $\lambda\in\R$, $\gamma\in\N^*$ such that $2\gamma<\frac{1}{2\kappa}$, $p>\max(2\gamma,\frac{1}{1+\kappa})$ and $\alpha\in(1-\frac{1}{p}+\kappa,1-\kappa)$. Then, for any $u_0\in\W^{\alpha,2}$, there exists a unique maximal solution $(I,u)$ of 
    \begin{equation}\label{Eq:AGP2d}
        \begin{cases}\mathrm{i}\p_t u+\wick{H+\xi} u+\lambda \left|u\right|^{2\gamma}u=0,\; t>0,\; x\in\R^2\\ u(0)=u_0\end{cases} 
    \end{equation} 
    in $\mathcal{C}(I,\W^{\alpha,2})\cap\L^p_{loc}(I,\L^\infty(\R^2))$. Moreover, $I=[0,T_+)$ with maximal time of existence $T_+$ lower-bounded as follow, there exists $C_1,C_2,c>0$ independent of $u_0$, such that
    \begin{equation}
        \forall u_0\in\W^{\alpha,2},\; T=\frac{C_1}{|u_0|_{\W^{\alpha,2}}^c} \Rightarrow \left(T < T_+ \text{ and } |u|_{Y_T}\leqslant C_2 |u_0|_{\W^{\alpha,2}}\right),\label{Eq:BoundExistTimeLWP}
    \end{equation}
    where $Y_T = \L^\infty((0,T),\W^{\alpha,2})\cap\L^p((0,T),\L^\infty(\R^2))$,
    and we have the finite time blow-up alternative,
    $$T_+=+\infty \text{ or } \lim_{t\to T_+}|u(t)|_{\W^{\alpha,2}}=+\infty.$$
    Finally, for any sequence $(u^n_0)\subset\W^{\alpha,2}$ which converges to $u_0$ in $\W^{\alpha,2}$ and any $T\in(0,T_+)$, then for $n$ large enough, the unique solution $u^n$ starting from $u_0^n$ is defined on $[0,T]$ and converges to $u$ in $Y_T$.
\end{theo}

Solving \cref{Eq:AGP2d} with deterministic initial data is only possible because Strichartz inequalities allow us to obtain low regularity solutions. In fact, we shall see in \cref{Prop:CaracDs2} that $$\forall |s|\leqslant 1-\kappa,\; \mathrm{D}\left((-A)^{s/2}\right)=\H^s\cap\L^2(\R^2,\langle x\rangle^s\d x)$$
and thus does not depend on the noise. This has to be opposed to compactness methods \cite{debussche_weber_T2,debussche2017solution,tzvetkov2020dimensional,Tzvetkov_2023,debussche2023global,Mackowiak_2025} for which the initial data is assumed to depend on the noise and to be regular enough (in an appropriate sens). Unfortunately, the loss of derivative in \cref{Th:Strichart(H+xi)} and the absence of $\L^2_t\L^\infty_x$ Strichartz inequalities only allow to solve \cref{Eq:2DRenormA} for initial data of regularity $\alpha>\max(1-\frac{1}{2\gamma},0)$. We then prove in \cref{Cor:GWPD(-A1/2)} what seems to be the first result of regularity propagation for Anderson-NLS type equations in unbounded domains.  Using conservation laws, we conclude to global existence of energy solutions in the defocusing case $\lambda\leqslant 0$, in the focusing sub-critical case $\lambda>0$ and $\gamma<1$ and global existence for small initial data in the focusing critical case $\lambda>0$ and $\gamma=1$.

\subsection*{Notations and conventions}

\begin{itemize}
    \item $\Delta = \p_{x_1}^2+\p_{x_2}^2$ is the Laplacian
    \item For $x\in\mathbb{R}^2$, $|x|^2=x_1^2+x_2^2$ and $\langle x\rangle^2 = 1+|x|^2$
    \item $\N=\{0,1,2,\dotsc\}$ and $\N^*=\{1,2,\dotsc\}$
    \item For $\alpha\in\N^2$, $|\alpha|=\alpha_1+\alpha_2$ and $\p^\alpha = \p_{x_1}^{\alpha_1}\p_{x_2}^{\alpha_2}$. We write $\alpha\leqslant\beta$ if $\alpha_1\leqslant \beta_1$ and $\alpha_2\leqslant \beta_2$.
    \item For $s\in\R$ and $p\in(1,+\infty)$, $\mathrm{H}^{s,p}=\left\{u \in\mathcal{S}'(\mathbb{R}^2), \mathcal{F}^{-1}\left(\langle \eta \rangle^s \mathcal{F}u\right)\in\L^p(\mathbb{R}^2)\right\}$ denotes the Bessel potential space endowed with the norm $|u|_{\mathrm{H}^{s,p}_x} =  |\mathcal{F}^{-1}\left(\langle \eta \rangle^s \mathcal{F}u\right)|_{\L^p_x}$, where $\mathcal{F}$ denotes the usual Fourier transform. These spaces extend the usual Sobolev spaces by complex interpolation. In particular, we write $\mathrm{H}^s=\mathrm{H}^{s,2}$.
    \item We denote by $\mathcal{S}(\R^d)$ the space of smooth rapidly decaying functions on $\R^d$ and by $\mathcal{S}'(\R^d)$ its continuous dual, the space of tempered distributions on $\R^d$. We endow the space of tempered distributions with the real bracket
    $$\forall T\in\mathcal{S}'(\mathbb{R}^2)\cap\L^1_{loc}(\mathbb{R}^2),\forall \phi\in\mathcal{S}(\mathbb{R}^2), \langle T,\phi\rangle_{\mathcal{S}',\mathcal{S}} = \re\left(\int_{\mathbb{R}^2} T(x) \overline{\phi(x)}\d x\right)$$
    so that for $T\in\L^2(\mathbb{R}^2)$, $\langle T,\phi\rangle_{\mathcal{S}',\mathcal{S}}=( T,\phi)_{\L^2}$.
    \item For a Banach space $E$, we denote by $|\cdot|_E$ its norm and $\mathcal{L}(E)$ the set of bounded linear maps from $E$ to itself. Moreover, we denote by $E'$ its continuous dual and by $\langle f,e\rangle_{E',E}=f(e)$ for any $f\in E'$ and $e\in E$.
    \item We write $A\lesssim B$ if there exists a constant $C>0$ independent of $A$ and $B$ such that $A\leqslant C B$. We write $A\approx B$ if $A\lesssim B$ and $B\lesssim A$. If the constant $C$ depend on parameters $\theta$, we may write $A\lesssim_\theta B$.
\end{itemize}

\section{Preliminaries}

\subsection{Hermite-Sobolev spaces}\label{Sub:Confining_potential}

We denote by $H=\Delta-|x|^2$ the Hermite operator. It is well-known that this operator has eigenfunctions $(h_k)_{k\in\N}$ verifying the relation $-H h_k = \lambda_k^2 h_k$ where $\lambda_k^2\sim c\sqrt{k}$. Moreover, $(h_k)_{k\in\N}$ is a complete orthonormal system of $\L^2(\mathbb{R}^2,\R)$.\\

Let $s\in\R$ and $p\in(1,+\infty)$ and define
$$\W^{s,p} =\left\{u\in\mathcal{S}'(\mathbb{R}^2), (-H)^{\frac{s}{2}}u\in\L^p(\mathbb{R}^2)\right\}$$
endowed with the norm $$|u|_{\W^{s,p}_x} = |(-H)^{\frac{s}{2}}u|_{\L^p_x}.$$
It is known that these spaces are Banach spaces and that for $q$ such that $\frac{1}{p}+\frac{1}{q}=1$, we have $\left(\W^{s,p}\right)'=\W^{-s,q}$ with equal norm.\\

Moreover, the subspace $\vspan\left\{h_k, k\in\N\right\}$ is dense in $\W^{s,q}$-spaces (see for example \cite{BongioanniHSspaces} for the case $s\geqslant0$) and for $s\geqslant0$, an equivalent norm is given for $p\in(1,+\infty)$ by (see \cite{Dziubanski})
$$|u|_{\W^{s,p}}\approx|u|_{\H^{s,p}_x}+|\langle x\rangle^s u|_{\L^p_x}.$$
This shows $\W^{s,p}=\left\{u\in \H^{s,p}(\mathbb{R}^2), \langle x\rangle^s u\in\L^p(\mathbb{R}^2)\right\}.$ The following result, is then a consequence of Theorem 2.7.1 in \cite{bergh2011interpolation}, characterizes $\W^{-s,q}$ for $s>0$ and $q\in(1,+\infty)$. 
\begin{lemme}\label[lemme]{Lem:CaracW-sq}
    Let $s>0$ and $q\in(1,+\infty)$. For any $f\in\W^{-s,q}$, there exists $f_r\in \mathrm{H}^{-s,q}$ and $f_l\in\L^q(\R^2)$ such that $f = f_r + \langle x\rangle^{s}f_l$ and
    $$|f|_{\W^{-s,q}} \approx \inf\left\{|f_r|_{\mathrm{H}^{-s,q}}+|f_l|_{\L^q},\; f = f_r + \langle x\rangle^{s}f_l\right\}$$
    in the sense of norm equivalence.
\end{lemme}
In particular, elements of $\W^{-s,q}$ may lack integrability and grow at infinity. As a typical example, let $\alpha\geqslant 0$, then $\langle x\rangle^\alpha\in\W^{-s,q}$ if and only if $s>\alpha+\frac{2}{q}$. Thus, for $s>\frac{2}{q}$ there exists unbounded smooth and globally Lipschitz functions in $\W^{-s,q}$.\\

Using the norm equivalence, we define $\W^{2,p}$ with $p\in\{1,+\infty\}$ as
$$\W^{2,p} = \left\{u\in \H^{2,p}(\mathbb{R}^2), \langle x\rangle^{2} u\in\L^p(\mathbb{R}^2)\right\}$$
endowed with the norm
$$|u|_{\W^{2,p}_x} = |u|_{\H^{2,p}_x}+|\langle x\rangle^{2} u|_{\L^p_{x}}$$
and extend this definition to $s\in[0,2]$ by complex interpolation. Moreover inspired by the relation $\left(\W^{s,p}\right)'=\W^{-s,q}$, for $s\in[-2,0)$, we define $\W^{s,\infty} = \left(\W^{-s,1}\right)'$. As usual, we cannot define spaces of negative $\L^1$ regularity as dual spaces of positive $\L^\infty$ regularity.

\begin{rem}{~}
    We do not claim that the first definition of $\W^{s,p}$ agrees with the one we give for spaces over $\L^1$ and $\L^\infty$. We use the same notation for simplicity as we will always control $\W^{s,\infty}$ norms by Sobolev embeddings.
\end{rem}

The choices we made ensure us that we can interpolate between spaces of positive regularity and some other convenient properties that we need in our analysis. For example, our choices allow us to have a simple product rule on our spaces, whose proof follows from Leibniz rule, Hölder inequality and complex bilinear interpolation. We start with product rules, recalling an important estimate going back to Kato and Ponce \cite{KatoPonce}. 
\begin{prop}{(Theorem 1.4 in \cite{GulisashviliKon96})} \label[prop]{Prop:Kato-Ponce}
    Let $r\in(1,+\infty)$ and $s\geqslant 0$. For any $1<p_1,q_1,p_2,q_2\leqslant +\infty$ with $\frac{1}{r}=\frac{1}{p_j}+\frac{1}{q_j}$ ($j\in\{1,2\}$), there exists $C>0$ such that, for any $f\in \mathrm{H}^{s,p_1}\cap\L^{p_2}(\R^2)$ and any $g\in \mathrm{H}^{s,q_2}\cap\L^{q_1}(\R^2)$, it holds
    $$|fg|_{\mathrm{H}^{s,r}}\leqslant C\left(|f|_{\mathrm{H}^{s,p_1}}|g|_{\L^{q_1}}+|f|_{\L^{p_2}}|g|_{\mathrm{H}^{s,q_2}}\right).$$    
\end{prop}

The proof of \cref{Lem:ProductRule,Cor:ProductRuleNegPosReg,Prop:action_d/dx,Sobolev-embeddings} below are given in \cite{Mackowiak_2025}.

\begin{lemme}\label[lemme]{Lem:ProductRule}
    Let $s\geqslant0$ and $1\leqslant p,q,r\leqslant+\infty$ such that $\frac{1}{p}+\frac{1}{q}=\frac{1}{r}$. There exists a constant $C>0$ such that for all $u\in\W^{s,p}$ and $v\in \mathrm{H}^{s,q}$, we have $uv\in\W^{s,r}$ and $|uv|_{\W^{s,r}}\leqslant C |u|_{\W^{s,p}}|v|_{\mathrm{H}^{s,q}}$.   
\end{lemme}

\begin{cor}\label[cor]{Cor:ProductRuleNegPosReg}
    Let $s\geqslant0$, $1<p<+\infty$ and $1\leqslant q,r\leqslant+\infty$ such that $1-\frac{1}{p}+\frac{1}{q}=1-\frac{1}{r}$. There exists a constant $C>0$ such that for all $u\in\W^{-s,r}$ and $v\in \mathrm{H}^{s,q}$, we have $uv\in\W^{-s,p}$ and $|uv|_{\W^{-s,p}}\leqslant C |u|_{\W^{-s,r}}|v|_{\mathrm{H}^{s,q}}$.
\end{cor}

On $\W^{s,q}$ spaces, derivation and multiplication by a power of $\langle x\rangle$ act directly on the regularity exponent.

\begin{prop}(see Corollary 2.5 in \cite{Mackowiak_2025})  \label[prop]{Prop:action_d/dx}
    Let $s\in\R$, $j\in\{1,2\}$ and $q\in(1,+\infty)$, then $\p_j,x_j\in\mathcal{L}(\W^{s,q},\W^{s-1,q})$.
    
\end{prop}   

The following corollary follows from \cref{Prop:action_d/dx} by Stein's complex interpolation method \cite{SteinInterpolation,VoigtAbstractInterpolation}.

\begin{cor}\label[cor]{Cor:action_V}

    Let $\alpha\geqslant 0$, $s\in\R$ and $q\in(1,+\infty)$. Then, $\langle x\rangle^{\alpha}\in\mathcal{L}(\W^{s,q},\W^{s-\alpha,q})$.
    
\end{cor}

The Sobolev-Hermite spaces verify some continuous embeddings as in the classical Sobolev framework. We will also refer to these embeddings as Sobolev embeddings.

\begin{prop}{(Sobolev embeddings)}\label[prop]{Sobolev-embeddings}
    Let $1< p\leqslant q<+\infty$ and $s>\sigma$ such that $\frac{1}{p}-\frac{s}{2}\leqslant\frac{1}{q}-\frac{\sigma}{2}$. Then $\W^{s,p}$ is continuously embedded in $\W^{\sigma, q}$. Moreover if $1<p< q\leqslant+\infty$ and $s>\sigma$ such that $\frac{1}{p}-\frac{s}{2}<\frac{1}{q}-\frac{\sigma}{2}$ and $\sigma\in[-2,2]$ if $q=+\infty$, then $\W^{s,p}$ is compactly embedded in $\W^{\sigma,q}$.
    
\end{prop}

Finally, we can also prove a "power rule" on $\W^{s,2}$ spaces.
\begin{lemme}{(Power rule)}\label[lemme]{Lem:PowerRuleWs2}
     Let $s\in[0,1)$ and $\alpha>0$. Let $u,u_0,u_1\in\W^{s,2}\cap\L^\infty(\R^2)$, then $|u|^\alpha u\in\W^{s,2}$ and there exists a constant $C=C(s,\alpha)>0$ such that
     \begin{equation}\label{Eq:PowerRule1}
         \left||u|^\alpha u\right|_{\W^{s,2}} \leqslant C |u|_{\L^\infty}^\alpha |u|_{\W^{s,2}}
     \end{equation}
    and
    \begin{equation}\label{Eq:PowerRule2}
        \left||u_0|^\alpha u_0-|u_1|^\alpha u_1\right|_{\W^{s,2}} \leqslant C \left(|u_0|_{\L^\infty}^\alpha+|u_1|_{\L^\infty}^\alpha\right) |u_0-u_1|_{\W^{s,2}}.
    \end{equation}

\end{lemme}

\begin{proof}
    Let $s\in[0,1)$ and $\alpha>0$. For $s=0$, the result follows by Hölder's inequality. Now, assume $s\in(0,1)$. Clearly, \cref{Eq:PowerRule2} implies \cref{Eq:PowerRule1}. Remark that for any $f\in\W^{s,2}$, by norm equivalence, it holds
    \begin{equation}\label{Eq-NormEquivWs2PowerRule}
        |f|_{\W^{s,2}}^2 \approx_{s} |f|_{H^s}^2 + \left|\langle x\rangle^s f\right|_{\L^2}^2\
        \approx_{s} |f|_{\Dot{H}^s}^2 + \left|\langle x\rangle^s f\right|_{\L^2}^2,
    \end{equation}
    where
    \begin{equation}\label{Eq-NormEquivHomogeneousHsPowerRule}
        |f|_{\Dot{H}^s}^2 = \int_{\R^2} |\eta|^{2s} |\mathcal{F}f(\eta)|^2\d \eta \approx_{s}\int_{\R^2}\int_{\R^2} \frac{|f(x)-f(y)|^2}{|x-y|^{2+2s}}\d x\d y,
    \end{equation}
    by Proposition 1.3.7 in \cite{BahouriCheminDanchin}.
    Let $u_1,u_2\in\W^{s,2}\cap\L^\infty(\R^2)$ and for $t\in(0,1)$, let $u_t = u_0 + t(u_1-u_0)$. Write $F:x\in\R^2\mapsto|x|^\alpha x\in\R^2$. Then, it holds
    $$F(u_1)-F(u_0) = \int_0^1 \frac{\d}{\d t}F(u_t) \d t = \int_0^1 DF(u_t)\cdot(u_1-u_0)\d t,$$
    with $|DF(x)| \lesssim |x|^\alpha.$ Let $f=F(u_1)-F(u_0)$, it holds
    $$|f| \lesssim_{\alpha} \int_0^1 |u_t|^{\alpha}|u_1-u_0|\d t\lesssim_{\alpha} \left(|u_0|^{\alpha}_{\L^\infty}+|u_1|^{\alpha}_{\L^\infty}\right) |u_1-u_0|.$$
    Thus, 
    $$\left|\langle x\rangle^sf\right|_{\L^2}\lesssim_{\alpha} \left(|u_0|^{\alpha}_{\L^\infty}+|u_1|^{\alpha}_{\L^\infty}\right) \left|\langle x\rangle^s(u_1-u_0)\right|_{\L^2}.$$
    Moreover, for $x,y\in\R^d$, it holds
    \begin{align*}
        |f(x)-f(y)| &\lesssim_{\alpha} \int_0^1 |u_t|^{\alpha}_{\L^\infty}|(u_1-u_0)(x)-(u_1-u_0)(y)|\d t\\
        &\lesssim_{\alpha} \left(|u_0|^{\alpha}_{\L^\infty}+|u_1|^{\alpha}_{\L^\infty}\right)|(u_1-u_0)(x)-(u_1-u_0)(y)|.
    \end{align*}
    Hence, \cref{Eq-NormEquivHomogeneousHsPowerRule} implies
    $$\left||u_1|^\alpha u_1-|u_0|^\alpha u_0\right|_{\Dot{H}^s} \lesssim_{s,\alpha} \left(|u_0|_{\L^\infty}^\alpha+|u_1|_{\L^\infty}^\alpha\right) |u_1-u_0|_{\Dot{H}^s}$$
    and thus , by \cref{Eq-NormEquivWs2PowerRule},
    $$\left||u_1|^\alpha u_1-|u_0|^\alpha u_0\right|_{\W^{s,2}} \lesssim_{s,\alpha} \left(|u_0|_{\L^\infty}^\alpha+|u_1|_{\L^\infty}^\alpha\right) |u_1-u_0|_{\W^{s,2}}.$$
\end{proof}

\subsection{Weighted Besov spaces}

In this section, we introduce the tools we need from the theory of weighted Besov spaces. For more details on Besov spaces and Paley-Littlewood theory, the interested reader can refer to Chapter 2 in \cite{BahouriCheminDanchin} for unweighted Besov spaces and to Chapter 6 in \cite{TriebelFunctionSpacesIII} for weighted spaces.\\

A function $w\in\mathcal{C}^\infty(\R^2,\R_+)$ is called a (admissible) weight if
\begin{itemize}
    \item $\forall\alpha\in\N^2,\; \exists c_\alpha>0,\;  |\p^\alpha w|\leqslant c_\alpha w$,
    \item $\exists c\geqslant 0,\; \exists C>0,\; \forall x,y\in\R^2,\; 0<w(x)<Cw(y)\langle x-y\rangle^c$.
\end{itemize}
If $w$ is a weight, we define the weighted $\L^p$ space as 
$$\L^p_w = \{u\in\mathcal{S}'(\R^2),\; wu\in\L^p(\R^2)\} = \L^p(\R^2,w^p\d x)$$
endowed with the norm
$$|u|_{\L^p_w}=|wu|_{\L^p}.$$ 
In the particular case where the weight is given by $w(x)=\langle x\rangle^\alpha$, we denote $\L^2_\alpha = \L^2_w$.\\

We now introduce Paley-Littlewood localization operators. Let $(\chi_j)_{j\geqslant -1}\subset \mathcal{C}^\infty_c(\R^2)$. We say $(\chi_j)_{j\geqslant -1}$ is a dyadic partition of unity if
\begin{itemize}
    \item There exists a ball $B$ with $\supp\chi_{-1}\subset B$
    \item There exists a ring $R$ with $\supp\chi_{j}\subset 2^jR$ for all $j\in\N$
    \item For all $j\in\N^*$, $\chi_j\chi_{-1}=0$ and for all $i,j\in\N$, $|i-j|>1\Rightarrow\chi_i\chi_j=0$
    \item For all $x\in\R^2$, $\sum_{j\geqslant -1} \chi_j(x)=1$.
\end{itemize}
We associate to a dyadic partition of unity $(\chi_j)_{j\geqslant -1}$, the family $(\Delta_j)_{j\geqslant -1}$ of endomorphisms of $\mathcal{S}'(\R^2)$ given by
$$\forall T\in\mathcal{S}'(\R^2),\; \forall j\geqslant -1,\; \mathcal{F}(\Delta_j T)=\chi_j\mathcal{F}T,$$
where $\mathcal{F}\in\mathcal{L}(\mathcal{S}'(\R^2))$ is the Fourier transform, defined on $\mathcal{S}(\R^2)$ by
$$\forall \phi\in\mathcal{S}(\R^2),\; \forall\eta\in\R^2,\; (\mathcal{F}\phi)(\eta)=\int_{\R^2}\e^{-\mathrm{i} \eta\cdot x}\phi(x)\frac{\d x}{2\pi}$$
and extended by duality. We call $(\Delta_j)_{j\geqslant -1}$ the Paley-Littlewood decomposition associated to $(\chi_j)_{j\geqslant -1}$.\\

Let $s\in\R$, $p,q\in[1,+\infty]$ and $w$ an admissible weight, the weighted Besov space $\mathrm{B}^s_{p,q,w}$ is defined as
$$\left\{T\in\mathcal{S}'(\R^2),\; |T|_{\mathrm{B}^s_{p,q,w}}<+\infty\right\}$$
with 
$$|T|_{\mathrm{B}^s_{p,q,w}} = \left|2^{sj}|\Delta_j T|_{\L^p_w}\right|_{l^q_{j\geqslant -1}}.$$
It is well known that $\mathrm{B}^s_{p,q,w}$ is a Banach space and does not depend on the choice of the dyadic partition of unity (see for example \cref{Lem:EquivBesovSum2jAnnulus} below). In the special case $w=\textbf{1}$, we use the lighter notation $\mathrm{B}^s_{p,q}=\mathrm{B}^s_{p,q,\textbf{1}}$. Moreover, Theorem 6.5 in \cite{TriebelFunctionSpacesIII} implies
\begin{equation}
    |T|_{\mathrm{B}^s_{p,q,w}} \approx |wT|_{\mathrm{B}^s_{p,q}} \label{Eq:NormEquivWeightedBesov}
\end{equation}
in the sense of norm equivalence. In particular,
$$\mathrm{B}^s_{p,q,w} = \left\{T\in\mathcal{S}'(\R^2),\; wT\in \mathrm{B}^s_{p,q}\right\}.$$
When $p=q=+\infty$, we get weighted Hölder-Zygmund spaces $\mathcal{C}^s_w=\mathrm{B}^s_{\infty,\infty,w}$ and when $p=q=2$, we recover the standard weighted Sobolev spaces $\mathrm{H}^s_w=\mathrm{B}^s_{2,2,w}$. Likewise, for $s\in\R$ and $p\in(1,+\infty)$, we also define $\mathrm{H}^{s,p}_w$ as 
$$\mathrm{H}^{s,p}_w =\left\{T\in\mathcal{S}'(\R^2),\; wT\in \mathrm{H}^{s,p}\right\}$$
endowed with the norm
$$|T|_{\mathrm{H}^{s,p}_w} = |wT|_{\mathrm{H}^{s,p}}.$$
As previously, we omit the subscript $w$ when $w=\textbf{1}$. As for weighted $\L^p$ spaces, we use the notation subscript $\alpha$ when $w(x)=\langle x\rangle^\alpha$.\\

We easily deduce from \cref{Eq:NormEquivWeightedBesov} and usual Besov embeddings (see Theorems 2.40. and 2.41. and Proposition 2.7.1 in \cite{BahouriCheminDanchin}) the following weighted Besov embeddings.
\begin{prop}\label[prop]{Prop:BesovEmbedding}
    Let $s\in\R$, $1\leqslant p_1,p_2,q_1,q_2\leqslant +\infty$ and $w_1$ and $w_2$ be admissible weights with $w_1\geqslant w_2$. The following embeddings are continuous
    \begin{enumerate}
        \item For $\eps>0$, $1\leqslant p \leqslant +\infty$ and $1\leqslant q_1, q_2\leqslant +\infty$, $\mathrm{B}^{s+\eps}_{p,q_1,w_1}\subset \mathrm{B}^s_{p,q_2,w_2},$
        \item For $1\leqslant p_1\leqslant p_2 \leqslant +\infty$ and $1\leqslant q_1\leqslant q_2\leqslant +\infty$, $\mathrm{B}^{s}_{p_1,q_1,w_1}\subset \mathrm{B}^{s-2\left(\frac{1}{p_1}-\frac{1}{p_2}\right)}_{p_2,q_2,w_2},$
        \item For $p\in[2,+\infty)$, $\L^p_{w_1}(\R^2)\subset \mathrm{B}^0_{p,p,w_2}$ and $\mathrm{B}^0_{p,2,w_1}\subset\L^p_{w_2}(\R^2)$.
    \end{enumerate}
\end{prop}

One easily checks the Paley-Littlewood projectors are bounded over Besov spaces as follow
$$\forall s\in\R,\; \forall p,q\in[1,+\infty],\; \forall j\geqslant -1,\; |\Delta_j|_{\mathcal{L}\left(\mathrm{B}^s_{p,q}\right)}=1.$$
For $N\in\N$, we introduce the low-frequency and the high-frequency cut-off
$$\Delta_{\leqslant N} = \sum_{j=-1}^N \Delta_j \text{ and } \Delta_{>N} = \sum_{j>N} \Delta_j = \operatorname{Id}-\Delta_{\leqslant N},$$ they are bounded operator over Besov spaces, satisfying
$$\forall s\in\R,\; \forall p,q\in[1,+\infty],\; \forall N\in\N,\; |\Delta_{\leqslant N}|_{\mathcal{L}\left(\mathrm{B}^s_{p,q}\right)}=|\Delta_{> N}|_{\mathcal{L}\left(\mathrm{B}^s_{p,q}\right)}=1.$$

In presence of weight, the norm are no longer 1. The following results give bounds for low-frequency and high-frequency cut-off in weighted Besov spaces, using the weighted Young inequality for convolution proved in \cite{ugurcan2022anderson}.

\begin{lemme}\label[lemme]{Lem:EstimDeltaNLqw}

    Let $N\in\N$, $p\in[1,+\infty]$ and $\nu>0$. Then, there exists $C=C(p,\nu,N)>0$ such that
    $$|\Delta_{>N}|_{\mathcal{L}\left(\L^p_\nu\right)},|\Delta_{\leqslant N}|_{\mathcal{L}\left(\L^p_\nu\right)}\leqslant C.$$
    
\end{lemme}

\begin{proof}

    As $(\Delta_{>N}+\Delta_{\leqslant N})f=f$ for any regular enough function $f$, it is sufficient to prove the bound for $\Delta_{\leqslant N}$. Denote by $\chi_j\in\mathcal{C}^\infty_c(\R^2)$ a partition of unity such that $\Delta_j=\chi_j(-\Delta)$ ($j\geqslant-1$). Then, it follows from Lemma A.22 in \cite{ugurcan2022anderson},
    $$|\Delta_{\leqslant N} f|_{\L^p_\nu}=|\overline{\psi_N}*f|_{\L^p_\nu}\leqslant |\psi_N|_{\L^1_\nu}| f|_{\L^p_\nu},$$
    where $\psi_N=\frac{1}{2\pi}\sum_{j=-1}^N\mathcal{F}^{-1}(\chi_j)\in\mathcal{S}(\R^2)$.
    
\end{proof}

\begin{cor}\label[cor]{Cor:EstimDeltaN}

    Let $N\in\N$, $p,q\in[1,+\infty]$, $\alpha\in\R$. Let $w$ be a weight function. Then $$|\Delta_{>N}|_{\mathcal{L}\left(\mathrm{B}^{\alpha}_{p,q,w}\right)}\leqslant|\Delta_{>N}|_{\mathcal{L}\left(\L^p_w\right)} \text{ and }|\Delta_{\leqslant N}|_{\mathcal{L}\left(\mathrm{B}^{\alpha}_{p,q,w}\right)}\leqslant|\Delta_{\leqslant N}|_{\mathcal{L}\left(\L^p_w\right)}.$$ Moreover, for any $\eps>0$, there exists $C=C(\eps)>0$ such that
    $$\forall f\in \mathrm{B}^{\alpha+\eps}_{p,q},\; |\Delta_{>N}f|_{\mathrm{B}^\alpha_{p,q,w}}\leqslant C 2^{-N\eps}|\Delta_{>N}f|_{\mathrm{B}^{\alpha+\eps}_{p,q,w}}$$
    and
    $$\forall f\in \mathrm{B}^{\alpha}_{p,q},\; |\Delta_{\leqslant N}f|_{\mathrm{B}^\alpha_{p,q,w}}\leqslant C 2^{N\eps}|\Delta_{\leqslant N}f|_{\mathrm{B}^{\alpha-\eps}_{p,q,w}}.$$
    
\end{cor}

\begin{proof}
    Let $\eps\geqslant 0$. For $q<+\infty$, a direct computation gives
    $$|\Delta_{> N}f|_{B^\alpha_{p,q,w}}^q = \sum_{j\geqslant N} 2^{jq\alpha}|\Delta_j \Delta_{> N}f|_{\L^p_w}^q  \lesssim 2^{-qN\eps}|\Delta_{>N}f|_{B^{\alpha+\eps}_{p,q,w}}^q,$$
    and
    $$|\Delta_{\leqslant N}f|_{B^\alpha_{p,q,w}}^q = \sum_{j\geqslant -1}^{N+1} 2^{jq\alpha}|\Delta_j \Delta_{\leqslant N}f|_{\L^p_w}^q\lesssim 2^{qN\eps}|\Delta_{\leqslant N}f|_{B^{\alpha-\eps}_{p,q,w}}^q.$$
    The first claim follows taking $\eps=0$. An analogous proof works for $q=+\infty$.
    
\end{proof}

\subsection{Paradifferential calculus}\label{Sub:paracontrolled}

In this section, we introduce the tools from paradifferential calculus we need for our analysis. Once again, we refer the interested reader to Chapter 2 in \cite{BahouriCheminDanchin} for a simple introduction to paradifferential calculus in unweighted Besov spaces. Let us start with a definition of paraproducts and resonant products. Let $f,g\in\mathcal{S}'(\R^2)$, we define the paraproduct of $g$ by $f$ as 
\begin{equation}
    P_f g = \sum_{j\geqslant -1}\sum_{i<j-1} \Delta_i f \Delta_j g = \sum_{j=1}^{+\infty} \Delta_{\leqslant j-2}f \Delta_j g\label{Eq:DefParaproduct}
\end{equation}
and, if it exists, the resonant product of $f$ and $g$ as
\begin{equation}
    R(f,g) = \sum_{|i-j|\leqslant 1 } \Delta_i f \Delta_j g = \sum_{j\geqslant -1} \left(\Delta_{\leqslant j+1} f - \Delta_{\leqslant j-2} f\right)\Delta_j g.\label{Eq:DefResonnantProduct}
\end{equation}
Remark that, formally, we have the Bony's decomposition
\begin{equation}
    fg = P_f g + R(f,g) + P_g f \label{Eq:BonyDecomposition}
\end{equation} and one can show only $R(f,g)$ may not exist. The following continuity result for paraproducts and resonant products in weighted Besov spaces shows that when the sum of regularity of $f$ and $g$ is positive, then $R(f,g)$ is well-defined, so is $fg$. However, $R(f,g)$ is not the term with the worst regularity in the Bony's decomposition \ref{Eq:BonyDecomposition}). The proof of the following results are given in \cref{Sec:ProofParaprod}.

\begin{prop}\label[prop]{Prop:ContinuityParaprodutcs}
    Let ${\alpha_1},{\alpha_2}\in\R$, $p\in(1,+\infty)$, $\eps>0$ and $\eps_1,\eps_2\in[0,\eps]$ such that $\eps_1+\eps_2 = \eps$. Let $w_1,w_2$ be admissible weights. 
    \begin{enumerate}
        \item Assume $\alpha_1\geqslant 0$, then 
        $$\exists C>0,\; \forall f\in \mathrm{H}^{\alpha_1,p}_{w_1},\; \forall g\in\mathcal{C}^{\alpha_2+\eps}_{w_2},\; |P_fg|_{\mathrm{H}^{\alpha_2,p}_{w_1w_2}}\leqslant C|f|_{\mathrm{H}^{\alpha_1,p}_{w_1}}|g|_{\mathcal{C}^{\alpha_2+\eps}_{w_2}}$$
        and
        $$\exists C>0,\; \forall f\in \mathcal{C}^{\alpha_1}_{w_1},\; \forall g\in \mathrm{H}^{\alpha_2+\eps,p}_{w_2},\; |P_fg|_{\mathrm{H}^{\alpha_2,p}_{w_1w_2}}\leqslant C|f|_{\mathcal{C}^{\alpha_1}_{w_1}}|g|_{\mathrm{H}^{\alpha_2+\eps,p}_{w_2}}.$$
        \item Assume $\alpha_1+\eps_1<0$, then 
        $$\exists C>0,\; \forall f\in \mathrm{H}^{\alpha_1+\eps_1,p}_{w_1},\; \forall g\in\mathcal{C}^{\alpha_2+\eps_2}_{w_2},\; |P_fg|_{\mathrm{H}^{\alpha_1+\alpha_2,p}_{w_1w_2}}\leqslant C|f|_{\mathrm{H}^{\alpha_1+\eps_1,p}_{w_1}}|g|_{\mathcal{C}^{\alpha_2+\eps_2}_{w_2}}$$
        and
        $$\exists C>0,\; \forall f\in \mathcal{C}^{\alpha_1+\eps_1}_{w_1},\; \forall g\in \mathrm{H}^{\alpha_2+\eps_2,p}_{w_2},\; |P_fg|_{\mathrm{H}^{\alpha_1+\alpha_2,p}_{w_1w_2}}\leqslant C|f|_{\mathcal{C}^{\alpha_1+\eps_1}_{w_1}}|g|_{\mathrm{H}^{\alpha_2+\eps_2,p}_{w_2}}.$$
        \item Assume $\alpha_1+\alpha_2+\eps>0$, then $$\exists C>0,\; \forall f\in \mathrm{H}^{\alpha_1+\eps_1,p}_{w_1},\; \forall g\in\mathcal{C}^{\alpha_2+\eps_2}_{w_2},\; |R(f,g)|_{\mathrm{H}^{\alpha_1+\alpha_2,p}_{w_1w_2}}\leqslant C|f|_{\mathrm{H}^{\alpha_1+\eps_1,p}_{w_1}}|g|_{\mathcal{C}^{\alpha_2+\eps_2}_{w_2}}.$$
    \end{enumerate}
\end{prop}

We deduce the following corollary, that allows us to work in Sobolev spaces $\W^{s,p}$ instead of weighted Sobolev spaces. The specific weighted Hölder spaces used in \cref{Cor:ContinuityParaprodutcsWsp} are motivated by their use in \cref{Sub:StrichartzProof}.

\begin{cor}\label[cor]{Cor:ContinuityParaprodutcsWsp}
    Let $0\leqslant\sigma'\leqslant \sigma$, $\eps,\eps'>0$, $0\leqslant \kappa'<\kappa$ and $p\in[2,+\infty)$, then
    \begin{enumerate}
        \item There exists $C>0$ such that for any $ f\in \mathcal{C}^{\eps'}$ and any $g\in\W^{\sigma+\eps,p}$, it holds
        $$|P_f g|_{\W^{\sigma,p}}\leqslant C|f|_{\mathcal{C}^{\eps'}}|g|_{\W^{\sigma+\eps,p}}.$$
        Moreover, there exists $C>0$ such that for any $ f\in \W^{\sigma-\sigma',p}$ and any $g\in\mathcal{C}^{\sigma+\eps}\cap\mathcal{C}^{\eps'}_{\sigma'}$, it holds
        $$|P_f g|_{\W^{\sigma,p}}\leqslant C|f|_{\W^{\sigma-\sigma',p}}\left(|g|_{\mathcal{C}^{\sigma+\eps}}+|g|_{\mathcal{C}^{\eps'}_{\sigma'}}\right).$$
        \item There exists $C>0$ such that for any $ f\in\mathcal{C}^{\kappa'-\kappa}_{-\kappa'}$ and any $g\in\W^{\sigma+\kappa+\eps,p}$, it holds
        $$|P_f g|_{\W^{\sigma,p}}\leqslant C|f|_{\mathcal{C}^{\kappa'-\kappa}_{-\kappa'}}|g|_{\W^{\sigma+\kappa+\eps,p}}.$$
        \item There exists $C>0$ such that for any $ f\in \W^{\sigma+\kappa+\eps,p}$ and any $g\in\mathcal{C}^{\kappa'-\kappa}_{-\kappa'}$, it holds
        $$ |R(f,g)|_{\W^{\sigma,p}}\leqslant C|f|_{\W^{\sigma+\kappa+\eps,p}}|g|_{\mathcal{C}^{\kappa'-\kappa}_{-\kappa'}}.$$
    \end{enumerate}
\end{cor}

\begin{rem}\label{Rem:Eps=0ParaproductWs2}
    In \cref{Prop:ContinuityParaprodutcs} 2. and 3. and \cref{Cor:ContinuityParaprodutcsWsp} 2. and 3., if $p=2$, then one can chose $\eps=0$ because $\mathrm{H}^s_w=\mathrm{H}^{s,2}_w=\mathrm{B}^s_{2,2,w}$ (see \cref{Prop:ContinuityParadifferentielGeneral} from which \cref{Prop:ContinuityParaprodutcs,Cor:ContinuityParaprodutcsWsp} are proven).
\end{rem}

When we need to insure some smallness in paraproducts, we may introduce truncated paraproducts defined for $N\in\N$ by
\begin{equation}
    P^N_{f}g = P_f\left(\Delta_{>N} g\right).\label{Eq:DefTruncatedParaprod}
\end{equation}
\cref{Lem:EstimDeltaNLqw,Cor:EstimDeltaN} imply $P^N$ verify the same kind of estimate than \cref{Prop:ContinuityParaprodutcs,Cor:ContinuityParaprodutcsWsp}. Finally, we give a result on continuity of commutators between truncated paraproducts and $H$ in $\W^{s,2}$ spaces.
\begin{lemme}\label[lemme]{Lem:ComPuH}
    Let $\eps>0$, $\kappa\in(0,1)$, $s\in[0,\kappa)$, $\sigma\in (0,1-\kappa)$ and $N\in\N$, there exists $C=C(\sigma,\kappa,s,\eps,N)>0$ such that for all $\forall f\in \W^{\sigma+\kappa+\eps,2}$ and all $\forall g\in \mathcal{C}^{2-\kappa}\cap\mathcal{C}^{-\kappa+s}_{2-s}$, it holds
    $$|[P^N_f,H]g|_{\W^{\sigma,2}}\leqslant C |f|_{\W^{\sigma+\kappa+\eps,2}}\left(|g|_{\mathcal{C}^{2-\kappa}}+|g|_{\mathcal{C}^{-\kappa+s}_{2-s}}\right).$$
\end{lemme}

\section{Construction of the operator}\label{Sec:ConstructionOperator}

\subsection{The quadratic form}

For $\alpha\in\R$ and $q\in[1,+\infty]$, define the set of enhanced noises $\mathcal{X}^{\alpha,q}$ as the closure of 
$$\left\{(f,|\nabla f|^2-c),\; f,c\in\mathcal{S}(\R^d)
\right\}$$
in $\left(\W^{\alpha+2,q}\times\W^{\alpha+1,q}\right)$. Let $\kappa\in\left(0,\frac{1}{2}\right)$ and $\xi\in\W^{-1-\kappa,\infty}$. We say $\Xi=(Y,Z)$ is an enhanced noise associated to $\xi$ if
\begin{equation}\label{Def:EnhancedNoise}
    -HY=\xi \text{ and } \Xi\in\mathcal{X}^{-1-\kappa,\infty}.
\end{equation}
\begin{rem}
    Let us stress out some points regarding the definition of the set of enhanced noise $\mathcal{X}^{s,q}$ compared to the one used, for example, in \cite{allez2015continuous} on the torus.
    \begin{itemize}
        \item We chose this definition as it corresponds to what we have in Proposition 3.11 in \cite{Mackowiak_2025} when renormalizing $|\nabla Y|^2$ for $\xi$ a white noise. The inhomogeneous setting implies the "renormalization constant" is no more a constant but rather a function, this is why the second term of the smoothed enhanced noise is defined as $|\nabla f|^2-c$ with $c$ a smooth function and not a constant. This was already the case on compact manifold (see for example \cite{mouzard2021weyl}). Let us mention the "renormalization function", on compact manifold, can be decomposed in a diverging constant part and a possibly non-constant bounded part (see for example \cite{bailleul2022analysis,DahlqvistPAM}). It is not clear if it is true in our setting too.
        \item As we are working on the full space, which is unbounded, it is convenient to impose some localization at infinity to smoothed enhanced noises.
    \end{itemize}
\end{rem}
In what follows, let $\Xi=(Y,Z)$ be an enhanced noise associated to $\xi$ satisfying
\begin{equation}
    \exists \kappa\in\left(0,\frac{1}{2}\right),\; \exists q>\frac{2}{\kappa},\; \exists s\in\left(0,\kappa-\frac{2}{q}\right),\; \Xi\in\W^{1-s,q}\times\W^{-s,q}.\label{Eq:Condqskappa}
\end{equation}
Let us stress out that,  choosing $q$ large enough and $s$ close enough to $\kappa-\frac{2}{q}$,  every enhanced noise $\Xi\in\mathcal{X}^{-1-\kappa+,\infty}$ verifies \cref{Eq:Condqskappa}. Thus, \cref{Eq:Condqskappa} is just an arbitrarily small loss of derivative.  As the regularity of the white noised obtained in Proposition 3.1 in \cite{Mackowiak_2025} and the condition $\kappa\in(0,\frac{1}{2})$ are not sharp, this arbitrarily small loss is negligible.\\

Let $\rho=\e^Y$. The following lemma, which is a consequence of Sobolev embeddings and \cref{Prop:action_d/dx,Cor:action_V,Prop:BesovEmbedding}, sum up the regularity of $Y,Z$ and $\rho$.
\begin{lemme}\label[lemme]{Lem:RegXI}
    It holds
    \begin{itemize}
        \item $\nabla Y, xY, Z \in \W^{-\kappa,\infty}$,
        \item $Y,\rho\in \mathrm{H}^{1-\kappa,\infty}\subset\mathcal{C}^{1-\kappa}$ and $\nabla Y\in\mathcal{C}^{-\kappa}$,
        \item $|x|^{1-\kappa}Y\subset\L^\infty(\R^2)$.
    \end{itemize}
\end{lemme}

In the case where $\xi$ is a white noise, according to Propositions 3.1 and 3.11 in \cite{Mackowiak_2025}, we can construct an enhanced noise $\Xi = (Y,\wick{|\nabla Y|^2})$ which almost surely belongs to $\W^{1-s,q}\times\W^{-s,q}$ for any $q>2$ and any $s>\frac{2}{q}$, thus verifying \cref{Eq:Condqskappa} for any $\kappa>0$.\\

In order to obtain the formal quadratic form associated to $\wick{H+\xi}$, we use a gauge transform inspired by the one introduced in \cite{Haire_Labbe}. Let define the space $\mathcal{D}^{1,2}=\rho\W^{1,2}$ endowed with the norm
$$\forall u=\rho v\in \mathcal{D}^{1,2},\; |u|_{\mathcal{D}^{1,2}}^2=\int_{\R^2}\left(|\nabla v|^2+|xv|^2\right)\rho^2(x)\d x.$$ 
Remark that this norm verifies
\begin{equation}
    \forall v\in \W^{1,2},\; (\inf\rho)|v|_{\W^{1,2}}\leqslant|\rho v|_{\mathcal{D}^{1,2}}\leqslant(\sup\rho)|v|_{\W^{1,2}},\label{Eq:NormEquivD12}
\end{equation}
thus it immediately follows that $(\mathcal{D}^{1,2},|\cdot|_{\mathcal{D}^{1,2}})$ is a Hilbert space. We deduce the following embeddings for $\mathcal{D}^{1,2}$.

\begin{lemme}\label[lemme]{Lem:EmbeddingQ}
    For any $q\in[2,+\infty)$, $\mathcal{D}^{1,2}$ is compactly embedded in $\L^q(\R^2)$.
\end{lemme}

For any $u_i=\rho v_i\in \mathcal{D}^{1,2}$, formally, one can write
$$\langle-\wick{H+\xi}u_1,u_2\rangle=\re \int_{\R^2}\left(\nabla v_1\cdot \overline{\nabla v_2} + |x|^2(1-Y)v_1\overline{v_2}\right)\rho^2\d x - \langle Z, \overline{v_1}v_2\rho^2\rangle.$$
Let define the bilinear form
\begin{equation}
    \forall u_1,u_2\in \mathcal{D}^{1,2},\; a_\Xi(u_1,u_2)=\re \int_{\R^2}\left(\nabla v_1\cdot \overline{\nabla v_2} + |x|^2(1-Y)v_1\overline{v_2}\right)\rho^2\d x - \langle Z, \overline{v_1}v_2\rho^2\rangle,\label{Eq:QuadraticFormH+xi}
\end{equation}
where $u_i=\rho v_i$. This bilinear form is well-defined, symmetric and continuous on $\mathcal{D}^{1,2}$, as \cref{Lem:ProductRule,Lem:RegXI,Eq:NormEquivD12} imply
$$|\langle Z, \overline{v_1}v_2\rho^2\rangle|\lesssim |Z|_{\W^{-\kappa,\infty}}|\rho^2|_{\mathrm{H}^{\kappa,\infty}}|v_1|_{\W^{\kappa,2}}|v_2|_{\W^{\kappa,2}}\lesssim_\Xi |u_1|_{\mathcal{D}^{1,2}}|u_2|_{\mathcal{D}^{1,2}},$$
using
$$|\rho^2|_{\mathrm{H}^{\kappa,\infty}}\lesssim |\rho|_{\mathrm{H}^{1-\kappa,\infty}}^2<+\infty$$ as $\kappa<\frac{1}{2}$. Moreover it is quasi-coercive in the following sense.

\begin{prop}\label[prop]{Prop:Quasi-coercivity-a}

    There exists $\delta_\Xi>0$ such that
    $$\forall u=\rho v\in \mathcal{D}^{1,2},\; a_\Xi(u,u)+\delta_\Xi|u|_{\L^2}^2\geqslant \frac{1}{2}|u|_{\mathcal{D}^{1,2}}^2.$$
    
\end{prop}

\begin{proof}
    Let $u\in \mathcal{D}^{1,2}$, we have
    $$a_\Xi(u,u) = |u|_{\mathcal{D}^{1,2}}^2 - \int_{\R^2}Y\rho^2|xv|^2\d x - \langle Z, |v|^2\rho^2\rangle.$$
    First,
    \begin{align*}
        \left|\int_{\R^2}VY\rho^2|v|^2\d x\right|&\leqslant |\rho^2|_{\L^\infty}\left||x|^{1-\kappa}Y\right|_{\L^\infty}\left||x|^{\frac{1+\kappa}{2}}v\right|_{\L^2}^2
        &\lesssim |\rho^2|_{\L^\infty}\left||x|^{1-\kappa}Y\right|_{\L^\infty}\left|v\right|_{\W^{\frac{1+\kappa}{2},2}}^2
    \end{align*}
    and $|\rho^2|_{\L^\infty}\left||x|^{1-\kappa}Y\right|_{\L^\infty}<+\infty$ by \cref{Lem:RegXI}. By interpolation and Young's inequality, there exists $C_1(\Xi)>0$ such that
    $$\left|\int_{\R^2}VY\rho^2|v|^2\d x\right|\leqslant \frac{1}{4}|u|_{\mathcal{D}^{1,2}}^2+C_1(\Xi)|u|_{\L^{2}}^2.$$
    Now, \cref{Lem:ProductRule},  interpolation and Young's inequality imply
    $$|\langle Z, |v|^2\rho^2\rangle|\leqslant C |Z|_{\W^{-\kappa,\infty}}|\rho^2|_{\mathrm{H}^{\kappa,\infty}}|v|_{\W^{\kappa,2}}^2\leqslant \frac{1}{4}|u|_{\mathcal{D}^{1,2}}^{2}+C_2(\Xi)|u|_{\L^{2}}^2,$$
    using $|\rho^2|_{\mathrm{H}^{\kappa,\infty}}\lesssim|\rho^2|_{\mathrm{H}^{1-\kappa,\infty}}<+\infty$ as $\kappa<\frac{1}{2}$. Setting $\delta_\Xi=C_1(\Xi)+C_2(\Xi)$, one obtains the claim. 
\end{proof}

 Hence, Theorem 6.2.6 in \cite{KatoPerturbationBook} implies the existence of a unique positive self-adjoint operator $-A:\mathrm{D}(-A)\subset\L^2(\R^2)\to\L^2(\R^2)$ such that $\mathrm{D}(-A)$ is dense in $\mathcal{D}^{1,2}$ and
 $$\forall u_1,u_2\in \mathcal{D}^{1,2},\; a_\Xi(u_1,u_2)+\delta_\Xi(u_1,u_2)_{\L^2} = \langle-Au_1,u_2\rangle_{\mathcal{D}^{-1,2},\mathcal{D}^{1,2}}.$$ We then define $\wick{H+\xi}=A+\delta_\Xi$.\\

 Remark that $-A$ naturally defines an isomorphism from $\mathcal{D}^{1,2}$ onto $\mathcal{D}^{-1,2}$. We stress out that, in general, $\mathcal{D}^{-1,2}$ is not a space of distributions. In fact, for $u\in \mathcal{D}^{1,2}$ it is easily seen that $u \in \mathrm{H}^1_{loc}(\R^2)$ if and only if $u\nabla Y\in\L^2_{loc}(\R^2)$. Thus, if there is a non-zero smooth function in $\mathcal{D}^{1,2}$, it must be $\xi\in\H^{-1}(B)$ for some ball $B\subset\R^2$. This, for example, is impossible when $\xi$ is a white noise.

 \subsection{The domain of -A}

 Inspired by the definition of $\mathcal{D}^{1,2}$, we define
 \begin{equation}
     \forall s\in[0,2),\; \forall p\in(1,+\infty),\;  \mathcal{D}^{s,p}=\rho\W^{s,p} \text{ and } \mathcal{D}^{-s,p}=(\mathcal{D}^{s,p'})',\label{Def:Dsq}
 \end{equation}
with $\frac{1}{p}+\frac{1}{p'}=1$. Let us emphasise \cref{Lem:ProductRule,Cor:ProductRuleNegPosReg} imply that for $|s|\in[0,1-\kappa]$, $\mathcal{D}^{s,p}=\W^{s,p}$, but we cannot give an explicit definition of $\mathcal{D}^{-s,p}=(\rho\W^{s,p'})'$ for $s>1-\kappa$. We also define $\mathcal{D}^{2,2}=\mathrm{D}(-A)$ and $\mathcal{D}^{-2,2}=(\mathcal{D}^{2,2})'$. In this section, we will show the following characterization of $\mathcal{D}^{s,2}$.
\begin{prop}\label[prop]{Prop:CaracDs2}
    For any $s\in[\kappa-2,2-\kappa]$, $\mathcal{D}^{s,2} = \mathrm{D}\left((-A)^{\frac{s}{2}}\right)$ with equivalent norm.
\end{prop}
In particular, $\mathrm{D}(-A)$ is dense in any $\mathcal{D}^{s,2}$  with $|s|\leqslant 2-\kappa$, allowing to make sense of the formal formula $$Au = \rho\left(Hv+2\nabla Y\cdot \nabla v + xY\cdot xv + Zv - \delta_\Xi v\right).$$ In the case of a white noise, where $\kappa$ can be taken arbitrary small, \cref{Prop:CaracDs2} characterises all the $\L^2$-based Sobolev spaces of order between $-2$ and $2$.\\

The bilinear form $a_\Xi$, and thus the operator $-A$, are not defined directly on $u\in \mathcal{D}^{1,2}$ but rather by their action on $v=\rho^{-1}u\in\W^{1,2}$. Thus, studying the underlying operator on the $v$ side may give information on $-A$. Let 
$$\forall v_1,v_2\in \W^{1,2},\; b_\Xi(v_1,v_2)= a_\Xi(\rho v_1,\rho v_2) + \delta_\Xi(\rho v_1, \rho v_2)_{\L^2}.$$
This bilinear form is symmetric, continuous and coercive on $\W^{1,2}$, thus associated to a unique positive self-adjoint operator $(-B,\mathrm{D}(-B))$. Moreover, one can give an exact formula
 \begin{equation}\label{Eq:DefB}
     \forall v\in\W^{1,2},\; -Bv = -\nabla\cdot(\rho^2\nabla v) + |x|^2(1-Y)\rho^2 v - Z\rho^2 v + \delta_\Xi \rho^2 v \in\W^{-1,2}.
 \end{equation}
 Recall that the domains of $-A$ and $-B$ are defined, in view of Riesz theorem, by
 \begin{align*}
     \mathrm{D}(-A) &=\left\{u\in \mathcal{D}^{1,2},\; \exists C>0,\; \forall \phi\in \mathcal{D}^{1,2},\;  \left|a_\Xi(u,\phi)\right|\leqslant C |\phi|_{\L^2}\right\}
 \end{align*}
 and 
 $$\mathrm{D}(-B) =\left\{v\in \W^{1,2},\; \exists C>0,\; \forall \phi\in \W^{1,2},\;  \left|b_\Xi(v,\phi)\right|\leqslant C |\phi|_{\L^2}\right\}.$$

 We immediately deduce the following lemma, so we can obtain knowledge on elements of $\mathrm{D}(-A)$ by studying $B$.
 \begin{lemme}\label[lemme]{Lem:NormEquivDomA/B}
     We have $\mathrm{D}(-A)=\rho \mathrm{D}(-B)$ with equivalent norm.
 \end{lemme}
 
For $v\in\mathrm{D}(-B)\subset\W^{1,2}$, $xv\in\L^2(\R^2)$ thus \cref{Lem:RegXI} implies $|x|^{2-\kappa}Yv\in\L^2(\R^2)$. It follows by \cref{Cor:action_V} that $xY\cdot xv\in\W^{-\kappa,2}$ and
 $$\forall v\in \mathrm{D}(-B),\; -\nabla\cdot(\rho^2\nabla v) + |x|^2\rho^2 v \in \W^{-\kappa,2}.$$
 Thus, one can expect that $\mathrm{D}(-B)\subset \W^{2-\kappa,2}$ thanks to some elliptic regularity argument. To shorten the notation, we set $-Lv=-\nabla\cdot(\rho^2\nabla v) + |x|^2\rho^2 v$.

\begin{lemme}\label[lemme]{Lem:ReguEllB}
    The linear map $B$ restrains to an isomorphism from $\W^{2-\kappa,2}$ to $\W^{-\kappa,2}$ and extends to an isomorphism from $\W^{\kappa,2}$ to $\W^{\kappa-2,2}$.
\end{lemme}

We need the following immediate lemma.
\begin{lemme}\label[lemme]{Lem:BoundedFourierMult}
    Let $s\in\R$ and $\phi:\R^2\to\R$ such that
    $$\exists C>0,\; \exists m\in\R,\; \forall x\in\R,\; |\phi(x)|\leqslant C\langle x\rangle^m. $$
    Then $\phi\left(\complexI\nabla\right)\in\mathcal{L}(\mathrm{H}^{s+m}(\R^2),\mathrm{H}^{s}(\R^2))$ and $\left|\phi\left(\complexI\nabla\right)\right|_{\mathcal{L}(\mathrm{H}^{s+m},\mathrm{H}^s)}\leqslant|\langle x\rangle^{-m}\phi|_{\L^\infty}.$
\end{lemme}

\begin{proof}[Proof of \cref{Lem:ReguEllB}.]

    First, for $v\in\W^{1,2}$, 
    $$-Bv = -Lv- VY\rho^2 v - Z\rho^2 v + \delta_\Xi \rho^2 v,$$
    thus \cref{Lem:ProductRule,Cor:ProductRuleNegPosReg} give
    \begin{align*}
        \forall v\in\W^{2-\kappa,2}\subset\W^{1,2},\; |Bv|_{\W^{-\kappa,2}} &\lesssim |Lv|_{\W^{-\kappa,2}} + |Z|_{\W^{-\kappa,\infty}}|\rho^2|_{W^{\kappa,\infty}}|v|_{\W^{\kappa,2}}\\
        &+ |xY|_{\W^{-\kappa,\infty}}|\rho^2|_{W^{\kappa,\infty}}|xv|_{\W^{\kappa,2}} + \delta_\Xi|\rho^2|_{\L^\infty}|v|_{\L^{2}}\\
        &\lesssim_\Xi |Lv|_{\W^{-\kappa,2}} + |v|_{\W^{1+\kappa,2}},
    \end{align*}
    using $xY,Z\in\W^{-\kappa,\infty}$ by \cref{Lem:RegXI}. Moreover, \cref{Lem:CaracW-sq} implies $\W^{-\kappa,2}=\mathrm{H}^{-\kappa}+\L^2_{-\kappa}.$ For $v\in\mathcal{S}(\R^2)$,
    $$\left||x|^2\rho^2 v\right|_{\L^2_{-\kappa}}\lesssim |\rho^2|_{\L^\infty}|v|_{\L^2_{(2-\kappa)}}$$
    and, using \cref{Lem:BoundedFourierMult,Lem:ProductRule},
    $$|-\nabla\cdot(\rho^2\nabla v)|_{\mathrm{H}^{-\kappa}}\lesssim |\rho^2|_{\mathrm{H}^{1-\kappa,\infty}}|v|_{\mathrm{H}^{2-\kappa}}.$$
    Thus, for any $v\in\W^{2-\kappa,2}$,
    $$|-Lv|_{\W^{-\kappa,2}}\lesssim |v|_{\W^{2-\kappa,2}}.$$
    Hence, by density, $B\in\mathcal{L}(\W^{2-\kappa,2},\W^{-\kappa,2})$ and by duality, its adjoint exists and verifies $B'\in\mathcal{L}(\W^{\kappa,2},\W^{\kappa-2,2})$. Remark $B'_{|\W^{1,2}}=B$, thus $B'$ extends $B$ to an element of $\mathcal{L}(\W^{\kappa,2},\W^{\kappa-2,2})$. We keep the notation $B'$ to emphasize when we work in $\mathcal{L}(\W^{\kappa,2},\W^{\kappa-2,2})$.\\

    As $\kappa<\frac{1}{2}$, an approximation method using \cref{Def:EnhancedNoise} allows to show
    $$\forall v\in\W^{2-\kappa,2},\; Bv = \left(Hv+2\nabla Y\cdot \nabla v + xY\cdot xv + Zv - \delta_\Xi v\right)\rho^2.$$
    Thus, \cref{Lem:ProductRule,Cor:ProductRuleNegPosReg,Prop:action_d/dx,Cor:action_V,Lem:RegXI} imply
    $$\forall v\in\W^{2-\kappa,2},\; \left|Bv-\rho^2Hv\right|_{\W^{-\kappa,2}} \lesssim_\Xi |v|_{\W^{1+\kappa,2}}$$
    as $\nabla Y, xY, Z\in\W^{-\kappa,\infty}$, $\rho\in \mathrm{H}^{1-\kappa,\infty}$ and $\kappa\in(0,\frac{1}{2})$. In particular, as $1+\kappa<2-\kappa$, \cref{Cor:ProductRuleNegPosReg}, interpolation and Young's inequality imply
    $$\forall v\in\W^{2-\kappa,2},\; |v|_{\W^{2-\kappa,2}}\lesssim_{\Xi,\kappa,V}|Bv|_{\W^{-\kappa,2}}.$$
    Theorem II.20. in \cite{BrezisAnalyseFonctionelle} implies $B'$ is surjective and $B$ is injective and has closed range. Let $v\in\W^{\kappa,2}$ such that $B'v=0$. As $B'$ extends $B$ and $v\in\L^2(\R^2)$, then $v\in \mathrm{D}(-B)\subset\W^{1,2}$. As \cref{Prop:Quasi-coercivity-a} implies $B$ is injective on $\W^{1,2}$, $B'$ is injective on $\W^{\kappa,2}.$ Corollary II.17 (ii) in \cite{BrezisAnalyseFonctionelle} then implies $B$ has dense range in $\W^{-\kappa,2}$. Hence, $B\in\mathcal{L}(\W^{2-\kappa,2},\W^{-\kappa,2})$ and $B'\in\mathcal{L}(\W^{\kappa,2},\W^{\kappa-2,2})$ are isomorphism.
\end{proof}

As a direct corollary, we deduce that the complex interpolation space $(\W^{1,2},\mathrm{D}(-B))_{\kappa}$ is exactly $\W^{2-\kappa,2}$.

\begin{cor}\label[cor]{Cor:InterpolationW12D(-B)}
    We have $(\W^{1,2},\mathrm{D}(-B))_{\kappa}=\W^{2-\kappa,2}$ with equivalent norms. In particular, $(\mathcal{D}^{1,2},\mathrm{D}(-A))_{\kappa}=\mathcal{D}^{2-\kappa,2}$.
\end{cor}

This result motivated the notation $\mathcal{D}^{2,2}=\mathrm{D}(-A)$. We can now prove \cref{Prop:CaracDs2}.

\begin{proof}[Proof of \cref{Prop:CaracDs2}.]

    First, by Theorem 6.2.6 in \cite{KatoPerturbationBook}, $\mathcal{D}^{1,2} = D\left(\sqrt{-A}\right)$ with equivalent norm. We now prove $\mathcal{D}^{s,2} = \mathrm{D}\left((-A)^{s/2}\right)$ by interpolation, using \cref{Cor:InterpolationW12D(-B)}. By self-adjointness, for any $s\in\R$, $\mathrm{D}\left((-A)^{s/2}\right)'=\mathrm{D}\left((-A)^{-s/2}\right)$. Let $s_0,s_1\in\R$ and $\theta\in(0,1)$, define $s_\theta = \theta s_1 + (1-\theta) s_0$.
    Then, by Stein interpolation \cite{SteinInterpolation}, as $(-A)^{\mathrm{i}t}\in\mathcal{L}(\L^2(\R^2))$ for any $t\in\R$, it follows 
    \begin{equation}
        (-A)^{s_\theta/2}\in\mathcal{L}\left(\left(\mathrm{D}\left((-A)^{s_0/2}\right),\mathrm{D}\left((-A)^{s_1/2}\right)\right)_{\theta},\L^2(\R^2)\right). \label{Eq:InterpolDomain+s}
    \end{equation}
    As it holds for any $s_0,s_1\in\R$ and any $\theta\in(0,1)$, we deduce 
    $$(-A)^{-s_\theta/2}\in\mathcal{L}\left(\left(\mathrm{D}\left((-A)^{-s_0/2}\right),\mathrm{D}\left((-A)^{-s_1/2}\right)\right)_{\theta},\L^2(\R^2)\right)$$
    and by duality
    \begin{equation}
        (-A)^{-s_\theta/2}\in\mathcal{L}\left(\L^2(\R^2),\left(\mathrm{D}\left((-A)^{s_0/2}\right),\mathrm{D}\left((-A)^{s_1/2}\right)\right)_{\theta}\right). \label{Eq:InterpolDomain-s}
    \end{equation}
    Putting \cref{Eq:InterpolDomain+s,Eq:InterpolDomain-s} together, we get
    $$\left(\mathrm{D}\left((-A)^{s_0/2}\right),\mathrm{D}\left((-A)^{s_1/2}\right)\right)_{\theta}=\mathrm{D}\left((-A)^{s_\theta/2}\right)$$
    with equivalent norm. In particular, in view of \cref{Cor:InterpolationW12D(-B)}, we have $\mathcal{D}^{2-\kappa,2}=\mathrm{D}\left((-A)^{\frac{2-\kappa}{2}}\right)$. For any $s\in(1,2-\kappa)$, let $\theta = \frac{s-1}{1-\kappa}\in(0,1)$, then
    $$\mathcal{D}^{s,2}=\left((\mathcal{D}^{1,2},\mathcal{D}^{2-\kappa,2}\right)_{\theta}= \left(\mathrm{D}\left((-A)^{1/2}\right),\mathrm{D}\left((-A)^{\frac{2-\kappa}{2}}\right)\right)_{\theta}= \mathrm{D}\left((-A)^{s/2}\right),$$
    with equivalent norm. Similarly, for $s\in(0,1)$,
    $$\mathcal{D}^{s,2} =\left(\L^2(\R^2),\mathcal{D}^{1,2}\right)_{s}= \left(\L^2(\R^2),\mathrm{D}\left((-A)^{1/2}\right)\right)_{s}= \mathrm{D}\left((-A)^{s/2}\right),$$
    with equivalent norm. The case $s\in[\kappa-2,0)$ follows by duality. 
\end{proof}

\section{Anderson-Gross-Pitaevskii equation}

\subsection{The linear equation}

As $-A$ is a positive self-adjoint operator on $\L^2(\R^2)$, it is associated, by Stone's theorem \cite{StoneOneParameterGroup}, to a unitary group of operator $(\e^{itA})_{t\in\R}$, called the propagator, such that for any $u^0\in\L^2(\R^2)$, $u=\e^{itA}u^0\in\mathcal{C}(\R,\L^2(\R^2))$ is the unique solution to
\begin{equation}{~}
    \begin{cases}\mathrm{i}\p_t u+ Au=0\\ u(0)=u^0\end{cases} \label{Eq:2D-Lin}
\end{equation}
in $\mathcal{C}(\R,\L^2(\R^2))\cap\mathcal{C}^1(\R,\mathcal{D}^{-2,2})$. Using the exponential transform, one has formally
$$A(\rho v) = \rho\left(Hv+2\nabla Y\cdot \nabla v + xY\cdot xv + Zv - \delta_\Xi v\right) = \rho^{-1}Bv.$$
Using an approximation procedure, one can show this formula makes sense for any $u=\rho v$ with $v\in\W^{1+\kappa,2}$. Thus, we define
\begin{equation}\label{Eq:DefAtilde}
    \Tilde{A} v = \rho^{-1}A(\rho v) = Hv+ 2\nabla Y\cdot \nabla v + xY\cdot xv + Zv - \delta_\Xi v
\end{equation}
on $\mathrm{D}(-\Tilde{A})=\rho^{-1}\mathrm{D}(-A)=\mathrm{D}(-B)\subset\W^{1+\kappa,2}$ by \cref{Lem:NormEquivDomA/B,Lem:ReguEllB}, as $\kappa<\frac{1}{2}$. The operator $\Tilde{A}$ is a closed operator because $B$ is closed, but it is no longer self-adjoint on $\L^2(\R^2)$. Moreover, the norm induced by $-B$ and by $-\Tilde{A}$ on $\mathrm{D}(-B)=\mathrm{D}(-\Tilde{A})$ are equivalent.\\

Remark that, by definition, $u\in\mathcal{C}(I,\L^2(\R^2))$ solves \cref{Eq:2D-Lin} if and only if $v=\rho^{-1}u$ solves
\begin{equation}{~}
    \begin{cases}\mathrm{i}\p_t v+ \Tilde{A}v=0\\ v(0)=v^0=\rho^{-1}u^0.\end{cases} \label{Eq:2D-Lin-Transformed}
\end{equation}
Moreover, the propagator associated to $\Tilde{A}$ verify
\begin{equation}\label{Eq:PropagTildeA}
    \forall t\in\R,\; \e^{it\Tilde{A}}=\rho^{-1}\e^{itA}\rho.
\end{equation}
Thus, one can study \cref{Eq:2D-Lin} through \cref{Eq:2D-Lin-Transformed} and vice versa.

\subsection{Strichartz inequalities}\label{Sub:StrichartzProof}

For $v\in \mathrm{D}(-\Tilde{A})$, it holds $v\in\W^{2-\kappa,2}$. Moreover, \cref{Lem:RegXI} implies $Y\in\mathcal{C}^{1-\kappa}$ and $\nabla Y\in\mathcal{C}^{-\kappa}$. One can use a paracontrolled approach to control the "worst regularity term" $2\nabla Y\cdot \nabla v + Zv$, as in \cite{allez2015continuous,mouzard2021weyl,mouzard2022strichartz}. Recall
$$\Tilde{A}v = Hv + 2\nabla Y\cdot \nabla v + xY\cdot xv + Zv -\delta_\Xi v.$$
By \cref{Lem:RegXI}, $|x|^{1-\kappa}Y\in\L^\infty(\R^2)$ and thus $xY\cdot xv\in\L^2(\R^2)$ for any $v\in\W^{1+\kappa,2}$. Recall \cref{Cor:InterpolationW12D(-B)} imply $\mathrm{D}(-\Tilde{A})\subset\W^{2-\kappa,2}\subset\W^{1+\kappa+,2}$ as $\kappa<\frac{1}{2}$. Remark $2\nabla Y\cdot \nabla v + Zv$ cannot be in $\L^2(\R^2)$ as $\nabla Y$ and $Z$ are in general not functions. Using Bony's decomposition \cref{Eq:BonyDecomposition}, we get
$$2\nabla Y\cdot \nabla v + Zv = P_{\nabla v}(2\nabla Y) + P_{2\nabla Y}\nabla v + R(2\nabla Y,\nabla v) + P_v Z + P_Z v + R(Z,v)$$
and \cref{Sobolev-embeddings,Prop:ContinuityParaprodutcs} together with \cref{Rem:Eps=0ParaproductWs2} imply $Hv$, $P_{\nabla v}(2\nabla Y)$ and $P_v Z$ are the only singular terms of $\Tilde{A}v$ when $v\in \W^{2-\kappa,2}\subset \mathrm{H}^{2-\kappa}$, and all have the same regularity $\W^{-\kappa,2}$, by \cref{Lem:CaracW-sq}. Hence, for an element $v$ of the domain $\mathrm{D}(-\Tilde{A})\subset\W^{2-\kappa,2}$, $Hv+P_{\nabla v}(2\nabla Y)+P_v Z\in\L^2(\R^2)$. Thus if $v\in \mathrm{D}(-\Tilde{A})$, $v$ should have a particular structure such that the singularity of $Hv$ compensates exactly that of $P_{\nabla v}(2\nabla Y)+P_v Z$. This is where the idea of paracontrolled functions appears. We are looking for functions in the domain of $\Tilde{A}$ satisfying the paracontrolled ansatz
\begin{equation}
    v=v^\sharp+P_{\nabla v}X_1+P_{v}X_2,\label{Eq:ParacontrolledAnsatz}
\end{equation}
for $v^\sharp\in \mathrm{D}(-H) = \W^{2,2}$ and functions $X_1,X_2\in\mathcal{C}^{2-\kappa}$ to determine so that $HP_{\nabla v}X_1$ and $HP_{v}X_2$ compensates the singular parts of $2P_{\nabla v}\nabla Y$ and $P_vZ$ respectively (see for example \cite{MouzardMagnetic} for a similar ansatz in the context of white noise magnetic field on the torus). Inserting the paracontrolled ansatz \cref{Eq:ParacontrolledAnsatz} in \cref{Eq:DefAtilde} and applying Bony's decomposition \cref{Eq:BonyDecomposition}, lead to
\begin{align*}
    \Tilde{A}v &= Hv^\sharp + [H,P_{\nabla v}]X_1 + P_{\nabla v}\left(HX_1\right) + [H,P_{v}]X_2 + P_{v}\left(HX_2\right) -\delta_\Xi v\\
    &+ xY\cdot xv + 2\left(P_{\nabla v}\nabla Y + P_{\nabla Y}\nabla v + R(\nabla Y,\nabla v)\right) + P_{v}Z + P_{Z}v + R(v,Z)\\
    &= Hv^\sharp + [H,P_{\nabla v}]X_1 + [H,P_{v}]X_2 +P_{\nabla v}\left(HX_1+2\nabla Y\right)+ P_{v}\left(HX_2+Z\right) -\delta_\Xi v\\
    &+ xY\cdot xv + P_{2\nabla Y}\nabla v + R(2\nabla Y,\nabla v) + P_{Z}v + R(v,Z).
\end{align*}
We may choose $X_1 = 2(-H)^{-1}\nabla Y\in\W^{2-s,q}\subset\mathcal{C}^{2-\kappa}$ and $X_2 = (-H)^{-1} Z\in\W^{2-s,q}\subset\mathcal{C}^{2-\kappa}$ in order to obtain cancellations. A natural question is then, for a fixed $v^\sharp\in\W^{2,2}$, does-it exist a $v\in D(\Tilde{A}$) satisfying \cref{Eq:ParacontrolledAnsatz} ?
We stress out that, in general, the map $v\mapsto v^\sharp = v-P_{\nabla v}X_1-P_{v}X_2$ may not be surjective. Thus we use a truncated paracontrolled ansatz 
\begin{equation}
    v=v^\sharp+P^N_{\nabla v}X_1+P^N_{v}X_2,\label{Eq:TruncatedParacontrolledAnsatz}
\end{equation}
for some large $N$, where $P^N$ is a truncated paraproduct defined by \cref{Eq:DefTruncatedParaprod}. This leads to a new paracontrolled formula for $\Tilde{A}v$, namely
\begin{equation}\label{Eq:ParacontrolledTildeAN}
    \begin{split}
        \Tilde{A}v &= Hv^\sharp + [H,P_{\nabla v}^N]X_1 + [H,P_{v}^N]X_2 + xY\cdot xv -\delta_\Xi v + P_{2\nabla Y}\nabla v\\
        &+ R(2\nabla Y,\nabla v) + P_{Z}v + R(v,Z) 
        + 2P_{\nabla v}(\Delta_{\leqslant N}\nabla Y) + P_{v}(\Delta_{\leqslant N}Z),
    \end{split}
\end{equation}
where appears the term $2P_{\nabla v}(\Delta_{\leqslant N}\nabla Y) + P_{v}(\Delta_{\leqslant N}Z)$ involving only low frequencies of the singular terms $\nabla Y$ and $Z$, thus being smooth. Let us stress out that $v^\sharp$ depends implicitly of $N$. Using \cref{Lem:EstimDeltaNLqw,Cor:EstimDeltaN}, the following lemma shows we can invert the map $v\mapsto v-P_{\nabla v}^N X_1-P_{v}^N X_2$ for $N$ large enough.

\begin{lemme}\label[lemme]{Lem:InversibilityGamma}
    Let $\sigma\in[0,2-\kappa)$ and $p\in[2,+\infty)$, there exists $N_{\sigma,p}=N(\Xi,V,\sigma,p)\in\N$ such that for any $N\geqslant N_{\sigma,p}$, the map 
    $$\Phi_N:v\in\W^{\sigma,p}\mapsto v-P_{\nabla v}^NX_1-P_{v}^NX_2\in\W^{\sigma,p}$$
    is invertible of continuous inverse.
\end{lemme}

\begin{proof}
    As $\Phi_N$ is a perturbation of the identity, by Neumann series theorem, it is sufficient to show
    $$\left|P_{\nabla v}^NX_1+P_{v}^NX_2\right|_{\W^{\sigma,p}}=\left|P_{\nabla v}\left(\Delta_{>N}X_1\right)+P_{v}\left(\Delta_{>N}X_2\right)\right|_{\W^{\sigma,p}}\leqslant C_N |v|_{\W^{\sigma,p}}$$
    with $C_N<1$ for $N$ large enough. Let $\eps>0$ small enough. For $\sigma = 0$, \cref{Prop:ContinuityParaprodutcs} imply
    $$\left|P_{\nabla v}\left(\Delta_{>N}X_1\right)+P_{v}\left(\Delta_{>N}X_2\right)\right|_{\L^2}\lesssim |\nabla v|_{\mathrm{H}^{-1,p}}|\Delta_{>N} X_1|_{\mathcal{C}^{1+\eps}}+|v|_{\mathrm{H}^{-1,p}}|\Delta_{>N} X_2|_{\mathcal{C}^{1+\eps}}.$$
    \Cref{Cor:EstimDeltaN} then ensures 
    $$\left|P_{\nabla v}\left(\Delta_{>N}X_1\right)+P_{v}\left(\Delta_{>N}X_2\right)\right|_{\L^2}\lesssim 2^{-(1-\kappa-\eps)N}\left(|\Delta_{>N} X_1|_{\mathcal{C}^{2-\kappa}}+|\Delta_{>N} X_2|_{\mathcal{C}^{2-\kappa}}\right)|v|_{\L^p}.$$
    Now assume $\sigma\in(1,2-\kappa)$ and let $\eps\in(0,2-\kappa-\sigma)$. \Cref{Cor:ContinuityParaprodutcsWsp} implies
    $$\left|P_{v}^NX_2\right|_{\W^{\sigma,p}}\lesssim | v|_{\W^{\sigma,p}}|\Delta_{>N} X_2|_{\mathcal{C}^{\sigma+\eps}}$$
    and
    $$\left|P_{\nabla v}^NX_1\right|_{\W^{\sigma,p}}\lesssim |\nabla v|_{\W^{\sigma-1,p}}\left(|\Delta_{>N} X_1|_{\mathcal{C}^{\sigma+\eps}}+|\Delta_{>N} X_1|_{\mathcal{C}^\eps_{1}}\right).$$
    Remark \cref{Cor:action_V,Prop:BesovEmbedding} imply $xX_1\in\W^{1-s,q}\subset \mathcal{C}^{1-\kappa}\subset \mathcal{C}^{\eps}$. Hence, by \cref{Prop:action_d/dx,Cor:EstimDeltaN}, there exists $\delta>0$ such that
    \begin{equation}\label{Eq:InversionPhiNWsigma2}
        \left|P_{\nabla v}\left(\Delta_{>N}X_1\right)+P_{v}\left(\Delta_{>N}X_2\right)\right|_{\W^{\sigma,p}}\lesssim C(\Xi) 2^{-\delta N}|v|_{\W^{\sigma,p}},
    \end{equation}
    where
    $$C(\Xi) = |\Delta_{>N} X_1|_{\mathcal{C}^{2-\kappa}}+|\Delta_{>N} X_1|_{\mathcal{C}^{1-\kappa}_{1}}+|\Delta_{>N} X_2|_{\mathcal{C}^{2-\kappa}} .$$
    Finally, by interpolation, for any $\sigma\in[0,2-\kappa)$, there exists $\delta>0$ such that \cref{Eq:InversionPhiNWsigma2} holds.
    
\end{proof}

In the following, we choose $N$ large enough for our computations. We stress out that this $N$ may change when needed, leading to a family of paracontrolled ansatz
$$v = v_N^\sharp + P^N_{\nabla v}X_1 + P_v^N X_2.$$ Using this paracontrolled ansatz, one can write
$$\Tilde{A}v = Hv_N^\sharp + R_N(v)$$
where $R_N(v)$ is given by
\begin{equation}
    \begin{split}
        R_N(v) &= [H,P_{\nabla v}^N]X_1 + [H,P_{v}^N]X_2 + xY\cdot xv + P_{2\nabla Y}\nabla v + R(2\nabla Y,\nabla v) \\
        &+ P_{Z}v + R(v,Z) + 2P_{\nabla v}(\Delta_{\leqslant N}\nabla Y) + P_{v}(\Delta_{\leqslant N}Z) -\delta_\Xi v
    \end{split}\label{Eq:RNv1}
\end{equation}
according to \cref{Eq:ParacontrolledTildeAN}. Remark $|x|^2X_1 = \Delta X_1 +2\nabla Y$. By \cref{Lem:RegXI}, $\nabla Y\in\mathcal{C}^{-\kappa}$ and moreover, $X_1\in\W^{2-s,q}\subset\mathcal{C}^{2-\kappa}$, thus, we deduce $|x|^2X_1 \in\mathcal{C}^{-\kappa}$. Similarly, $X_2\in\W^{2-s,q}\subset\mathcal{C}^{2-\kappa}$ but $-HX_2=Z\in\W^{-s,q}$ may not be in $\mathcal{C}^{-\kappa}$ due to \cref{Lem:CaracW-sq}. Instead, one can use \cref{Lem:CaracW-sq} to show there exists $Z_r\in W^{-s,q}\subset\mathcal{C}^{-\kappa}$ and $Z_l\in\L^q(\R^2)$ such that $Z = Z_r+|x|^sZ_l$. This decomposition of $Z$ as a sum of a term lacking regularity $Z_r$ and a term lacking of localization $|x|^sZ_l$ is somehow similar to the decomposition used in \cite{GubinelliHofmanova}. It implies $$X_2 = (-H)^{-1}Z_r + (-H)^{-1}|x|^sZ_l.$$ Remark $Z_r\in\mathcal{C}^{-\kappa}$ and $(-H)^{-1}Z_r\in\W^{2-s,q}\subset\mathcal{C}^{2-\kappa}$, thus $(-H)^{-1}Z_r\in\mathcal{C}^{-\kappa}_2$. Now using $\langle x\rangle^2(-H)^{-1}|x|^sZ_l = (1+\Delta)(-H)^{-1}|x|^sZ_l + |x|^sZ_l$ and $(-H)^{-1}|x|^sZ_l\in\W^{2-s,q}\subset\mathcal{C}^{2-\kappa}$, it follows $X_2$ can be decomposed as the sum of a term in $\mathcal{C}^{-\kappa}_{2}$ and a term in $\mathcal{C}^{-\kappa+s}_{2-s}$. Namely, if we define
\begin{equation}\label{Eq:DefX2rl}
    X_{2,r}=(-H)^{-1}Z_r + \langle x\rangle^{-2}(1+\Delta)(-H)^{-1}|x|^sZ_l\in\mathcal{C}^{-\kappa}_{2} \text{ and } X_{2,l} = |x|^{s-2}Z_l \in\mathcal{C}^{-\kappa+s}_{2-s},
\end{equation}
using the Besov embedding $\L^q(\R^2)\subset\mathcal{C}^{-\kappa+s}$ as $\kappa-s>\frac{2}{q}$, then $X_2 = X_{2,r}+X_{2,l}$. From this, we obtain a paracontrolled formulation on the domain $\mathrm{D}(\Tilde{A})$ and thus a complete caracterization of $\mathrm{D}(A)$ thanks to \cref{Lem:NormEquivDomA/B}.

\begin{cor}\label[cor]{Cor:ParacontrolD(-TildeA)}
    The vector space $\Phi_N^{-1}\W^{2,2}$ does not depend on $N$ and is equal to
    $\mathrm{D}(-\Tilde{A})=\left\{v\in\W^{2-\kappa-,2},\; v-P_{\nabla v}X_1-P_{v}X_2\in\W^{2,2}\right\}.$ In particular, for any $v$ in $\mathrm{D}(-\Tilde{A})$, $\langle x\rangle^2 v\in\L^2(\R^2)$.
\end{cor}

\begin{proof}
    Let $N\in\N$, it holds
    $$\Phi_N^{-1}\W^{2,2}=\left\{v\in\W^{2-\kappa-,2},\; v-P_{\nabla v}^N X_1-P_{v}^N X_2\in\W^{2,2}\right\}.$$
    In order to show
    $$\Phi_N^{-1}\W^{2,2}=\left\{v\in\W^{2-\kappa-,2},\; v-P_{\nabla v} X_1-P_{v} X_2\in\W^{2,2}\right\},$$
    it is sufficient to show $P_{\nabla v}(\Delta_{\leqslant N}X_1) + P_{v}(\Delta_{\leqslant N}X_2)\in\W^{2,2}$. Let $\eps\in(0,1-\kappa)$, \cref{Prop:ContinuityParaprodutcs,Rem:Eps=0ParaproductWs2} imply
    $$|P_{\nabla v}(\Delta_{\leqslant N}X_1)|_{\L^2_2}\lesssim |\nabla v|_{\mathrm{H}^{-\eps}}|\Delta_{\leqslant N}X_1|_{\mathcal{C}^\eps_2}\lesssim |v|_{\mathrm{H}^{1}}|\Delta_{\leqslant N}X_1|_{\mathcal{C}^{\eps}_{2}}$$ 
    and 
    \begin{align*}
        |P_{v}(\Delta_{\leqslant N}X_2)|_{\L^2_2} &\lesssim | v|_{\mathrm{H}^{-\eps}}|\Delta_{\leqslant N}X_{2,r}|_{\mathcal{C}^{\eps}_{2}}+| v|_{\mathrm{H}^{-\eps}_{s}}|\Delta_{\leqslant N}X_{2,l}|_{\mathcal{C}^\eps_{(2-s)}}\\
        &\lesssim\left(|\Delta_{\leqslant N}X_{2,r}|_{\mathcal{C}^{\eps}_{2}}+|\Delta_{\leqslant N}X_{2,l}|_{\mathcal{C}^\eps_{(2-s)}}\right)|v|_{\W^{1,2}},
    \end{align*}
    using \cref{Cor:action_V} and $s<\kappa<1$. As $X_1,X_{2,r}\in\mathcal{C}^{-\kappa}_{2}$ and $X_{2,l}\in\mathcal{C}^{-\kappa+s}_{2-s}$ (see \cref{Eq:DefX2rl}), \cref{Cor:EstimDeltaN} implies
    $$|\Delta_{\leqslant N}X_1|_{\mathcal{C}^\eps_2}+|\Delta_{\leqslant N}X_{2,r}|_{\mathcal{C}^{\eps}_{2}}+|\Delta_{\leqslant N}X_{2,r}|_{\mathcal{C}^\eps_{(2-s)}}<+\infty,$$
    thus 
    $$|P_{\nabla v}(\Delta_{\leqslant N}X_1) + P_v (\Delta_{\leqslant N}X_2)|_{\L^2_2}\lesssim_{N,\Xi,V,\eps}|v|_{\W^{1,2}}.$$
    Moreover, $P_{\nabla v}(\Delta_{\leqslant N}X_1) + P_v (\Delta_{\leqslant N}X_2)\in \mathrm{H}^2$ as $v,\nabla v\in\W^{1-\kappa-,2}$ and $\Delta_{\leqslant N}X_1,\Delta_{\leqslant N}X_2\in\mathcal{C}^\alpha$ for any $\alpha>0$. Thus $P_{\nabla v}(\Delta_{\leqslant N}X_1) + P_v (\Delta_{\leqslant N}X_2)\in\W^{2,2}$ and 
    $$\Phi_N^{-1}\W^{2,2}=\left\{v\in\W^{2-\kappa-,2},\; v-P_{\nabla v} X_1-P_{v} X_2\in\W^{2,2}\right\}$$
    does not depend of N.
    
    We now fix $\sigma\in(1+\kappa,2-\kappa)$ and $N\in\N$ large enough so that $\Phi_N$ is invertible on $\L^2(\R^2)$ and $\W^{\sigma,2}$, by \cref{Lem:InversibilityGamma}. Let $v^\sharp\in\W^{2,2}\subset\W^{\sigma,2}$ and $v=\Phi_N^{-1} v^\sharp\in\W^{\sigma,2}$, it holds
    $$\Tilde{A}v = Hv^\sharp + R_N(v),$$
    with $R_N(v)$ given by \cref{Eq:RNv1}. For $v\in\W^{\sigma,2}$, as $\sigma>1+\kappa$, it holds
    $$[H,P_{\nabla v}^N]X_1 + [H,P_{v}^N]X_2\in\L^2(\R^2)$$
    by \cref{Lem:ComPuH},
    $$P_{2\nabla Y}\nabla v + R(2\nabla Y,\nabla v)+ P_{Z}v + R(v,Z) \in\L^2(\R^2)$$
    by \cref{Eq:DefX2rl,Cor:ContinuityParaprodutcsWsp},
    $$xY\cdot xv\in\L^2(\R^2)$$
    by \cref{Cor:action_V} as $|x|^{1-\kappa}Y\in\L^\infty(\R^2)$, and
    $$2P_{\nabla v}(\Delta_{\leqslant N}\nabla Y) + P_{v}(\Delta_{\leqslant N}Z)\in\L^2(\R^2)$$
    by \cref{Prop:ContinuityParaprodutcs,Lem:EstimDeltaNLqw,Cor:EstimDeltaN}. Thus $R_N(v)\in\L^2(\R^2)$ and $\Tilde{A}v\in\L^2(\R^2).$ It implies $\Phi_N^{-1}\W^{2,2}\subset \mathrm{D}(-\Tilde{A}).$

    Conversely, let $v\in \mathrm{D}(-\Tilde{A})$ and set $v^\sharp=\Phi_N v$. By \cref{Lem:NormEquivDomA/B,Lem:ReguEllB,Lem:InversibilityGamma}, $v^\sharp\in\W^{\sigma,2}$ and it holds
    $$ -Hv^\sharp = -\Tilde{A}v + R_N(v).$$
    As previously, $R_N(v)\in\L^2(\R^2)$ and $ -\Tilde{A}v\in\L^2(\R^2)$ as $v\in \mathrm{D}(-\Tilde{A})$. Hence $v^\sharp\in\W^{2,2}$ and $\Phi_N^{-1}\W^{2,2}=\mathrm{D}(-\Tilde{A}).$
    
\end{proof}

Thus, for any $v^\sharp\in\W^{2,2}$, it holds for $N$ large enough,
$$\Tilde{A}\Phi_N^{-1}v^\sharp-Hv^\sharp = R_N(\Phi_N^{-1}v^\sharp)$$
and $R_N(\Phi_N^{-1}v^\sharp)$ can be seen as a lower order perturbation. Estimating $R_N(v)$ in $\W^{\sigma,2}$ for $\sigma>0$ will be needed in order to deal with estimates of the form 
$$|\e^{it\wick{H+\xi}}|_{\L^p_T\W^{\alpha,r}_x}$$
with $\alpha>0$ in \cref{Th:Strichart(H+xi)}. However, when estimating $R_N(v)$ in $\W^{\sigma,2}$ with $\sigma\in(0,1-\kappa)$, the term $xY\cdot xv$ behaves badly. Namely
$$|xY\cdot xv|_{\W^{\sigma,2}}\lesssim ||x|^{1-\kappa-\sigma}Y|_{W^{\sigma,\infty}}||x|^{1+\kappa+\sigma}v|_{\W^{\sigma,2}} \lesssim |Y|_{\W^{1-s,q}}|v|_{\W^{1+\kappa+2\sigma,2}}$$
using \cref{Eq:Condqskappa,Cor:action_V}. As \cref{Lem:InversibilityGamma} controls $|v|_{\W^{\alpha,2}}$ by $|v^\sharp|_{\W^{\alpha,2}}$ only for $\alpha<2-\kappa$, we are limited to $1+\kappa+2\sigma<2-\kappa$, that is $\sigma<\frac{1}{2}-\kappa$. Thus, we will modify \cref{Eq:TruncatedParacontrolledAnsatz} by adding a new term $P^N_{xv}X_3$, in order to control the term $xY\cdot xv$. Remark $|x|^2$ is smooth thus $|x|^2Y\in \mathcal{C}^{1-\kappa}_{loc}$; on the other hand, \cref{Cor:action_V} implies $|x|^{1-s}Y\in\L^q(\R^2)$, recall $q$ and $s$ are defined in \cref{Eq:Condqskappa}, thus we only have $xY\in\L^q_{-s}$. Hence there is no lack of local regularity to compensate in $xY\cdot xv$, just a lack of integrability. Define $X_3 = (-H)^{-1}xY\in\W^{2-s,q}\subset\mathcal{C}^{2-\kappa}$. Then $\Delta X_3 \in\mathcal{C}^{-\kappa}$ and 
$$|x|^2X_3 = HX_3 + \Delta X_3 = xY + \Delta X_3.$$
Define 
\begin{equation}\label{Eq:DefX3rl}
    X_{3,r}=\langle x\rangle^{-2}(1+\Delta)X_3\in\mathcal{C}^{-\kappa}_{2} \text{ and } X_{3,l} = xY \in\mathcal{C}^{-\kappa+s}_{2-s},
\end{equation}
using once again the Besov embedding $\L^q(\R^2)\subset\mathcal{C}^{-\kappa+s}$. Then, $X_3 = X_{3,r}+X_{3,l}$. Adapting the proof of \cref{Lem:InversibilityGamma}, it is not difficult to prove the following lemma.
\begin{lemme}\label[lemme]{Lem:InversibilityGammaWsq}
    Let $\sigma\in[0,2-\kappa)$, $p\in[2,+\infty)$ there exists $N_{\sigma,p}=N(\Xi,\sigma,p)\in\N$ such that for any $N\geqslant N_{\sigma,p}$, the map 
    \begin{equation}\label{Eq:Gamma-1}
        \Tilde{\Phi}_N:v\in\W^{\sigma,p}\mapsto v-P_{\nabla v}^NX_1-P_{v}^NX_2-P^N_{xv}X_3\in\W^{\sigma,p}
    \end{equation}
    is invertible of continuous inverse.
\end{lemme}
For N large enough, we define $\Gamma_N = \Tilde{\Phi}_N^{-1}$. Then, for any $v_N^\sharp\in\W^{2,2}$, it holds
\begin{equation}
    \Tilde{A}\Gamma_N v_N^\sharp = Hv_N^\sharp + \mathcal{R}_N(\Gamma_N v_N^\sharp),
\end{equation}
where
\begin{equation}
    \begin{split}
        \mathcal{R}_N(v) &= \left[H,P_{\nabla v}^N\right]X_1 + \left[H,P_{v}^N\right]X_2 + \left[H,P_{xv}^N\right]X_3+ P_{2\nabla Y}\nabla v \\
        &+ P_{Z}v + P_{xY}xv + R(2\nabla Y,\nabla v) + R(v,Z) + R(xv,xY) \\
        &+ 2P_{\nabla v}(\Delta_{\leqslant N}\nabla Y) + P_{v}(\Delta_{\leqslant N}Z) + P_{xv}(\Delta_{\leqslant N}xY) -\delta_\Xi v.
    \end{split}\label{Eq:RNv2}
\end{equation}
We can now prove the following estimate on $\mathcal{R}_N$.

\begin{lemme}\label[lemme]{Lem:EstimReminderL2}
    Let $\sigma\in(0,1-\kappa)$ and $N\in\N$, there exists $C=C(N,\sigma,\Xi)>0$ such that
    $$\forall v\in\W^{1+\kappa+\sigma,2},\; |\mathcal{R}_N(v)|_{\W^{\sigma,2}}\leqslant C|v|_{\W^{1+\kappa+\sigma,2}}.$$
\end{lemme}

\begin{proof}

    We have to estimate the $\W^{\sigma,2}$-norm of $\mathcal{R}_N(v)$; by density, it is sufficient to prove the estimate for $v$ smooth enough. First, by \cref{Lem:ComPuH,Prop:action_d/dx,Cor:action_V}, it holds
    \begin{equation}
        \left|[H,P_{\nabla v}^N]X_1 + [H,P_{v}^N]X_2 +  \left[H,P_{xv}^N\right]X_3\right|_{\W^{\sigma,2}}\lesssim_{N,\sigma}  C_0(\Xi)|v|_{\W^{1+\sigma+\kappa,2}},\label{Eq:ProofReminderCom}
    \end{equation}
    where
    \begin{align*}
        C_0(\Xi) &= |X_{1}|_{\mathcal{C}^{2-\kappa}}+|X_{1}|_{\mathcal{C}^{-\kappa}_{2}} + |X_{2}|_{\mathcal{C}^{2-\kappa}} + |X_{2,r}|_{\mathcal{C}^{-\kappa}_{2}}\\
        &+ |X_{2,l}|_{\mathcal{C}^{-\kappa+s}_{2-s}}+ |X_{3}|_{\mathcal{C}^{2-\kappa}}+|X_{3,r}|_{\mathcal{C}^{-\kappa}_{2}} + |X_{3,l}|_{\mathcal{C}^{-\kappa+s}_{2-s}},
    \end{align*}
    $X_{2,l}$ and $X_{2,r}$ being defined in \cref{Eq:DefX2rl} and $X_{3,l}$ and $X_{3,r}$ being defined by \cref{Eq:DefX3rl}. Now, recall \cref{Lem:CaracW-sq} implies $Z = Z_r+|x|^s Z_l$ with $Z_r\in\mathcal{C}^{-\kappa}$ and $Z_l\in\L^q(\R^2)\subset\mathcal{C}^{s-\kappa}$, where $s$, $q$ and $\kappa$ are linked by \cref{Eq:Condqskappa}. By \cref{Cor:ContinuityParaprodutcsWsp,Rem:Eps=0ParaproductWs2,Prop:action_d/dx,Cor:action_V},
    \begin{align*}
        |P_{2\nabla Y}\nabla v|_{\W^{\sigma,2}} &\lesssim |\nabla Y|_{\mathcal{C}^{-\kappa}} |\nabla v|_{\W^{\sigma+\kappa,2}}&\lesssim_{\Xi,\sigma} |v|_{\W^{1+\sigma+\kappa,2}},\\
        |P_{Z_r}v|_{\W^{\sigma,2}} &\lesssim |Z_r|_{\mathcal{C}^{-\kappa}} |v|_{\W^{\sigma+\kappa,2}},\\
        |P_{|x|^{s}Z_l}v|_{\W^{\sigma,2}} &\lesssim |v|_{\W^{\sigma+\kappa,2}}||x|^{s}Z_l|_{\mathcal{C}^{s-\kappa}_{-s}}&\lesssim_{\Xi,\sigma} |v|_{\W^{\sigma+\kappa,2}},\\
        |P_{xY}xv|_{\W^{\sigma,2}} &\lesssim |xv|_{\W^{\sigma+\kappa,2}}|xY|_{\mathcal{C}^{s-\kappa}_{-s}}&\lesssim_{\Xi,\sigma} |v|_{\W^{1+\sigma+\kappa,2}}.
    \end{align*}
    This gives
    \begin{equation}
        \left|P_{2\nabla Y}\nabla v + P_{Z}v + P_{xY}xv\right|_{\W^{\sigma,2}}\lesssim_{N,\sigma,\Xi} |v|_{\W^{1+\sigma+\kappa,2}}.\label{Eq:ProofReminderParaprod}
    \end{equation}
    Similarly, as $\sigma>0$, by \cref{Cor:ContinuityParaprodutcsWsp,Rem:Eps=0ParaproductWs2,Prop:action_d/dx,Cor:action_V},
    \begin{equation}
        |R(2\nabla Y,\nabla v) + R(v,Z) + R(xv,xY)|_{\W^{\sigma,2}}\lesssim_{\Xi,\sigma} |v|_{\W^{1+\sigma+\kappa,2}}.\label{Eq:ProofReminderResonant}
    \end{equation}
    Moreover, using the decomposition $Z=Z_r+|x|^sZ_l$ and reasoning as previously, \cref{Cor:ContinuityParaprodutcsWsp,Rem:Eps=0ParaproductWs2,Prop:action_d/dx,Cor:action_V,Lem:EstimDeltaNLqw,Cor:EstimDeltaN} imply
    \begin{equation}
        |2P_{\nabla v}(\Delta_{\leqslant N}\nabla Y) + P_{v}(\Delta_{\leqslant N}Z) + P_{xv}(\Delta_{\leqslant N}xY)|_{\W^{\sigma,2}}\lesssim_{\Xi,\sigma,N} |v|_{\W^{1+\sigma+\kappa,2}}.\label{Eq:ProofReminderLowFrequencies}
    \end{equation}
    Hence, by \cref{Eq:RNv2,Eq:ProofReminderCom,Eq:ProofReminderParaprod,Eq:ProofReminderResonant,Eq:ProofReminderLowFrequencies}, it holds
    $$|\mathcal{R}_N(v)|_{\W^{\sigma,2}}\lesssim_{N,\sigma,\Xi} |v|_{\W^{1+\sigma+\kappa,2}}.$$

\end{proof}

For $N$ large enough, we define 
\begin{equation}\label{Eq:DefAsharp}
    A^\sharp = \Gamma_N^{-1}\Tilde{A}\Gamma_N = H + \mathcal{R}_N^\sharp,
\end{equation}
where
\begin{equation}\label{Eq:DefRsharp}
    \mathcal{R}_N^\sharp(v^\sharp) = \Gamma_N^{-1}\mathcal{R}_N(\Gamma_N v^\sharp) - P^N_{\nabla H v^\sharp} X_1 - P^N_{H v^\sharp} X_2 - P^N_{xH v^\sharp} X_3,
\end{equation}
by \cref{Eq:RNv2,Eq:Gamma-1}. We have the following estimate on $\mathcal{R}_N^\sharp$.

\begin{cor}\label[cor]{Cor:EstimReminderL2Sharp}
    Let $\sigma\in(0,1-2\kappa)$ and $N\geqslant N_{1+\kappa+\sigma,2}$, there exists $C=C(N,\sigma,\Xi)>0$ such that
    $$\forall v^\sharp\in\W^{1+\kappa+\sigma,2},\; |\mathcal{R}_N^\sharp(v^\sharp)|_{\W^{\sigma,2}}\leqslant C|v^\sharp|_{\W^{1+\kappa+\sigma,2}}.$$
\end{cor}

\begin{proof}

    We only have to prove
    $$\left|(1-\Gamma_N^{-1})Hv^\sharp\right|_{\W^{\sigma,2}}\lesssim\left|v^\sharp\right|_{\W^{1+\sigma+\kappa,2}}.$$
    For this, we show
    \begin{equation}\label{Eq:BoundGamma-1ReminderL2Sharp}
        \left|(1-\Gamma_N^{-1})f\right|_{\W^{\varsigma,2}}\lesssim\left|f\right|_{\W^{\varsigma-1+\kappa,2}}
    \end{equation}
    for any $\varsigma\in[0,2-\kappa)$. We show \cref{Eq:BoundGamma-1ReminderL2Sharp} by interpolation. Let $\varsigma=0$ and $f\in\W^{\kappa-1,2}$. By \cref{Cor:action_V,Lem:CaracW-sq}, there exists $f_r\in H^{1-\kappa}$ and $f_l\in\L^2_{\kappa-1}$ such that $f=f_r+f_l$.
    Then by \cref{Cor:ContinuityParaprodutcsWsp,Rem:Eps=0ParaproductWs2}, it holds
    \begin{align*}
        \left|P^N_{\nabla f}X_1\right|_{\L^2} &\lesssim |\nabla f_r|_{H^{2-\kappa}}|\Delta_{>N} X_1|_{\mathcal{C}^{2-\kappa}} + |\nabla f_l|_{H^{-1}_{\kappa-1}}|\Delta_{>N} X_1|_{\mathcal{C}^{1}_{1-\kappa}}\\
        &\lesssim \left(|f_r|_{H^{1-\kappa}}+|f_l|_{L^2_{\kappa-1}}\right)\left(|\Delta_{>N} X_1|_{\mathcal{C}^{2-\kappa}}+|\Delta_{>N} X_1|_{\mathcal{C}^{1}_{1-\kappa}}\right)\\
        &\lesssim \left(|\Delta_{>N} X_1|_{\mathcal{C}^{2-\kappa}}+|\Delta_{>N} X_1|_{\mathcal{C}^{1}_{1-\kappa}}\right)|f|_{\W^{\kappa-1,2}}
    \end{align*}
    and by \cref{Eq:Condqskappa}, $X_1\in\W^{2-s,q}$ so that $X_1\in\mathcal{C}^{2-\kappa}$ and by \cref{Cor:action_V,Prop:BesovEmbedding}, $\langle x\rangle^{1-\kappa}X_1\in\W^{1+\kappa-s}\subset\mathcal{C}^1$, thus
    $$|\Delta_{>N} X_1|_{\mathcal{C}^{2-\kappa}}+|\Delta_{>N} X_1|_{\mathcal{C}^{1}_{1-\kappa}}<+\infty.$$
    Similarly, one can show
    $$\left|P^N_{f}X_2\right|_{\L^2}+\left|P^N_{xf}X_3\right|_{\L^2}\lesssim |f|_{\W^{\kappa-1,2}}.$$
    Thus \cref{Eq:Gamma-1} implies \cref{Eq:BoundGamma-1ReminderL2Sharp} in the case $\varsigma=0$.\\

    Now let $\varsigma\in(1-\kappa,2-\kappa)$, $\eps=2-\kappa-\varsigma$ and $f\in\W^{\varsigma-1+\kappa}$. By \cref{Cor:ContinuityParaprodutcsWsp,Rem:Eps=0ParaproductWs2}, it holds
    \begin{align*}
        \left|P^N_{\nabla f}X_1\right|_{H^\varsigma} &\lesssim |\nabla f|_{H^{-\eps}}|\Delta_{>N} X_1|_{\mathcal{C}^{2-\kappa}}\\
        &\lesssim |\Delta_{>N} X_1|_{\mathcal{C}^{2-\kappa}}|f|_{H^{\varsigma-1+\kappa}}\\
        &\lesssim |\Delta_{>N} X_1|_{\mathcal{C}^{2-\kappa}}|f|_{\W^{\varsigma-1+\kappa,2}},
    \end{align*}
    as $\varsigma-1+\kappa>0$. Similarly,
    \begin{align*}
        \left|P^N_{\nabla f}X_1\right|_{L^2_{\varsigma}} &\lesssim |\nabla f|_{H^{-\eps}}|\Delta_{>N} X_1|_{\mathcal{C}^{\eps}_{\varsigma}}\\
        &\lesssim |\Delta_{>N} X_1|_{\mathcal{C}^{\eps}_{\varsigma}}|f|_{\W^{\varsigma-1+\kappa,2}}.
    \end{align*}
    By \cref{Eq:Condqskappa}, $X_1\in\W^{2-s,q}$ so that by \cref{Cor:action_V,Prop:BesovEmbedding}, $\langle x\rangle^{\varsigma}X_1\in\W^{2-s-\varsigma}=\W^{\kappa-s+\eps}\subset\mathcal{C}^\eps$, thus
    $$\left|P^N_{\nabla f}X_1\right|_{\W^{\varsigma,2}}\lesssim |f|_{\W^{\varsigma-1+\kappa,2}}.$$
    Similarly, one can show
    $$\left|P^N_{f}X_2\right|_{\W^{\varsigma,2}}+\left|P^N_{xf}X_3\right|_{\W^{\varsigma,2}}\lesssim |f|_{\W^{\varsigma-1+\kappa,2}}.$$
    Thus \cref{Eq:Gamma-1} implies \cref{Eq:BoundGamma-1ReminderL2Sharp} in the case $\varsigma\in(1-\kappa,2-\kappa)$. By interpolation, \cref{Eq:BoundGamma-1ReminderL2Sharp} holds for any $\varsigma\in[0,2-\kappa)$ and thus
    $$\left|(1-\Gamma_N^{-1})Hv^\sharp\right|_{\W^{\sigma,2}}\lesssim\left|v^\sharp\right|_{\W^{1+\sigma+\kappa,2}}.$$ The conclusion follows from \cref{Eq:RNv2,Eq:Gamma-1,Lem:EstimReminderL2,Lem:InversibilityGammaWsq}.
    
\end{proof}

As $\Gamma_N:\L^2(\R^2)\to\L^2(\R^2)$ is invertible for $N\geqslant N_{0,2}$, it is easy to show 
$v^\sharp=\Gamma_N v$ solves
\begin{equation}{~}
    \begin{cases}\mathrm{i}\p_t v^\sharp+ A^\sharp v^\sharp=0\\ v^\sharp(0)=v^\sharp_0=\Gamma^{-1} v^0\end{cases} \label{Eq:2D-Lin-Sharpened}
\end{equation}
if and only if $v$ solves \cref{Eq:2D-Lin-Transformed}. We can thus define the propagator $\left(\e^{\mathrm{i}tA^\sharp}\right)_{t\in\R}$ as
\begin{equation}\label{Eq:PropagAsharp}
    \forall t\in\R,\; \e^{\mathrm{i}tA^\sharp} = \Gamma^{-1}_N\e^{\mathrm{i}t\Tilde{A}}\Gamma_N.
\end{equation}
We now give simple lemmas whose proof follow those of Lemma 2.4 and Lemma 2.2 in \cite{mouzard2022strichartz}, respectively. \Cref{Lem:eitA-eitH} follows from Duhamel's formula.

\begin{lemme}\label[lemme]{Lem:eitA-eitH}
    Let $t,t_0\in\R$, for any $v^\sharp\in\W^{2,2}$, it holds
    $$\e^{i(t-t_0)A^\sharp}v^\sharp = \e^{i(t-t_0)H}v^\sharp+i\int_{t_0}^t \e^{i(t-s)H}(A^\sharp-H)\e^{i(s-t_0)A^\sharp}v^\sharp\d s.$$
\end{lemme}

\begin{lemme}\label[lemme]{Lem:Estim_eitAsharp}
    Let $\sigma\in[0,2]$, then there exists $C=C(\Xi,N,V)>0$ such that
    $$\forall t\in\R,\; \forall v\in\W^{\sigma,2},\; \left|\e^{itA^\sharp}v\right|_{\W^{\sigma,2}}\leqslant C\left|v\right|_{\W^{\sigma,2}}.$$
\end{lemme}

We can finally prove our Strichartz inequalities. We say $(p,r)$ is admissible if
\begin{equation}
   \frac{1}{p}+\frac{1}{r}=\frac{1}{2},\; p,r\geqslant 2 \text{ and } r<+\infty \label{Eq:CondStichartzPair}
\end{equation}

We recall Strichartz inequalities for the harmonic oscillator.

\begin{prop}{Strichartz inequality (see \cite{CarlesLSvsNLS,FujiwaraSubquadratic})}\label{Th:StrichartzDeterministe}

    Let $(p,r)$ be an admissible pair. Then, there exists a time $T>0$ and a constant $C=C(T,p)>0$ such that
    $$\forall\alpha\geqslant 0,\; \forall u\in\W^{\alpha,2},\; \left|\e^{itH}u\right|_{\L^p_{(-T,T)}\W^{\alpha,r}_x}\leqslant C |u|_{\W^{\alpha,2}}.$$
    
\end{prop}

We now state Strichartz inequalities for $A^\sharp$. \Cref{Th:Strichart(H+xi)} follows from \cref{Th:StrichartzAsharp} below because \cref{Eq:PropagTildeA,Eq:PropagAsharp,Prop:Quasi-coercivity-a} imply
$$\forall t\in\R,\; \e^{\mathrm{i}t\wick{H+\xi}} = \e^{\mathrm{i}t\delta_\Xi}\Upsilon\e^{\mathrm{i}t A^\sharp}\Upsilon^{-1}$$
with $\Upsilon$ defined by 
\begin{equation}\label{Eq:DefUpsilon}
    \Upsilon v^\sharp = \rho\Gamma_N v^{\sharp}
\end{equation}
for $N$ large enough such that $\Upsilon:\W^{\alpha,r}\to\mathcal{D}^{\alpha,r}$ and $\Upsilon:\W^{\alpha+\frac{1}{p}+\kappa+\eps,2}\to\mathcal{D}^{\alpha+\frac{1}{p}+\kappa+\eps,2}$ are invertible by \cref{Lem:InversibilityGammaWsq} (we refer to the statement of \cref{Th:Strichart(H+xi),Th:StrichartzAsharp} for the choice of spaces and the definition of the parameters and to \cref{Def:Dsq} for the definition of $\mathcal{D}^{s,p}-$spaces.). Thus, by \cref{Th:StrichartzAsharp,Lem:InversibilityGammaWsq}, for $\alpha'=\alpha+\frac{1}{p}+\kappa+\eps$,
\begin{align*}
    |\e^{\mathrm{i}t\wick{H+\xi}}u|_{\L^p_T\mathcal{D}^{\alpha,r}} &\leqslant |\Upsilon|_{\mathcal{L}(\mathcal{D}^{\alpha,r},\W^{\alpha,r})}|\e^{\mathrm{i}tA^\sharp}\Upsilon^{-1} u|_{\L^p_T\W^{\alpha,r}_x}\\
    &\leqslant |\Upsilon|_{\mathcal{L}(\mathcal{D}^{\alpha,r},\W^{\alpha,r})} |\e^{\mathrm{i}tA^\sharp}|_{\mathcal{L}(\W^{\alpha',2},\L^p_T\W^{\alpha,r}_x)} |\Upsilon^{-1} u|_{\W^{\alpha',2}}\\
    &\leqslant |\Upsilon|_{\mathcal{L}(\mathcal{D}^{\alpha,r},\W^{\alpha,r})} |\e^{\mathrm{i}tA^\sharp}|_{\mathcal{L}(\W^{\alpha',2},\L^p_T\W^{\alpha,r}_x)} |\Upsilon^{-1}|_{\mathcal{L}(\W^{\alpha',2},\mathcal{D}^{\alpha',2})} |u|_{\mathcal{D}^{\alpha',2}_\Xi}
\end{align*}
and \cref{Th:Strichart(H+xi)} follows. In particular, for any $\alpha\in[0,2-\kappa]$, \cref{Lem:Estim_eitAsharp} implies
\begin{equation}\label{Eq:UnifBoundExp(it:H+xi:)}
    |\e^{\mathrm{i}t\wick{H+\xi}}u|_{\L^\infty_T\mathcal{D}^{\alpha,2}_\Xi} \leqslant C|u|_{\mathcal{D}^{\alpha,2}_\Xi}.
\end{equation}

\begin{theo}{(Strichartz estimates for $A^\sharp$)}\label{Th:StrichartzAsharp}

    Let $\kappa\in\left(0,\frac{1}{2}\right)$, $\xi\in\W^{-1-\kappa,\infty}$ and $\Xi$ verifying \cref{Def:EnhancedNoise}. Let $(p,r)\in[2,+\infty]^2$ satisfying $\frac{1}{p}+\frac{1}{r}=\frac{1}{2}$ and $r<+\infty$. Let $\alpha\in[0,1-2\kappa)$ and $\eps>0$. Then, there exists a time $T>0$ and a constant $C=C(\Xi,N,q,\alpha,\eps)>0$ such that
    $$\forall u\in\W^{\alpha+\frac{1}{p}+\kappa+\eps,2},\; \left|\e^{\mathrm{i}tA^\sharp}v\right|_{\L^p_{(-T,T)}\W^{\alpha,r}_x}\leqslant C |v|_{\W^{\alpha+\frac{1}{p}+\kappa+\eps,2}}.$$
\end{theo}

\begin{proof}[Proof of \cref{Th:StrichartzAsharp}.]
    
    First remark $(p,r)$ is admissible with $p=+\infty$ if and only if $r=2$. Thus, the case $p=+\infty$ follows from \cref{Lem:Estim_eitAsharp}. Thus, assume $p<+\infty$. By density, it is sufficient to prove the claim for $v\in\W^{2,2}$. Let $T$ be given by \cref{Th:StrichartzDeterministe}. Let $M\in\N^*$ to be fixed later and for $l\in[\![0,M]\!]$ define $t_l = -T + l\frac{2T}{M}$. Let $(\psi_j)_{j\geqslant -1}$ be a dyadic partition of unity and let $\Psi_{j}=\psi_j(-H)$, it holds
    \begin{equation}
        \left|\e^{itA^\sharp}v\right|_{\L^p_{(-T,T)}\W^{\alpha,r}_x}\leqslant \sum_{i,j\geqslant -1} \left|\Psi_j\e^{itA^\sharp}\Psi_i v\right|_{\L^p_{(-T,T)}\W^{\alpha,r}_x}\label{Eq:LocalizationDoubleSumStrichartzProof}
    \end{equation}
    and
    \begin{equation}
        \left|\Psi_j\e^{itA^\sharp}\Psi_i v\right|_{\L^p_{(-T,T)}\W^{\alpha,r}_x}^p = \sum_{l=0}^{M-1}\left|\Psi_j\e^{itA^\sharp}\Psi_i v\right|_{\L^p_{(t_l,t_{l+1})}\W^{\alpha,r}_x}^p.\label{Eq:SmallTimeDivisionStrichartzProof}
    \end{equation}
    Using \cref{Lem:eitA-eitH}, it follows
    \begin{align*}
        \left|\Psi_j\e^{itA^\sharp}\Psi_i v\right|_{\L^p_{(t_l,t_{l+1})}\W^{\alpha,r}_x}^p &\lesssim \left|\Psi_j\e^{i(t-t_l)H}\e^{it_l A^\sharp}\Psi_i v\right|_{\L^p_{(t_l,t_{l+1})}\W^{\alpha,r}_x}^p \\
        &+ \left|\int_{t_l}^t \Psi_j\e^{i(t-s)H}(A^\sharp-H)\e^{isA^\sharp}\Psi_i v\d s\right|_{\L^p_{(t_l,t_{l+1})}\W^{\alpha,r}_x}^p.
    \end{align*}
    We estimate each term separately. Using \cref{Lem:Estim_eitAsharp,Th:StrichartzDeterministe}, it holds for any $\eta,\eta'>0$ with $\alpha+\eta+\eta'\leqslant 2$,
    \begin{align*}
         \left|\Psi_j\e^{i(t-t_l)H}\e^{it_l A^\sharp}\Psi_i v\right|_{\L^p_{(t_l,t_{l+1})}\W^{\alpha,r}_x}^p &= \left|\e^{i(t-t_l)H}\Psi_j\e^{it_l A^\sharp}\Psi_i v\right|_{\L^p_{(t_l,t_{l+1})}\W^{\alpha,r}_x}^p\\
         &\lesssim  \left|\Psi_j\e^{it_l A^\sharp}\Psi_i v\right|_{\W^{\alpha,2}_x}^p\\
         &\lesssim 2^{-jp\eta} \left|\e^{it_l A^\sharp}\Psi_i v\right|_{\W^{\alpha+\eta,2}}^p\\
         &\lesssim 2^{-jp\eta}2^{-ip\eta'} \left|\Psi_i v\right|_{\W^{\alpha+\eta+\eta',2}}^p.
    \end{align*}
    Similarly, for $\eps\in(0,1-2\kappa-\alpha')$ and $\eps'>0$  with $1+\alpha+\kappa+\eps+\frac{\eps'}{2}\leqslant 2-\kappa$, by \cref{Cor:EstimReminderL2Sharp,Lem:Estim_eitAsharp,Th:StrichartzDeterministe}, we have
    \begin{align*}
        &\left|\int_{t_l}^t \Psi_j\e^{i(t-s)H}(A^\sharp-H)\e^{isA^\sharp}\Psi_i v\d s\right|_{\L^p_{(t_l,t_{l+1})}\W^{\alpha,r}_x}^p\\
        &\lesssim \left(\int_{t_l}^{t_{l+1}} \left|\Psi_j(A^\sharp-H)\e^{isA^\sharp}\Psi_i v\right|_{\W^{\alpha,2}_x}\d s\right)^p\\
        &\lesssim 2^{-jp\eps}\left(\int_{t_l}^{t_{l+1}} \left|(A^\sharp-H)\e^{isA^\sharp}\Psi_i v\right|_{\W^{\alpha+\eps,2}}\d s\right)^p\\
        &\lesssim 2^{-jp\eps} \left(\int_{t_l}^{t_{l+1}} \left|\e^{isA^\sharp}\Psi_i v\right|_{\W^{1+\alpha+\kappa+\eps+\eps'/2,2}}\d s\right)^p\\
        &\lesssim M^{-p}2^{-jp\eps}2^{-ip\eps'/2} \left|\Psi_i v\right|_{\W^{1+\alpha+\kappa+\eps+\eps',2}}^p.
    \end{align*}
    Thus, it holds
    \begin{equation}
        \left|\Psi_j\e^{itA^\sharp}\Psi_i v\right|_{\L^p_{(-T,T)}\W^{\alpha,r}_x} \lesssim M^{1/p}2^{-j\eta}2^{-i\eta'} \left|\Psi_i v\right|_{\W^{\alpha+\eta+\eta',2}}+ M^{\frac{1}{p}-1}2^{-j\eps}2^{-i\eps'/2} \left|\Psi_i v\right|_{\W^{1+\alpha+\kappa+\eps+\eps',2}}.\label{Eq:EstimProofStrichartzPsijPsii}
    \end{equation}
    Now, we split the right hand side in \cref{Eq:LocalizationDoubleSumStrichartzProof} as 
    $$\sum_{j\geqslant -1}\sum_{i\leqslant j} \left|\Psi_j\e^{itA^\sharp}\Psi_i v\right|_{\L^p_{(-T,T)}\W^{\alpha,r}_x} + \sum_{i\geqslant -1}\sum_{j\leqslant i} \left|\Psi_j\e^{itA^\sharp}\Psi_i v\right|_{\L^p_{(-T,T)}\W^{\alpha,r}_x} = (I) + (II).$$
    We may chose different values of the parameter $\eps,\eps',\eta,\eta'$ depending on the sum we are estimating. Moreover, we may chose $M$ depending on $j$ and $i$.\\

    First, let us estimate $(I)$. We chose $M=2^j$, then \cref{Eq:EstimProofStrichartzPsijPsii} implies
    \begin{align*}
        (I) = \sum_{j\geqslant -1}\sum_{i\leqslant j} \left|\Psi_j\e^{itA^\sharp}\Psi_i v\right|_{\L^p_{(-T,T)}\W^{\alpha,r}_x} &\lesssim \sum_{j\geqslant -1}\sum_{i\leqslant j}2^{j/p}2^{-j\eta}2^{-i\eta'} \left|\Psi_i v\right|_{\W^{\alpha+\eta+\eta',2}}\\
        &+ \sum_{j\geqslant -1}\sum_{i\leqslant j} 2^{-j\frac{p-1}{p}}2^{-j\eps}2^{-i\eps'/2} \left|\Psi_i v\right|_{\W^{1+\alpha+\kappa+\eps+\eps',2}}\\
        &\lesssim  \sum_{j\geqslant -1} 2^{j/p}2^{-j\eta}\left| v\right|_{\W^{\alpha+\eta+\eta',2}}+ 2^{-j\eps}\left| v\right|_{\W^{\alpha+\kappa+\frac{1}{p}+\eps+\eps',2}}.
    \end{align*}
    Remark for this sum we can chose $\eta',\eps,\eps'$ arbitrarily small but $\eta>1/q$. Thus
    \begin{equation}
        (I)\lesssim |v|_{\W^{\alpha+\frac{1}{p}+,2}}+|v|_{\W^{\alpha+\kappa+\frac{1}{p}+,2}}\lesssim |v|_{\W^{\alpha+\kappa+\frac{1}{p}+,2}}.\label{Eq:ProofStrichartzEstim(I)}
    \end{equation}

    To estimate $(II)$, let $M=2^{i}$, it follows from \cref{Eq:EstimProofStrichartzPsijPsii} that
    \begin{align*}
        (II)=\sum_{i\geqslant -1}\sum_{j\leqslant i} \left|\Psi_j\e^{itA^\sharp}\Psi_i v\right|_{\L^p_{(-T,T)}\W^{\alpha,r}_x} &\lesssim \sum_{i\geqslant -1}\sum_{j\leqslant i}2^{i/p}2^{-j\eta}2^{-i\eta'} \left|\Psi_i v\right|_{\W^{\alpha+\eta+\eta',2}}\\
        &+ \sum_{i\geqslant -1}\sum_{j\leqslant i} 2^{-i\frac{p-1}{p}}2^{-j\eps}2^{-i\eps'/2} \left|\Psi_i v\right|_{\W^{1+\alpha+\kappa+\eps+\eps',2}}\\
        &\lesssim  \sum_{i\geqslant -1} 2^{i/p}2^{-i\eta'}\left|v\right|_{\W^{\alpha+\eta+\eta',2}}+ 2^{-i\eps'}\left|v\right|_{\W^{\alpha+\kappa+\frac{1}{q}+\eps+\eps',2}}.
    \end{align*}
    Remark this times we can chose $\eta,\eps,\eps'$ arbitrarily small but $\eta'>1/q$. Thus
    \begin{equation}
        (II)\lesssim |v|_{\W^{\alpha+\frac{1}{p}+,2}}+|v|_{\W^{\alpha+\kappa+\frac{1}{p}+,2}}\lesssim |v|_{\W^{\alpha+\kappa+\frac{1}{p}+,2}}.\label{Eq:ProofStrichartzEstim(II)}
    \end{equation}
    Then, injecting \cref{Eq:ProofStrichartzEstim(I),Eq:ProofStrichartzEstim(II)} in \cref{Eq:LocalizationDoubleSumStrichartzProof} gives
    $$\left|\e^{itA^\sharp}v\right|_{\L^p_{(-T,T)}\W^{\alpha,r}_x} \lesssim |v|_{\W^{\alpha+\kappa+\frac{1}{p}+,2}}.$$
    
\end{proof}

\begin{rem}
    Up to having a greater constant, iterating Strichartz estimates on smaller intervals allows to obtain \cref{Th:StrichartzAsharp,Th:Strichart(H+xi)} for any $T>0$.
\end{rem}

\subsection{The nonlinear equation}

Using Strichartz estimates of \cref{Th:Strichart(H+xi)}, we can now study solutions of the nonlinear equation \cref{Eq:AGP2d}. Recall the noise $\xi$ is associated to an enhanced noise $\Xi$, which we assume to verify \cref{Def:EnhancedNoise,Eq:Condqskappa}. We start with an a priori result on mass conservation.

\begin{lemme}\label[lemme]{Lem:ConservationMass}

    Let $T,\gamma>0$ and $u\in\L^\infty_{loc}((-T,T),\L^2(\R^2))\cap\L^{2\gamma+1}_{loc}((-T,T),\L^{4\gamma+2}(\R^2))$ solution of \cref{Eq:AGP2d}, then it holds
    $$\forall t\in(-T,T),\; |u(t)|_{\L^2_x}=|u(0)|_{\L^2_x}.$$
    
\end{lemme}

\begin{proof}

    Let $N\in\N$ and $\Pi_N$ be the spectral projector associated to $\wick{H+\xi}$. Let $u^N=\Pi_N u$. Then, it holds
    $$\frac{1}{2}\frac{\d}{\d t}|\Pi_N u(t)|_{\L^2_x}^2 = \left\langle \mathrm{i}\left(\wick{H+\xi}u^N+\lambda\Pi_N\left[\left|u\right|^{2\gamma}u\right]\right), u^N \right\rangle.$$
    Thus, as $\Pi_N^2=\Pi_N$ and $\wick{H+\xi}$ is self-adjoint,
    \begin{equation}\label{Eq:ProofConservationMass}
        \frac{1}{2}|\Pi_N u(t)|_{\L^2_x}^2=\frac{1}{2}|\Pi_Nu(0)|_{\L^2_x}^2+\lambda\int_0^t \langle \mathrm{i} |u|^{2\gamma}u,\Pi_Nu\rangle(s)\d s.
    \end{equation}
    Now, for almost all $s\in(-T,T)$, 
    $\langle \mathrm{i} |u|^{2\gamma}u,\Pi_Nu\rangle(s)$ converges to $\langle \mathrm{i} |u|^{2\gamma}u,u\rangle(s)=0$
    as $u(s)\in\L^2(\R^2)\cap\L^{4\gamma+2}(\R^2)$, and 
    $$|\langle \mathrm{i} |u|^{2\gamma}u,\Pi_Nu\rangle(s)| \leqslant |\Pi_Nu(s)|_{\L^2}||u(s)|^{2\gamma}u(s)|_{\L^2}\leqslant |u|_{\L^\infty_{(-T,T)}\L^2_x}|u(s)|^{2\gamma+1}_{\L^{4\gamma+2}_x},$$
    which is locallty integrable on $(-T,T)$. Thus, by dominated convergence, for any $t\in(-T,T)$, one can pass \cref{Eq:ProofConservationMass} to the limit to obtain mass conservation.

\end{proof}

We also have an a priori uniqueness result.

\begin{lemme}\label[lemme]{Lem:APrioriUniqueness}
    Let  $\lambda\in\R$, $\gamma\in\N^*$, $p>2\gamma$ and $u_0\in\L^2(\R^2)$. Let $I$ an open interval of $\R$ containing 0, there is at most one solution of \cref{Eq:AGP2d} starting at $u_0$ in $\L^\infty_{loc}(I,\L^2(\R^2))\cap\L^p_{loc}(I,\L^\infty(\R^2))$.
\end{lemme}

\begin{proof}
    Let $u_1$ and $u_2$ be two such solutions and set $R=u_1-u_2$. We prove uniqueness in the future as reversing time gives the result in the past. Let $T_0\in I$, $T_0>0$. For any $0<T<T_0$, let $Y_T = \L^\infty((0,T),\L^2(\R^2))\cap\L^p((0,T),\L^\infty(\R^2))$. It is sufficient to show that for $T$ small enough, depending only on $|u_1|_{Y_{T_0}}$ and $|u_2|_{Y_{T_0}}$, $|R|_{Y_T}=0$. By unitarity in $\L^2$ of $\left(\e^{\mathrm{i}t\wick{H+\xi}}\right)_{t\in\R}$, we have
    $$|R|_{\L^\infty_T\L^2_x}\leqslant|\lambda|\int_0^T \left||u_1(s)|^{2\gamma}u_1(s)-|u_2(s)|^{2\gamma}u_2(s)\right|_{\L^{2}_x}\d s.$$ 
    Using \cref{Lem:PowerRuleWs2} and Hölder inequality, for $\frac{2\gamma}{p}+\frac{1}{l}=1$, it holds
    $$|R|_{\L^\infty_T\L^2_x}\leqslant T^{\frac{1}{l}}|\lambda|\left(|u_0|_{Y_{T_0}}^{2\gamma}+|u_1|_{Y_{T_0}}^{2\gamma} \right)|R|_{\L^\infty_T\L^2_x}.$$
    Thus, taking $T<\left(|\lambda|\left(|u_0|_{Y_{T_0}}^{2\gamma}+|u_1|_{Y_{T_0}}^{2\gamma} \right)\right)^{-l}$ allows to conclude.
    
\end{proof}

We now make use of Strichartz estimates in order to obtain local wellposedness. All the results are stated in the future only, as reversing time allows to obtain the same results in the past. 

\begin{proof}[Proof of \cref{Th:LocalWP}.]

    The statement is a classical one for local wellposedness using a contraction argument. Showing the contraction argument and \cref{Eq:BoundExistTimeLWP} implies the rest of the statement as usual. Let $\lambda\in\R$, $\gamma\in\N^*$ such that $2\gamma<\frac{1}{2\kappa}$, $p>\max(2\gamma,\frac{1}{1+\kappa})$, $\alpha\in(1-\frac{1}{p}+\kappa,1-\kappa)$ and $u_0\in\W^{\alpha,2}$.
    Let $r$ such that $(p,r)$ is admissible and $\eps\in(0,\alpha-1-\kappa+\frac{1}{p})$. Remark that existence in $X_T = \L^\infty((0,T),\W^{\alpha,2})\cap\L^p((0,T),\W^{\frac{2}{r}+\eps,r})$ of a mild solution will imply existence in $Y_T$ using Sobolev embeddings. Uniqueness is ensured by \cref{Lem:APrioriUniqueness}. Let us show the contraction argument. For any $u\in X_T$, define
    $$\forall t\in[0,T],\; (\Phi u)(t) = \e^{\mathrm{i}t\wick{H+\xi}}u_0+\mathrm{i}\lambda\int_{0}^t \e^{\mathrm{i}(t-s)\wick{H+\xi}}\left[|u|^{2\gamma}u\right](s)\d s.$$
    For any $R>0$, define $B_T(R) = \left\{u\in X_T,\; |u|_{X_T}\leqslant R\right\}$, which is complete when endowed with the inherited topology. In what follows, $C$ denotes a finite constant independent of $u_0$ and $T$, which may vary from line to line. Let $u\in X_T$, \cref{Lem:PowerRuleWs2,Eq:UnifBoundExp(it:H+xi:)} show
    $$|\Phi u|_{\L^\infty_T\mathcal{D}^{\alpha,2}}\leqslant C\left(|u_0|_{\mathcal{D}^{\alpha,2}}+ \int_0^T \left|u(s)\right|_{\L^\infty}^{2\gamma}\left|u(s)\right|_{\mathcal{D}^{\alpha,2}}\d s\right)$$
    and, by Sobolev embeddings,
    $$\int_0^T \left|u(s)\right|_{\L^\infty}^{2\gamma}\left|u(s)\right|_{\W^{\alpha,2}}\d s \leqslant C \int_0^T \left|u(s)\right|_{\W^{\frac{2}{r}+\eps,r}}^{2\gamma}\left|u(s)\right|_{\mathcal{D}^{\alpha,2}}\d s.$$
    As $p>2\gamma$, let $l\in(1,+\infty)$ such that $\frac{2\gamma}{q}+\frac{1}{l}=1$, then Hölder's inequality implies
    \begin{equation}
        |\Phi u|_{\L^\infty_T\W^{\alpha,2}_x}\leqslant C\left( |u_0|_{\mathcal{D}^{\alpha,2}}+ T^{\frac{1}{l}}|u|_{X_T}^{2\gamma+1}\right).\label{Eq:ProofLWPLinfL2}
    \end{equation}
    Now, using \cref{Th:Strichart(H+xi)} for $\L^q_T\W^{\frac{2}{r}+\eps,r}_x$, one obtain a similar bound 
    \begin{equation}
        |\Phi u|_{\L^q_T\W^{\frac{2}{r}+\eps,r}_x}\leqslant C\left( |u_0|_{\mathcal{D}^{\alpha,2}}+ T^{\frac{1}{l}}|u|_{X_T}^{2\gamma+1}\right).\label{Eq:ProofLWPLqLr}
    \end{equation}
    Combining \cref{Eq:ProofLWPLinfL2,Eq:ProofLWPLqLr}, there exists $C_0>0$ such that
    \begin{equation}\label{Eq:ProofLWPBoundXT}
        |\Phi u|_{X_T}\leqslant C_0\left(|u_0|_{\mathcal{D}^{\alpha,2}}+ T^{\frac{1}{l}}|u|_{X_T}^{2\gamma+1}\right).
    \end{equation}
    Let $R\geqslant 2C_0|u_0|_{\mathcal{D}^{\alpha,2}}$ and $T \leqslant (2 C_0 R^{2\gamma})^{-l}$, then $\Phi : B_T(R)\to B_T(R)$. Now it suffices to prove that for some $T\leqslant (2 C_0 R^{2\gamma})^{-l}$, $\Phi : B_T(R)\to B_T(R)$ is a contraction. Let $u_1,u_2\in B_T(R)$, as previously using \cref{Th:Strichart(H+xi),Eq:UnifBoundExp(it:H+xi:)}, one has
    $$|\Phi u_1-\Phi u_2|_{X_T}\leqslant C\int_0^T \left||u_1|^{2\gamma}u_1(s)-|u_2|^{2\gamma}u_2(s)\right|_{\mathcal{D}^{\alpha,2}}\d s$$
    and by \cref{Lem:PowerRuleWs2}, it holds
    $$|\Phi u_1-\Phi u_2|_{X_T}\leqslant C\left|u_1-u_2\right|_{\L^\infty_T\mathcal{D}^{\alpha,2}}\int_0^T |u_1(s)|^{2\gamma}_{\L^\infty}+|u_2(s)|^{2\gamma}_{\L^\infty}\d s.$$
    By Hölder's inequality, there exists $C_1>0$ such that
    $$\forall u_1,u_2\in B_T(R),\; |\Phi u_1-\Phi u_2|_{X_T}\leqslant C_1 |u_1-u_2|_{X_T} T^{\frac{1}{l}}R^{2\gamma}.$$
    Hence, for $T = \left(2\max(C_0,C_1)R^{2\gamma}\right)^{-l}$, $\Phi : B_T(R)\to B_T(R)$ is a contraction and the end of the proof is as usual.
    
\end{proof}

Now, recall the energy associated to \cref{Eq:2DRenormA}, defined by 
$$\forall u\in\mathcal{D}^{1,2},\; \mathcal{E}(u)=\frac{1}{2}\langle-Au,u\rangle-\frac{\lambda}{2\gamma+2}\int_{\R^2}|u|^{2\gamma+2}\d x.$$
This quantity is formally conserved by the nonlinear flow. In order to show this conservation, we start by studying solutions starting from the domain. We shall show there is a unique solution starting and staying in the domain.

\begin{cor}\label[cor]{Cor:LWPD(-A)}

    Let $k\in\{1,2\}$. Let $\lambda\in\R$ and $\gamma$ as in \cref{Th:LocalWP}. Let $u_0\in\mathcal{D}^{k,2}$. There exists a unique maximal solution $(I,u)$ of \cref{Eq:AGP2d} in $\L^\infty_{loc}(I,\mathcal{D}^{k,2})$ and it is the one given by \cref{Th:LocalWP}. Moreover its energy is conserved,
    $$\forall t\in I,\; \mathcal{E}(u(t))=\mathcal{E}(u_0).$$
    
\end{cor}

\begin{proof}

    First assume $k=2$. As for any $T>0$, Sobolev embeddings and \cref{Prop:CaracDs2} imply $\L^\infty([0,T],\mathcal{D}^{2,2})$ is continuously embedded in $Y_T$, there is at most one maximal $\mathcal{D}^{2,2}$-valued solution and if there is one, it must be the one given by \cref{Th:LocalWP}. For such a solution $u\in\L^\infty(I,\mathcal{D}^{2,2})$, $\p_t u(t)$ and $\nabla\mathcal{E}(u(t))$ are in $\L^2(\R^2)$ for any $t\in I$. Thus, one can compute
    \begin{align*}
        \frac{\d}{\d t}\mathcal{E}(u(t)) &= \left(\nabla\mathcal{E}(u(t)),\p_t u(t)\right)_{\L^2}\\
        &= -\left(Au(t)+\lambda|u(t)|^{2\gamma}u(t),\mathrm{i}\left(Au(t)+\lambda|u(t)|^{2\gamma}u(t)\right)\right)_{\L^2}=0
    \end{align*}
    and obtain energy conservation. We only need to prove such a solution exists. Let $(I,u)$ be the unique maximal solution in $Y_T$ starting at $u_0$ given by \cref{Th:LocalWP}. Let $0<T_0\in I$ and $0<T<T_0$. Remark
    $$|u|_{\L^\infty_T\mathcal{D}^{2,2}}\lesssim |\wick{H+\xi}u|_{\L^\infty_T\L^{2}_x} +|u_0|_{\L^2}\lesssim |\p_tu|_{\L^\infty_T\L^{2}_x}+\left||u|^{2\gamma}u\right|_{\L^\infty_T\L^2_x}+|u_0|_{\L^2},$$
    by conservation of mass. The term $\left||u|^{2\gamma}u\right|_{\L^\infty_T\L^{2}_x}=\left|u\right|_{\L^\infty_T\L^{4\gamma+2}_x}^{2\gamma+1}$ is finite. Indeed, as $2\gamma<\frac{1}{2\kappa}$,  $\frac{2\gamma}{2\gamma+1}<1-\kappa$ implying $u\in\mathcal{C}(I,\mathcal{D}^{\frac{2\gamma}{2\gamma+1},2})\subset\L^\infty_{loc}(I,\L^{4\gamma+2}(\R^2))$. Now, define $w=\p_t u$. We want to show $w\in\L^\infty_{loc}(I,\L^2(\R^2))$. Remark $w$ solves
    \begin{equation}\label{Eq:EDPdu/dt}
        \begin{cases}
            \mathrm{i}\p_t w + \wick{H+\xi}w + \lambda\left(|u|^{2\gamma}w+2\gamma|u|^{2\gamma-2}u\re\left(\overline{u}w\right)\right) = 0,\\
            w(0) = \mathrm{i}\left(\wick{H+\xi}u_0+\lambda|u_0|^{2\gamma}u_0\right).
        \end{cases}
    \end{equation}
    We write \cref{Eq:EDPdu/dt} in mild form
    $$w(t)=\mathrm{i}\e^{\mathrm{i}t\wick{H+\xi}}w(0)+\mathrm{i}\lambda\int_0^t\e^{\mathrm{i}(t-s)\wick{H+\xi}}\left[|u|^{2\gamma}w+2\gamma|u|^{2\gamma-2}u\re(\overline{u}w)\right](s)\d s.$$
    The unitarity of $\left(\e^{\mathrm{i}t\wick{H+\xi}}\right)_{t\in\R}$ on $\L^{2}(\R^2)$ gives
    \begin{align*}
        |w|_{\L^\infty_T\L^2_x} &\leqslant \left|\wick{H+\xi}u_0+\lambda|u_0|^{2\gamma}u_0\right|_{\L^2_x}+ |\lambda|\int_0^T \left||u|^{2\gamma}w(s)+2\gamma|u|^{2\gamma-2}u\re(\overline{u}w)(s)\right|_{\L^2_x}\d s\\
        &\leqslant \left|\wick{H+\xi}u_0+\lambda|u_0|^{2\gamma}u_0\right|_{\L^2_x} + |\lambda|(2\gamma+1)|w|_{\L^\infty_T\L^2_x}\int_0^{T_0} \left|u(s)\right|_{\L^{\infty}_x}^{2\gamma}\d s.
    \end{align*}
    Let $l\in(1,+\infty)$ such that $\frac{2\gamma}{q}+\frac{1}{l}=1$, then using Sobolev embeddings and Hölder inequality, 
    $$|w|_{\L^\infty_T\L^2_x}\leqslant \left|\wick{H+\xi}u_0+\lambda|u_0|^{2\gamma}u_0\right|_{\L^2_x} + |\lambda|(2\gamma+1)T^{\frac{1}{l}}|w|_{\L^\infty_T\L^2_x}\left|u\right|_{Y_{T_0}}^{2\gamma}.$$
    This allows to conclude $u\in\L^\infty_{loc}(I,\mathcal{D}^{2,2})$.\\

    Now let $k=1$. Let $(u_0^n)_{n\in\N}\subset\mathcal{D}^{2,2}$ such that $u_0^n$ converges to $u_0$ in $\mathcal{D}^{1,2}$ and $(I_n,u^n)$ the maximal $\mathcal{D}^{2,2}$-valued solution given by the previous case. Let $p$ and $\alpha$ as in \cref{Th:LocalWP}. Then, for any $T\in I$, there exists some $n_T\in\N$ such that for $n\geqslant n_T$, $u^n\in\L^\infty([0,T],\mathcal{D}^{2,2})$ and converges to $u$ in $Y_T$. Now using energy conservation,
    $$\forall t\in[0,T],\; \frac{1}{2}\langle -Au^n(t),u^n(t)\rangle= \mathcal{E}(u_0^n) + \frac{\lambda}{2\gamma+2}\int_{\R^2}\left|u^n(t)\right|^{2\gamma+2}\d x.$$
    Using convergence in $\mathcal{C}([0,T],\W^{\alpha,2})$ and Sobolev embeddings 
    $$
    \sup_{n\geqslant n_T}\sup_{t\in[0,T]}\langle -Au^n(t),u^n(t)\rangle<+\infty.$$
    A standard compactness argument shows, by uniqueness of the limit, $u\in\L^\infty([0,T],\mathcal{D}^{1,2})$ and a lower semi-continuity argument gives
    $$\forall t\in[0,T],\; \mathcal{E}(u(t))\leqslant \mathcal{E}(u_0).$$
    We conclude to energy conservation by reversing time.\\

    To conclude, let $T_0>0$ and $u=\in\L^\infty_{loc}([0,T_0),\mathcal{D}^{1,2})$ solution of \cref{Eq:AGP2d}. We will show $u\in\L^p_{loc}([0,T_0),\L^\infty(\R^2))$ for some $p>\max(2\gamma,\frac{1}{1+\kappa})$, thus showing $u$ coincide with the solution given by \cref{Th:LocalWP}. Let us show as in \cite{BurqStrichartzManifold} that $|u|^{2\gamma}u\in\L^\infty_{loc}([0,T_0),\W^{s,2})$ for some $s<1$. By Sobolev embbedings, $u\in\L^\infty_{loc}([0,T_0),\L^r(\R^2))$ for any $r\in[2,+\infty)$. Let $v=\rho^{-1}u\in\L^\infty_{loc}([0,T_0),\W^{1,2})$, then Hölder's inequality and \cref{Cor:action_V} imply $|v|_{\L^r}\lesssim |\langle x\rangle^{-1}|_{\L^{s}}|v|_{\W^{1,2}}$, with $\frac{1}{s}+\frac{1}{2}=\frac{1}{r}$. Thus, for $r>1$, $v\in\L^\infty_{loc}([0,T_0),\L^r(\R^2))$. In particular, $|v|^{2\gamma}v\in\L^\infty_{loc}([0,T_0),\L^r(\R^2))$ for any $r\in\left(1,+\infty\right)$. Moreover, as $\nabla v,xv\in\L^\infty_{loc}([0,T_0),\L^2(\R^2))$ and $|v|^{2\gamma}\in\L^\infty_{loc}([0,T_0),\L^r(\R^2))$ for any $r\in\left(1,+\infty\right)$, Hölder's inequality implies $|v|^{2\gamma}v\in\L^\infty_{loc}([0,T_0),\W^{1,r})$ for $r\in(1,2)$. In particular, by Sobolev embeddings, $|v|^{2\gamma}v\in\L^\infty_{loc}([0,T_0),\W^{s,2})$ for any $s<1$ and by \cref{Lem:ProductRule}, $|u|^{2\gamma}u\in\L^\infty_{loc}([0,T_0),\mathcal{D}^{1-\kappa,2})$. Now, write $u$ in mild form
    $$u(t) = \e^{\mathrm{i}t\wick{H+\xi}}u_0 - \mathrm{i}\lambda\int_0^t\e^{\mathrm{i}(t-s)\wick{H+\xi}}\left[|u|^{2\gamma}u\right](s)\d s.$$ Then, by \cref{Th:Strichart(H+xi)}, for any admissible pair $(p,r)$ and $\eps>0$ such that $\frac{2}{r}+\frac{1}{p}+\kappa+2\eps<1-\kappa$, it holds for any $0<T<T_0$,
    $$|u|_{\L^p_{T}\L^\infty_x} \lesssim |u|_{\L^p_{T}\W^{\frac{2}{r}+\eps,r}_x} \lesssim |u_0|_{\W^{\frac{2}{r}+\frac{1}{p}+\kappa+2\eps,2}}+\int_0^T ||u|^{2\gamma}u(s)|_{\W^{\frac{2}{r}+\frac{1}{p}+\kappa+2\eps,2}}\d s <+\infty$$
    as $\frac{2}{r}+\frac{1}{p}+\kappa+2\eps<1-\kappa$, $u_0\in\mathcal{D}^{1,2}$ and $|u|^{2\gamma}u\in\L^\infty([0,T],\mathcal{D}^{1-\kappa,2})$. Hence, $u\in\L^p_{loc}([0,T_0),\L^\infty(\R^2))$ and thus uniqueness follows from \cref{Th:LocalWP} as there exists an admissible pair $(p,r)$ such that $\frac{2}{r}+\frac{1}{p}+\kappa<1-\kappa$ and $p>\max(2\gamma,\frac{1}{1+\kappa})$ because $\kappa\in(0,\frac{1}{2})$ and $2\gamma<\frac{1}{2\kappa}$.
\end{proof}

Using energy conservation, we easily obtain the following global existence result as usual. We introduce the noisy Gagliardo-Niremberg constant
\begin{equation}
    J_\Xi=\inf_{u\in\mathcal{D}^{1,2}\backslash\{0\}}\frac{|u|_{\mathcal{D}^{1,2}}^2|u|_{\L^2}^2}{|u|_{\L^4}^4}\in(0,+\infty).\label{Eq:NoisyGagliardoNiremberg}
\end{equation}

\begin{cor}\label[cor]{Cor:GWPD(-A1/2)}

    Let $\lambda\in\R$ and $\gamma$ as in \cref{Th:LocalWP}. Let $u_0\in\mathcal{D}^{1,2}$ and $(I,u)$ be the unique maximal solution starting from $u_0$ given by \cref{Th:LocalWP}.
    Assume $\lambda\leqslant 0$, $\lambda>0$ and $\gamma<1$ or $\lambda>0$, $\gamma=1$ and $|u_0|^2_{\L^2}<\frac{2J_\Xi}{\lambda}$ where $J_\Xi$ is given by \cref{Eq:NoisyGagliardoNiremberg}. Then $I=\R$.
    
\end{cor}

\section*{Acknowledgments}

This work was supported by a public grant from the Fondation Mathématique Jacques Hadamard. The author was supported by the ANR project Smooth ANR-22-CE40-0017. This research was supported by the Japan Society for the Promotion of Science (JSPS) Summer Program 2024 (fellow SP24214). The author want to thank Waseda University for hosting their JSPS Summer Program fellowship and in particular Reika Fukuizumi who took a lot of time and energy to organize the fellowship. The author express its deepest gratitude to Antoine Mouzard for the fruitful conversation they had on paracontrolled ansatz and paracontrolled calculus, which has been of great help for this work. Finally, the author warmly thanks their PhD advisor, Anne de Bouard, for her constant support and helpful advices during the redaction of this paper.

\appendix

\section{Paraproducts in weighted spaces}\label{Sec:ProofParaprod}

In this appendix, we give a proof to \cref{Prop:ContinuityParaprodutcs,Cor:ContinuityParaprodutcsWsp,Lem:ComPuH}. The following lemmas are weighted equivalent to Lemma 2.69, 2.84 and 2.79 in \cite{BahouriCheminDanchin}, the proof are omitted as they are similar.

\begin{lemme}\label[lemme]{Lem:EquivBesovSum2jAnnulus}
    Let $\alpha\in\R$, $1\leqslant p,q\leqslant+\infty$ and $w$ an admissible weight. Let $\mathcal{C}$ be an annulus in $\R^2$. There exists a constant $C=C(\alpha,\mathcal{C})>0$ such that, for any sequence of smooth functions $(f_j)_{j\in\N}$ with
    $$\supp \mathcal{F}(f_j)\subset 2^j \mathcal{C} \text{ and } \left(2^{j\alpha}|f_j|_{\L^p_w}\right)_{j\in\N}\in l^q,$$
    it holds
    $$f=\sum_{j=0}^{+\infty} f_j \in \mathrm{B}^{\alpha}_{p,q,w} \text{ and } |f|_{\mathrm{B}^{\alpha}_{p,q,w}} \leqslant C \left|2^{j\alpha}|f_j|_{\L^p_w}\right|_{l^q}.$$
\end{lemme}

\begin{lemme}\label[lemme]{Lem:EquivBesovSum2jBall}
    Let $\alpha>0$, $1\leqslant p,q\leqslant+\infty$ and $w$ an admissible weight. Let $\mathrm{B}$ be a ball in $\R^2$. There exists a constant $C=C(\alpha,\mathrm{B})>0$ such that, for any sequence of smooth functions $(f_j)_{j\in\N}$ with
    $$\supp \mathcal{F}(f_j)\subset 2^j \mathrm{B} \text{ and } \left(2^{j\alpha}|f_j|_{\L^p_w}\right)_{j\in\N}\in l^q,$$
    it holds
    $$f=\sum_{j=0}^{+\infty} f_j \in \mathrm{B}^{\alpha}_{p,q,w} \text{ and } |f|_{\mathrm{B}^{\alpha}_{p,q,w}} \leqslant C \left|2^{j\alpha}|f_j|_{\L^p_w}\right|_{l^q}.$$
\end{lemme}

\begin{lemme}\label[lemme]{Lem:EquivBesovnegRegDelta<=j}
    Let $\alpha<0$, $1\leqslant p,q\leqslant+\infty$ and $w$ an admissible weight. Let $u\in\mathcal{S}'(\R^2)$. Then $u\in\mathrm{B}^{\alpha}_{p,q,w}$ if and only if $\left(2^{j\alpha}|\Delta_{\leqslant j} u|_{\L^p_w}\right)_{j\in\N}\in l^q.$
    Moreover, it holds
    $$|u|_{\mathrm{B}^{\alpha}_{p,q,w}}\approx_\alpha \left|2^{j\alpha}|\Delta_{\leqslant j} u|_{\L^p_w}\right|_{l^q}$$
    is the sense of equivalent norm.
\end{lemme}

As a direct consequence of \cref{Lem:EquivBesovSum2jAnnulus}, $\mathrm{B}^{\alpha}_{p,q,w}$ does not depends on the choice of Paley-Littlewood projectors. We will use \cref{Lem:EquivBesovnegRegDelta<=j,Lem:EquivBesovSum2jAnnulus,Lem:EquivBesovSum2jBall} to prove the following proposition on regularity of paraproducts and resonant products in weighted Besov spaces.

\begin{prop}\label[prop]{Prop:ContinuityParadifferentielGeneral}
    Let $\alpha_1,\alpha_2\in\R$, $1\leqslant p,p_1,p_2,q,q_1,q_2\leqslant+\infty$ such that $\frac{1}{p}=\frac{1}{p_1}+\frac{1}{p_2}$ and $\frac{1}{q}=\frac{1}{q_1}+\frac{1}{q_2}$. Let $w_1$ and $w_2$ be admissible weights.
    \begin{enumerate}
        \item Assume $\alpha_1>0$. There exists $C>0$ such that
        $$\forall f\in\mathrm{B}^{\alpha_1}_{p_1,\infty,w_1},\; \forall g\in\mathrm{B}^{\alpha_2}_{p_2,q,w_2},\; |P_f g|_{\mathrm{B}^{\alpha_2}_{p,q,w_1w_2}}\leqslant C|f|_{\mathrm{B}^{\alpha_1}_{p_1,\infty,w_1}}|g|_{\mathrm{B}^{\alpha_2}_{p_2,q,w_2}}.$$
        \item Assume $\alpha_1<0$. There exists $C>0$ such that
        $$\forall f\in\mathrm{B}^{\alpha_1}_{p_1,q_1,w_1},\; \forall g\in\mathrm{B}^{\alpha_2}_{p_2,q_2,w_2},\; |P_f g|_{\mathrm{B}^{\alpha_1+\alpha_2}_{p,q,w_1w_2}}\leqslant C|f|_{\mathrm{B}^{\alpha_1}_{p_1,q_1,w_1}}|g|_{\mathrm{B}^{\alpha_2}_{p_2,q_2,w_2}}.$$
        \item Assume $\alpha_1+\alpha_2>0$. There exists $C>0$ such that
        $$\forall f\in\mathrm{B}^{\alpha_1}_{p_1,q_1,w_1},\; \forall g\in\mathrm{B}^{\alpha_2}_{p_2,q_2,w_2},\; |R(f,g)|_{\mathrm{B}^{\alpha_1+\alpha_2}_{p,q,w_1w_2}}\leqslant C|f|_{\mathrm{B}^{\alpha_1}_{p_1,q_1,w_1}}|g|_{\mathrm{B}^{\alpha_2}_{p_2,q_2,w_2}}.$$
    \end{enumerate}
\end{prop}

\begin{proof}
    Let $f\in\mathrm{B}^{\alpha_1}_{p_1,q_1,w_1}$ and $g\in\mathrm{B}^{\alpha_2}_{p_2,q_2,w_2}$.
    \begin{enumerate}
        \item Assume $\alpha_1>0$ and $q_1=+\infty$. Recall $P_f g$ is defined by \cref{Eq:DefParaproduct}. Now remark there exists an annulus $\mathcal{C}\subset\R^2$ such that for any $j\geqslant 1$,
        $$\supp \mathcal{F}\left(\Delta_{\leqslant j-2}f\Delta_j g\right)\subset 2^j \mathcal{C}.$$ Thus, by \cref{Lem:EquivBesovSum2jAnnulus}, it is sufficient to show
        $$\left|2^{j\alpha_2}|\Delta_{\leqslant j-2}f\Delta_j g|_{\L^p_{w_1w_2}}\right|_{l^q}<+\infty.$$
        By Hölder's inequality, it holds
        $$|\Delta_{\leqslant j-2}f\Delta_j g|_{\L^p_{w_1w_2}} \leqslant |\Delta_{\leqslant j-2}f|_{\L^{p_1}_{w_1}}|\Delta_j g|_{\L^{p_2}_{w_2}}.$$
        Now,
        $$|\Delta_{\leqslant j-2}f|_{\L^{p_1}_{w_1}} \leqslant \sum_{i=-1}^{j-2} |\Delta_i f|_{\L^{p_1}_{w_1}} \leqslant \left(\sum_{i=-1}^{j-2} 2^{-\alpha_1 i}\right)|f|_{\mathrm{B}^{\alpha_1}_{p_1,\infty,w_1}}\leqslant \frac{2^{\alpha_1}}{1-2^{-\alpha_1}}|f|_{\mathrm{B}^{\alpha_1}_{p_1,\infty,w_1}}.$$
        Thus \cref{Lem:EquivBesovSum2jAnnulus}, implies 
        $$|P_fg|_{\mathrm{B}^{\alpha_2}_{p,q,w_1w_2}}\lesssim |f|_{\mathrm{B}^{\alpha_1}_{p_1,\infty,w_1}}|g|_{\mathrm{B}^{\alpha_2}_{p_2,q_2,w_2}}.$$
        \item Assume $\alpha_1<0$. As previously, by \cref{Lem:EquivBesovSum2jAnnulus}, it is sufficient to show
        $$\left|2^{j(\alpha_1+\alpha_2)}|\Delta_{\leqslant j-2}f\Delta_j g|_{\L^p_{w_1w_2}}\right|_{l^q}<+\infty.$$
        By Hölder's inequality, it holds
        \begin{align*}
            \left|2^{j(\alpha_1+\alpha_2)}|\Delta_{\leqslant j-2}f\Delta_j g|_{\L^p_{w_1w_2}}\right|_{l^q} &\leqslant \left|2^{j\alpha_1}|\Delta_{\leqslant j-2}f|_{\L^{p_1}_{w_1}}\times 2^{j\alpha_2}|\Delta_j g|_{\L^{p_2}_{w_2}}\right|_{l^q}\\
            &\leqslant \left|2^{j\alpha_1}|\Delta_{\leqslant j-2}f|_{\L^{p_1}_{w_1}}\right|_{l^{q_1}}\left|2^{j\alpha_2}|\Delta_j g|_{\L^{p_2}_{w_2}}\right|_{l^{q_2}}\\
            &\lesssim |f|_{\mathrm{B}^{\alpha_1}_{p_1,q_1,w_1}} |g|_{\mathrm{B}^{\alpha_2}_{p_2,q_2,w_2}},
        \end{align*}
        where we used \cref{Lem:EquivBesovnegRegDelta<=j} for the last line. By \cref{Lem:EquivBesovSum2jAnnulus}, we conclude 
        $$|P_fg|_{\mathrm{B}^{\alpha_1+\alpha_2}_{p,q,w_1w_2}}\lesssim  |f|_{\mathrm{B}^{\alpha_1}_{p_1,q_1,w_1}}|g|_{\mathrm{B}^{\alpha_2}_{p_2,q_2,w_2}}.$$
        \item Recall $R(f,g)$ is defined by \cref{Eq:DefResonnantProduct}. Now remark there exists a ball $\mathrm{B}\subset\R^2$ such that for any $j\geqslant 1$,
        $$\supp \mathcal{F}\left(\left(\Delta_{\leqslant j+1} f - \Delta_{\leqslant j-2} f\right)\Delta_j g\right)\subset 2^j \mathrm{B}.$$ As $\alpha_1+\alpha_2>0$, by \cref{Lem:EquivBesovSum2jBall}, it is sufficient to show
        $$\left|2^{j(\alpha_1+\alpha_2)}|\left(\Delta_{\leqslant j+1} f - \Delta_{\leqslant j-2} f\right)\Delta_j g|_{\L^p_{w_1w_2}}\right|_{l^q}<+\infty.$$
        By Hölder's inequality and sublinearity of the norm, it holds
        \begin{align*}
            \left|2^{j(\alpha_1+\alpha_2)}|\left(\Delta_{\leqslant j+1} f - \Delta_{\leqslant j-2} f\right)\Delta_j g|_{\L^p_{w_1w_2}}\right|_{l^q} &\leqslant \sum_{i=-1}^1  \left|2^{j\alpha_1}|\Delta_{j+i}f|_{\L^{p_1}_{w_1}}\right|_{l^{q_1}}\left|2^{j\alpha_2}|\Delta_j g|_{\L^{p_2}_{w_2}}\right|_{l^{q_2}}\\
            &\lesssim |f|_{\mathrm{B}^{\alpha_1}_{p_1,q_1,w_1}} |g|_{\mathrm{B}^{\alpha_2}_{p_2,q_2,w_2}}.
        \end{align*}
        Thus, by \cref{Lem:EquivBesovSum2jBall},
        $$|R(f,g)|_{\mathrm{B}^{\alpha_1+\alpha_2}_{p,q,w_1w_2}}\lesssim |f|_{\mathrm{B}^{\alpha_1}_{p_1,q_1,w_1}}|g|_{\mathrm{B}^{\alpha_2}_{p_2,q_2,w_2}}.$$
    \end{enumerate}
\end{proof}

We can now prove \cref{Prop:ContinuityParaprodutcs,Cor:ContinuityParaprodutcsWsp}. We will make use of the following embeddings.
    
\begin{lemme}\label[lemme]{Lem:EmbeddingSobolevBesov}

    Let $s\in\R$, $p\in(1,+\infty)$ and $w$ an admissible weight. For any $\eps>0$, it holds
    $$ B^{s+\eps}_{p,\infty,w}\subset \mathrm{H}^{s,p}_{w} \subset B^s_{p,\infty,w}.$$
    
\end{lemme}

\begin{proof}
    Let $s\in\R$ and $p\in(1,+\infty)$. According to Section 2.5.6 of \cite{TriebelFunctionSpacesI}, $\mathrm{H}^{s,p}=F^s_{p,2}$ where $F^s_{p,2}$ is a Triebel-Lizorkin space. Now Remark 2 in Section 2.3.2 of \cite{TriebelFunctionSpacesII} implies
        $$\forall q\in[1,p],\; B^s_{p,q}\subset \mathrm{H}^s_{p,q}\subset B^s_{p,p},$$
        $$\forall q\in[p,+\infty],\; B^s_{p,p}\subset \mathrm{H}^s_{p,q}\subset B^s_{p,q}$$
        and
        $$\forall \eps>0,\; \forall q_1,q_2\in[1,+\infty],\; B^{s+\eps}_{p,q_1}\subset B^s_{p,q_2}.$$
        As Hölder inequality for series implies
        $$\forall 1\leqslant q_1\leqslant q_2\leqslant +\infty,\; B^s_{p,q_1}\subset B^s_{p,q_2},$$
        the result now follows using norm equivalence \cref{Eq:NormEquivWeightedBesov}.
\end{proof}

\begin{proof}[Proof of \cref{Prop:ContinuityParaprodutcs}.]

    Let ${\alpha_1},{\alpha_2}\in\R$, $p\in(1,+\infty)$, $\eps>0$ and $\eps_1,\eps_2\in[0,\eps]$ such that $\eps_1+\eps_2 = \eps$. Let $w_1,w_2$ be admissible weights.
    \begin{enumerate}
        \item Assume $\alpha_1\geqslant 0$. 
        \begin{itemize}
            \item Let $f\in \mathrm{H}^{\alpha_1,p}_{w_1}$ and $g\in\mathcal{C}^{\alpha_2+\eps}_{w_2}$, then by \cref{Lem:EmbeddingSobolevBesov}, it holds
            $$|P_fg|_{\mathrm{H}^{\alpha_2,p}_{w_1w_2}} \lesssim |P_fg|_{\mathrm{B}^{\alpha_2+\eps}_{p,\infty,w_1w_2}}.$$
            If $\alpha_1>0$, \cref{Prop:ContinuityParadifferentielGeneral} $(i)$ and \cref{Lem:EmbeddingSobolevBesov}, it holds
            \begin{align*}
                |P_fg|_{W^{\alpha_2,p}_{w_1w_2}} &\lesssim |f|_{\mathrm{B}^{\alpha_1}_{p,\infty,w_1}}|g|_{\mathcal{C}^{\alpha_2+\eps}_{w_2}}\\
                &\lesssim |f|_{\mathrm{H}^{\alpha_1,p}_{w_1}}|g|_{\mathcal{C}^{\alpha_2+\eps}_{w_2}}
            \end{align*}
            and if $\alpha_1=0$, by \cref{Lem:EmbeddingSobolevBesov},\cref{Prop:ContinuityParadifferentielGeneral} $(i)$ and \cref{Lem:EmbeddingSobolevBesov}, it holds
            \begin{align*}
                |P_fg|_{\mathrm{H}^{\alpha_2,p}_{w_1w_2}} &\lesssim |P_fg|_{\mathrm{B}^{\alpha_2+\frac{\eps}{2}}_{p,\infty,w_1w_2}}\\
                &\lesssim |f|_{\mathrm{B}^{-\frac{\eps}{2}}_{p,\infty,w_1}}|g|_{\mathcal{C}^{\alpha_2+\eps}_{w_2}}\\
                &\lesssim |f|_{\mathrm{H}^{\alpha_1,p}_{w_1}}|g|_{\mathcal{C}^{\alpha_2+\eps}_{w_2}}.
            \end{align*}
            \item Similarly, let $f\in \mathcal{C}^{\alpha_1}_{w_1}$ and $g\in \mathrm{H}^{\alpha_2+\eps,p}_{w_2}$, then by \cref{Lem:EmbeddingSobolevBesov}, it holds
            $$|P_fg|_{\mathrm{H}^{\alpha_2,p}_{w_1w_2}} \lesssim |P_fg|_{\mathrm{B}^{\alpha_2+\eps}_{p,\infty,w_1w_2}}$$
            and by \cref{Prop:ContinuityParadifferentielGeneral} $(i)$ and \cref{Lem:EmbeddingSobolevBesov}, it holds
            \begin{align*}
                |P_f g|_{\mathrm{H}^{\alpha_2,p}_{w_1w_2}} &\lesssim |f|_{\mathcal{C}^{\alpha_1}_{w_1}}|g|_{\mathrm{B}^{\alpha_2+\eps}_{p,\infty,w_2}}\\
                &\lesssim |f|_{\mathcal{C}^{\alpha_1}_{w_1}}|g|_{\mathrm{H}^{\alpha_2+\eps,p}_{w_2}},
            \end{align*}
            with the same modification as previously if $\alpha_1=0$.
        \end{itemize}
        \item Assume $\alpha_1+\eps_1<0$. 
        \begin{itemize}
            \item Let $f\in \mathrm{H}^{\alpha_1+\eps_1,p}_{w_1}$ and $g\in\mathcal{C}^{\alpha_2+\eps_2}_{w_2}$, by \cref{Lem:EmbeddingSobolevBesov}, it holds
            $$|P_fg|_{\mathrm{H}^{\alpha_1+\alpha_2,p}_{w_1w_2}} \lesssim |P_fg|_{\mathrm{B}^{\alpha_1+\alpha_2+\eps}_{p,\infty,w_1w_2}}.$$
            By \cref{Prop:ContinuityParadifferentielGeneral} $(ii)$ and \cref{Lem:EmbeddingSobolevBesov}, it holds
            \begin{align*}
                |P_fg|_{\mathrm{H}^{\alpha_1+\alpha_2,p}_{w_1w_2}} &\lesssim |f|_{\mathrm{B}^{\alpha_1+\eps_1}_{p,\infty,w_1}}|g|_{\mathcal{C}^{\alpha_2+\eps_2}_{w_2}}\\
                &\lesssim |f|_{\mathrm{H}^{\alpha_1+\eps_1}_{w_1}}|g|_{\mathcal{C}^{\alpha_2+\eps_2}_{w_2}}.
            \end{align*}
            \item Similarly, let $f\in \mathcal{C}^{\alpha_1+\eps_1}_{w_1}$ and $g\in \mathrm{H}^{\alpha_2+\eps_2,p}_{w_2}$, then by \cref{Lem:EmbeddingSobolevBesov}, it holds
            $$|P_fg|_{\mathrm{H}^{\alpha_1+\alpha_2,p}_{w_1w_2}} \lesssim |P_fg|_{\mathrm{B}^{\alpha_1+\alpha_2+\eps}_{p,\infty,w_1w_2}}$$
            and by \cref{Prop:ContinuityParadifferentielGeneral} $(ii)$ and \cref{Lem:EmbeddingSobolevBesov}, it holds
            \begin{align*}
                |P_f g|_{\mathrm{H}^{\alpha_1+\alpha_2,p}_{w_1w_2}} &\lesssim |f|_{\mathcal{C}^{\alpha_1+\eps_1}_{w_1}}|g|_{\mathrm{B}^{\alpha_2+\eps_2}_{p,\infty,w_2}}\\
                &\lesssim |f|_{\mathcal{C}^{\alpha_1+\eps_1}_{w_1}}|g|_{\mathrm{H}^{\alpha_2+\eps_2,p}_{w_2}}.
            \end{align*}
        \end{itemize}
        \item The case $3.$ is similar to $2$.
    \end{enumerate}

\end{proof}

\begin{proof}[Proof of \cref{Cor:ContinuityParaprodutcsWsp}.]

    Let $0\leqslant\sigma'\leqslant \sigma$, $\eps,\eps'>0$, $0\leqslant \kappa'<\kappa$ and $p\in[2,+\infty)$.
    \begin{enumerate}
        \item\begin{itemize}
            \item Let $f\in \mathcal{C}^{\eps'}$ and any $g\in\W^{\sigma+\eps,p}$. By norm equivalence, it holds
            $$|P_f g|_{\W^{\sigma,p}}\lesssim |P_f g|_{\mathrm{H}^{\sigma,p}}+|P_f g|_{\L^{p}_{\sigma}}.$$
            Then, as $\eps'>0$, \cref{Prop:ContinuityParaprodutcs} implies
            $$|P_f g|_{\mathrm{H}^{\sigma,p}}\lesssim |f|_{\mathcal{C}^{\eps'}} |g|_{\mathrm{H}^{\sigma+\eps,p}}$$
            and
            $$|P_f g|_{\L^{p}_{\sigma}}\lesssim |f|_{\mathcal{C}^{\eps'}} |g|_{\mathrm{H}^{\eps,p}_{\sigma}}.$$
            As $\W^{\alpha,p}$ is continuously embedded in $\mathrm{H}^{\alpha,p}$ for $\alpha\geqslant 0$, \cref{Cor:action_V} implies 
            $$|P_f g|_{\W^{\sigma,p}}\lesssim |f|_{\mathcal{C}^{\eps'}}|g|_{\W^{\sigma+\eps,p}}.$$
            \item Let $ f\in \W^{\sigma-\sigma',p}$ and $g\in\mathcal{C}^{\sigma+\eps}\cap\mathcal{C}^{\eps'}_{\sigma'}$. By norm equivalence, it holds
            $$|P_f g|_{\W^{\sigma,p}}\lesssim |P_f g|_{\mathrm{H}^{\sigma,p}}+|P_f g|_{\L^{p}_{\sigma}}.$$
            Then, as $\sigma-\sigma'\geqslant 0$, \cref{Prop:ContinuityParaprodutcs} implies
            $$|P_f g|_{\mathrm{H}^{\sigma,p}}\lesssim |f|_{\mathrm{H}^{\sigma-\sigma',p}} |g|_{\mathcal{C}^{\sigma+\eps}}$$
            and
            $$|P_f g|_{\L^{p}_{\sigma}}\lesssim |f|_{\L^{p}_{V^{\frac{\sigma-\sigma'}{2}}}} |g|_{\mathcal{C}^{\eps'}_{\sigma'}}.$$
            As $\W^{\alpha,p}$ is continuously embedded in $\mathrm{H}^{\alpha,p}$ for $\alpha\geqslant 0$ and by \cref{Cor:action_V}, it holds 
            $$|P_f g|_{\W^{\sigma,p}}\lesssim|f|_{\W^{\sigma-\sigma',p}}\left(|g|_{\mathcal{C}^{\sigma+\eps}}+|g|_{\mathcal{C}^{\eps'}_{\sigma'}}\right).$$
        \end{itemize}
        \item Let $ f\in\mathcal{C}^{\kappa'-\kappa}_{-\kappa'}$ and  $g\in\W^{\sigma+\kappa+\eps,p}$. By norm equivalence, it holds
        $$|P_f g|_{\W^{\sigma,p}}\lesssim |P_f g|_{\mathrm{H}^{\sigma,p}}+|P_f g|_{\L^{p}_{\sigma}}.$$
        Let $w_1 = \langle x\rangle^{\sigma+\kappa'}$ and $w_2=\langle x\rangle^{-\kappa'}$, which are admissible weights. As $\kappa'-\kappa<0$, \cref{Prop:ContinuityParaprodutcs} implies
        $$|P_f g|_{\mathrm{H}^{\sigma,p}}\lesssim |f|_{\mathcal{C}^{\kappa'-\kappa}_{w_2}} |g|_{\mathrm{H}^{\sigma+\kappa-\kappa'+\eps,p}_{w_2^{-1}}}.$$
        Similarly, 
        $$|P_f g|_{\L^{p}_{\sigma}} = |P_f g|_{\L^{p}_{w_1w_2}} \lesssim |f|_{\mathcal{C}^{\kappa'-\kappa}_{w_2}} |g|_{\mathrm{H}^{\kappa-\kappa'+\eps,p}_{w_1}}.$$
        As $\W^{\alpha,p}$ is continuously embedded in $\mathrm{H}^{\alpha,p}$ for $\alpha\geqslant 0$, \cref{Cor:action_V} implies 
        $$|P_f g|_{\W^{\sigma,p}}\lesssim |f|_{\mathcal{C}^{\kappa'-\kappa}_{-\kappa'}}|g|_{\W^{\sigma+\kappa+\eps,p}}.$$
        \item Let $f\in\W^{\sigma+\kappa+\eps,p}$ and $g\in\mathcal{C}^{\kappa'-\kappa}_{-\kappa'}$. By norm equivalence, it holds
        $$|R(f,g)|_{\W^{\sigma,p}}\lesssim |R(f,g)|_{\mathrm{H}^{\sigma,p}}+|R(f,g)|_{\L^{p}_{\sigma}}.$$
        Let $w_1$ and $w_2$ as above. As $\sigma+\eps>0$, \cref{Prop:ContinuityParaprodutcs} implies
        $$|R(f,g)|_{\mathrm{H}^{\sigma,p}}\lesssim |f|_{\mathrm{H}^{\sigma+\kappa-\kappa'+\eps,p}_{w_2^{-1}}}|g|_{\mathcal{C}^{\kappa'-\kappa}_{w_2}}.$$
        Similarly, 
        $$|R(f,g)|_{\L^{p}_{\sigma}} = |R(f,g)|_{\L^{p}_{w_1w_2}} \lesssim |f|_{\mathrm{H}^{\kappa-\kappa'+\eps,p}_{w_1}}|g|_{\mathcal{C}^{\kappa'-\kappa}_{w_2}}.$$
        As $\W^{\alpha,p}$ is continuously embedded in $\mathrm{H}^{\alpha,p}$ for $\alpha\geqslant 0$, \cref{Cor:action_V} implies 
        $$ |R(f,g)|_{\W^{\sigma,p}}\lesssim |f|_{\W^{\sigma+\kappa+\eps,p}}|g|_{\mathcal{C}^{\kappa'-\kappa}_{-\kappa'}}.$$
    \end{enumerate}

\end{proof}

In order to prove \cref{Lem:ComPuH}, we need the following lemma, known as Bernstein's inequalities, and its corollary.

\begin{lemme}{(Weighted Bernstein's inequalities)}\label[lemme]{Lem:WeightBerstein}
     Let $w$ be an admissible weight, $B\subset\R^2$ be a ball and $N\in\N$. Then, there exists $C_N>0$ such that for any $1\leqslant p\leqslant q\leqslant +\infty$, any $\lambda\geqslant 1$ and any $u\in\L^p_w$,
     $$\supp\mathcal{F}u\subset \lambda B\;\Rightarrow\; \sup_{|\alpha|=N} |\p^\alpha u|_{\L^q_w}\leqslant C_N \lambda^{N+2\left(\frac{1}{p}-\frac{1}{q}\right)}|u|_{\L^p_w}.$$
\end{lemme}

\begin{cor}\label[cor]{Cor:ActionBesovd/dx}
    Let $M\in\N$ and $w$ an admissible weight. Then, there exists $C_M>0$ such that for any $s\in\R$, $1\leqslant p,q\leqslant+\infty$ it holds
    $$\forall f\in\mathrm{B}^{s+M}_{p,q,w},\; \sup_{|\alpha|= M} |\p^\alpha f|_{\mathrm{B}^{s}_{p,q,w}} \leqslant C_M |f|_{\mathrm{B}^{s+M}_{p,q,w}}.$$
\end{cor}

\begin{lemme}\label[lemme]{Lem:ComPuDelta}
    Let $\kappa\in[0,1)$, $\sigma\in[0,1-\kappa)$ and $N\in\N$. Let $w_1$ and $w_2$ be admissible weights. Then, there exists $C=C(\kappa,\sigma,N,w_1,w_2)>0$ such that
    $$\forall f\in \mathrm{H}^{\sigma+\kappa}_{w_1},\; \forall g\in \mathcal{C}^{2-\kappa}_{w_2},\; |[P^N_f,\Delta]g|_{\mathrm{H}^\sigma_{w_1w_2}}\leqslant C |f|_{\mathrm{H}^{\sigma+\kappa}_{w_1}}|g|_{\mathcal{C}^{2-\kappa}_{w_2}}.$$
\end{lemme}

\begin{proof}

    An explicit computations gives
    \begin{equation}
        -[P^N_f,\Delta]g = P^N_{\Delta f}g+2P^N_{\nabla f}\nabla g.\label{Eq:ComPNDelta}
    \end{equation}
    By \cref{Prop:ContinuityParadifferentielGeneral,Cor:ActionBesovd/dx,Eq:ComPNDelta},
    \begin{align*}
        |[P^N_f,\Delta]g|_{\mathrm{H}^\sigma_{w_1 w_2}} &\lesssim |P^N_{\Delta f}g|_{\mathrm{H}^\sigma_{w_1 w_2}}+|P^N_{\nabla f}\nabla g|_{\mathrm{H}^\sigma_{w_1 w_2}}\\
        &\lesssim |\Delta f|_{\mathrm{H}^{\sigma-2+\kappa}_{w_1}}|g|_{\mathcal{C}^{2-\kappa}_{w_2}} + |\nabla f|_{\mathrm{H}^{\sigma-1+\kappa}_{w_1}}|\nabla g|_{\mathcal{C}^{1-\kappa}_{w_2}}\\
        &\lesssim |f|_{\mathrm{H}^{\sigma+\kappa}_{w_1}}|g|_{\mathcal{C}^{2-\kappa}_{w_2}}.
    \end{align*}

\end{proof}

\begin{proof}[Proof of \cref{Lem:ComPuH}.]

    Let $\kappa\in(0,1)$, $s\in[0,\kappa)$, $\sigma\in (0,1-\kappa)$ and $N\in\N$. Let $f\in \W^{\sigma+\kappa,2}$ and $ g\in \mathcal{C}^{2-\kappa}\cap\mathcal{C}^{-\kappa+s}_{2-s}$. First, remark
    \begin{equation}
        [P^N_f,H]g = [P^N_f,\Delta]g - [P^N_f,V]g,\label{Eq:ComPNH}
    \end{equation}
    with $V(x)=1+|x|^2$ and by Bony's decomposition \cref{Eq:BonyDecomposition},
    \begin{align*}
        [P^N_f,V]g &= P_f(\Delta_{>N}Vg)-VP_f(\Delta_{>N}g)\\
        &= f\Delta_{>N}(Vg) - P_{\Delta_{>N}Vg}f-R(f,\Delta_{>N}Vg)\\
        &- Vf\Delta_{>N}g + VP_{\Delta_{>N}g}f + VR(f,\Delta_{>N}g)\\
        &= Vf\Delta_{\leqslant N}g-f\Delta_{\leqslant N}(Vg)-[P_{(\Delta_{>N}\cdot)} f,V]g - [R(f,\Delta_{>N}\cdot),V]g
    \end{align*}
    thus
    \begin{equation}
        -[P^N_f,V]g = [P_{(\Delta_{>N}\cdot)} f,V]g + [R(f,\Delta_{>N}\cdot),V]g +f\Delta_{\leqslant N}(Vg)-Vf\Delta_{\leqslant N}g.\label{Eq:ComPNV}
    \end{equation}
    First, \cref{Prop:ContinuityParaprodutcs,Cor:EstimDeltaN,Cor:action_V,Eq:NormEquivWeightedBesov} imply
    \[|Vf\Delta_{\leqslant N}g|_{\H^\sigma}\lesssim |f|_{\H^\sigma_{s}}|\Delta_{\leqslant N}g|_{\mathcal{C}^\sigma_{2-s}}\lesssim_N |f|_{\W^{\sigma+\kappa,2}}|g|_{\mathcal{C}^{s-\kappa}_{2-s}},\]
    \[|Vf\Delta_{\leqslant N}g|_{\L^2_{\sigma}}\lesssim |f|_{\L^2_{s+\sigma}}|\Delta_{\leqslant N}g|_{\L^\infty_{2-s}}\lesssim_N |f|_{\W^{\sigma+\kappa,2}}|g|_{\mathcal{C}^{s-\kappa}_{2-s}},\]
    \[|f\Delta_{\leqslant N}(Vg)|_{\H^\sigma}\lesssim |f|_{\H^\sigma_{s}}|\Delta_{\leqslant N}(Vg)|_{\mathcal{C}^\sigma_{-s}}\lesssim_N |f|_{\W^{\sigma+\kappa,2}}|g|_{\mathcal{C}^{s-\kappa}_{2-s}}\]
    and
    \[|f\Delta_{\leqslant N}(Vg)|_{\L^2_{\sigma}}\lesssim |f|_{\L^2_{s+\sigma}}|\Delta_{\leqslant N}(Vg)|_{\L^\infty_{-s}}\lesssim_N |f|_{\W^{\sigma+\kappa,2}}|g|_{\mathcal{C}^{s-\kappa}_{2-s}}.\]
    By \cref{Lem:ComPuDelta,Cor:action_V,Eq:ComPNDelta}, it holds
    \begin{equation}
        |[P_f^N,\Delta]g|_{\W^{\sigma,2}}\lesssim |f|_{\W^{\sigma+\kappa,2}}|g|_{\mathcal{C}^{2-\kappa}}.\label{Eq:EstimComPNDelta}
    \end{equation}
    Then, we estimate $[P_{(\Delta_{>N}\cdot)} f,V]g$ by estimating each term of the commutators in \cref{Eq:ComPNV}. Using \cref{Cor:ContinuityParaprodutcsWsp,Lem:EstimDeltaNLqw,Cor:EstimDeltaN,Rem:Eps=0ParaproductWs2,Eq:NormEquivWeightedBesov},
    $$|P_{\Delta_{>N}(Vg)}f|_{\W^{\sigma,2}}\lesssim|f|_{\W^{\sigma+\kappa,2}}|\Delta_{>N}(Vg)|_{\mathcal{C}^{s-\kappa}_{-s}}\lesssim|f|_{\W^{\sigma+\kappa,2}}|g|_{\mathcal{C}^{s-\kappa}_{2-s}}$$
    and 
    $$|R(\Delta_{>N}(Vg),f)|_{\W^{\sigma,2}}\lesssim|f|_{\W^{\sigma+\kappa,2}}|\Delta_{>N}(Vg)|_{\mathcal{C}^{s-\kappa}_{-s}}\lesssim|f|_{\W^{\sigma+\kappa,2}}|g|_{\mathcal{C}^{s-\kappa}_{2-s}}.$$
    Now, using \cref{Prop:ContinuityParadifferentielGeneral,Cor:EstimDeltaN,Cor:action_V,Eq:NormEquivWeightedBesov}, we obtain
    $$|VP_{\Delta_{>N}g}f|_{\mathrm{H}^\sigma}\lesssim|P_{\Delta_{>N}g}f|_{\mathrm{H}^\sigma_{2}}\lesssim|f|_{\mathrm{H}^{\sigma+\kappa-s}_{s}}|\Delta_{>N}g|_{\mathcal{C}^{s-\kappa}_{2-s}}\lesssim|f|_{\W^{\sigma+\kappa,2}}|g|_{\mathcal{C}^{s-\kappa}_{2-s}}$$
    and
    $$|VP_{\Delta_{>N}g}f|_{\L^2_{\sigma}}\lesssim|P_{\Delta_{>N}g}f|_{\L^2_{\sigma+2}}\lesssim|f|_{\mathrm{H}^{\kappa-s}_{s+\sigma}}|\Delta_{>N}g|_{\mathcal{C}^{s-\kappa}_{2-s}}\lesssim|f|_{\W^{\sigma+\kappa,2}}|g|_{\mathcal{C}^{s-\kappa}_{2-s}}.$$
    For $VR(f,\Delta_{>N}g)$, as previously,
    $$|V R(\Delta_{>N}g,f)|_{\mathrm{H}^\sigma}\lesssim|R(\Delta_{>N}g,f)|_{\mathrm{H}^\sigma_{2}}\lesssim|f|_{\W^{\sigma+\kappa,2}}|g|_{\mathcal{C}^{s-\kappa}_{2-s}}$$
    but, in order to apply \cref{Prop:ContinuityParadifferentielGeneral}, we need to add an arbitrary small loss of derivative, namely
    $$|VR(f,\Delta_{>N}g)|_{\L^2_{\sigma}}\lesssim|R(\Delta_{>N}g,f)|_{\mathrm{H}^\eps_{\sigma+2}}\lesssim|f|_{\mathrm{H}^{\kappa+\eps-s}_{s+\sigma}}|\Delta_{>N}g|_{\mathcal{C}^{s-\kappa}_{2-s}}\lesssim|f|_{\W^{\sigma+\kappa+\eps,2}}|g|_{\mathcal{C}^{s-\kappa}_{2-s}}.$$
    It follows,
    \begin{equation}
        |[P_f^N,V]g|_{\W^{\sigma,2}}\lesssim |f|_{\W^{\sigma+\kappa+\eps,2}}|g|_{\mathcal{C}^{s-\kappa}_{2-s}}.\label{Eq:EstimComPNV}
    \end{equation}
    Combining \cref{Eq:EstimComPNDelta,Eq:EstimComPNV} in \cref{Eq:ComPNH}, we finally obtain 
    $$|[P^N_f,H]g|_{\W^{\sigma,2}}\lesssim |f|_{\W^{\sigma+\kappa+\eps,2}}\left(|g|_{\mathcal{C}^{2-\kappa}}+|g|_{\mathcal{C}^{-\kappa+s}_{2-s}}\right).$$
    
\end{proof}

\printbibliography

\end{document}